\documentclass[bj]{imsart}

\RequirePackage[OT1]{fontenc}
\RequirePackage{amsthm,amsmath,natbib}
\RequirePackage[colorlinks,citecolor=blue,urlcolor=blue]{hyperref}

\usepackage{xr}
\usepackage[english]{babel}
\usepackage{amsmath,amssymb,dsfont, amsthm}
\usepackage{version} 
\usepackage{epsfig}
\usepackage{bbm}
\usepackage{color}

\definecolor{orange}{RGB}{255,127,0}
\definecolor{green2}{RGB}{0,153,0}

\graphicspath{{Figures/}}

\startlocaldefs
\renewcommand{\hat}{\widehat}

\newcommand{\PP}{{\mathbb P}}

\newcommand{\QQ}{{\mathbb Q}}

\newcommand{\MM}{{\mathbb M}}

\newcommand{\RR}{\mathbb{R}}

\newcommand{\NN}{\mathbb{N}}

\newcommand{\eps}{\varepsilon}

\newcommand{\mc}[1]{\mathcal {#1}}

\DeclareMathOperator{\argmin}{argmin}

\renewcommand\theequation{\thesection.\arabic{equation}}
\newtheorem{theorem}{Theorem}[section]
\newtheorem{lemma}[theorem]{Lemma}
\newtheorem{prop}[theorem]{Proposition}
\newtheorem{cor}[theorem]{Corollary}

\newtheorem{rem}[theorem]{Remark}

\endlocaldefs

\begin{document}

\begin{frontmatter}
\title{Least squares estimation in the monotone single index model}
\runtitle{LSE in the monotone single index model}

\begin{aug}
\author{\fnms{Fadoua} \snm{Balabdaoui}\thanksref{a,b}\ead[label=e1]{fadoua@ceremade.dauphine.fr}},
\author{\fnms{C\'ecile} \snm{Durot}\thanksref{c}\ead[label=e2]{cecile.durot@gmail.com}} 
\and
\author{\fnms{Hanna} \snm{Jankowski}\thanksref{d}\ead[label=e3]{hkj@yorku.ca}}
\runauthor{Balabdaoui, F., Durot, C. and Jankowski, H.}

\affiliation{Universit\'e Paris Dauphine and ETH, Universit\'e Paris Nanterre, and York University}

\address[a]{Universit\'e Paris-Dauphine,
PSL Research University,
CNRS, CEREMADE,
 75016 Paris, France,
\printead{e1}}

\address[b]{Seminar f\"ur Statistik,
 ETH Z\"urich,
 8092, Z\"urich, Schweiz}

\address[c]{Modal'x, UPL, Univ Paris Nanterre,
F92000 Nanterre France,
\printead{e2}}

\address[d]{Department of Mathematics and Statistics,
 York University,
 4700 Keele Street,
Toronto ON,  M3J 1P3, Canada
\printead{e3}}

\end{aug}

\begin{abstract}. 
We study the monotone single index model where a real response variable $Y$ is linked to a $d$-dimensional covariate $X$ through the relationship $E[Y | X] = \Psi_0(\alpha^T_0 X),$ almost surely. Both the ridge function, $\Psi_0$, and  the index parameter, $\alpha_0$, are unknown and the ridge function is assumed to be monotone. Under some appropriate conditions, we  show that the rate of convergence in the $L_2$-norm for the least squares estimator of the bundled function  $\Psi_0({\alpha}^T_0 \cdot)$ is $n^{1/3}$.  A similar result is established for the isolated ridge function and for the index.   Furthermore, we show that the least squares estimator is nearly parametrically rate-adaptive to piecewise constant ridge functions. Since the least squares estimator of the index is computationally intensive, we also consider alternative estimators of the index $\alpha_0$ from earlier literature.  Moreover, we show that if the rate of convergence of such an alternative estimator is at least $n^{1/3}$, then the corresponding least-squares type estimators (obtained via a ``plug-in" approach) of both the bundled and isolated ridge functions still converge at the rate $n^{1/3}$.
\end{abstract}



\end{frontmatter}

\section{Introduction}\label{sec: intro}
\subsection{The generalized linear model and the single index model}
The generalized linear model is widely used  in econometrics and biometrics as a standard tool in parametric regression analysis, see {\it e.g.} \cite{dobson2008introduction}. It assumes that 
%
the observations are $n$ i.i.d. copies of a pair $(X,Y)$ such that 
\begin{equation}\label{linkPsi}
E(Y|X)=\Psi_{0}(\alpha_{0}^TX)
\end{equation} 
almost surely with an unknown index $\alpha_{0}\in\RR^d\backslash\{0\}
$ and a monotone ridge function $\Psi_0$. In the generalized linear model, $\Psi_0$ is assumed to be known and the conditional density of $Y$ given $X=x$ with respect to a given dominating measure (typically either  Lebesgue measure  or counting measure) is assumed to be of the form
\begin{equation}\label{eq: EF}
y\mapsto h(y,\phi)\exp\left\{\frac{y\ell (\mu(x))-B\circ \ell(\mu(x))}{\phi}\right\}
\end{equation}
where $h$ is the normalizing function, $\mu(x)$ is the mean, $\phi>0$ is a possibly unknown dispersion parameter,  $\ell $ is a given real valued function with first derivative $\ell '>0$, and inverse $\ell^{-1}=B'$.  The generalized linear model includes very popular parametric regression models but nevertheless,  it lacks the flexibility offered by more robust approaches.

The single index model extends the generalized linear model in order to gain more flexibility. It is widely used, for instance, in econometrics, as a compromise between restrictive parametric assumptions and a fully non-parametric setting that can suffer from the ``curse of dimensionality" in high-dimensional problems. 
It assumes that the conditional expectation of $Y$ depends only on the linear predictor $\alpha_0^TX$. Hence, as in the generalized linear model,  we have $E(Y|X)=\Psi_{0}(\alpha_{0}^TX)$ almost surely, however, the ridge function $\Psi_0$ is now unknown. Furthermore,  it is no longer assumed that the conditional distribution of $Y$ given $X$ takes the form \eqref{eq: EF}, making the  model even more flexible.

Standard methods for estimating $\alpha_0$ and $\Psi_0$ rely on smoothness assumptions on $\Psi_0$,  and hence involve a smoothing parameter which has to be carefully chosen, see {\it e.g.} \cite{hardle1993optimal}, \cite{chiou2004quasi},  \cite{hristache2001direct} and references therein. Note also that $\alpha_0$ and $\Psi_0$ are not identifiable if left unrestricted.  To see this, let $\|\alpha_0\|$ denote the Euclidean norm of $\alpha_0$, and note that $\Psi_0(\alpha_0^Tx)=\Phi_0(\beta_0^Tx)$ if $\beta_0=\alpha_0/\|\alpha_0\|$ and $\Phi_0(t)=\Psi_0(\|\alpha_0\| t)$ for all~$t$.  Similarly,  $\Psi_0(\alpha_0^Tx)=\Phi_0(\beta_0^Tx)$ if $\beta_0=-\alpha_0$ and $\Phi_0(t)=\Psi_0(-t)$ for all~$t$.
This issue could be resolved by assuming, e.g.,  that $\|\alpha_0\|=1$ and the first non-null entry of $\alpha_0$ is positive. Under some additional constraints on $\Psi_0$ and the distribution of $X$, the model can be shown to be identifiable.

\subsection{The monotone single index model}

In this paper, we assume that the unknown ridge function  in the single index model is monotone. This is motivated by the fact that  monotonicity appears naturally in various applications, which is one of the reasons behind the popularity of the generalized linear model. Moreover, the monotonicity assumption has a great advantage.  Estimators based only on smoothness conditions on the ridge function typically depend on a tuning parameter that has to be chosen by the practitioner. The monotonicity assumption avoids all this by opening the door to non-parametric estimators which are completely data driven, and do not involve any tuning parameters. 
To be precise, we assume \eqref{linkPsi} where $\alpha_0\in \mathbb R^d\backslash\{0\}$ is such that $\|\alpha_0\|=1$, and $\Psi_0$ is assumed to be  {\it non-decreasing}.  Note that the assumption made on the direction of monotonicity of  the ridge function replaces the assumption that the first non-null entry of $\alpha_0$ is positive in the identifiability conditions. This can be seen by defining  the function $\Phi_0(t)=\Psi_0(-t)$ for $t \in \RR$, which is non-increasing if and only if $\Psi_0$ is non-decreasing. Throughout this paper, we will refer to this model as the monotone single index model, a term that has been used previously in the literature.

The monotone single index model, with the additional assumption that $Y-E(Y|X)$ is independent of $X$, has been considered by \cite{foster2013variable}, where an estimator for $\alpha_0$ was proposed based on combining isotonic regression with a smoothing  method (which involves a tuning parameter), and also by \cite{kakade2011efficient}, where an algorithm for simultaneously  estimating the index and the ridge function is provided under the assumption that the ridge function is Lipschitz (the Lipschitz constant is a parameter of the algorithm).  This fits in the setting of \cite{han1987non} with (using the notation of that paper) $F(x,u)=f(x)+u$ and $D(t)=t$ with $f$ a monotone function. 
\cite{han1987non} proves consistency of a non-parametric estimator of the index, which does not require a tuning parameter.
The monotone single index model is also closely related to the model considered by \cite{chen2014}, who in contrast to the approach followed here, assume that the conditional distribution of $Y$ given $X$ takes the form \eqref{eq: EF}. \citet{chen2014} also consider  additive index models where, 
with $\ell$ as in \eqref{eq: EF}, $\ell(E(Y|X))$ can be written as a sum of ridge functions of  linear predictors, with each ridge function satisfying a certain shape constraint. The  authors show consistency for a slightly modified maximum likelihood estimator obtained by maximizing the likelihood over  the closure of the set of all possible parameters.

Current status regression can also be seen as a special case of the monotone single index model.  In the current status regression setting, the response $Y\geq 0$ is subjected to interval censoring and is not completely observed.  Instead, independent copies of  $(X,C,\Delta)$ are observed, where
$X\in\mathbb R^{d}$ is the predictor, $C\ge 0$ is an observed censoring time independent of $Y$, 
and $\Delta=1_{\{Y\le C\}}.$ 
Although not observed, $Y$ is assumed to satisfy the linear regression model
$Y=\alpha_0^TX+\varepsilon,$
where $\alpha_0\in\mathbb R^{d}$ and $\varepsilon$ is independent of $(C,X)$ with unknown distribution function~$F_0$.  Let $\tilde X$ denote  the random vector in $\mathbb R^{d+1}$  such that $\tilde{X}^T = (C, X^T)$,    $\tilde\alpha_0$ the vector in $\mathbb R^{d+1}$ such that $\tilde{\alpha}^T_0  = (1, -\alpha_0^T)$  and $\tilde Y=\Delta.$  Then, 
$\displaystyle E(\tilde Y|\tilde X)=F_0(\tilde \alpha_0^T\tilde X)$ where $F_0$ is non-decreasing (since it is a distribution function).   Here, the conditional distribution of $\tilde Y$ given $\tilde X$ is Bernoulli, with  $\ell (\mu)=\log(\mu/(1-\mu))$ for $\mu\in(0,1)$ in \eqref{eq: EF}. 
Note that the particular case where the censoring time $C\equiv 0$ has been widely used in econometrics and is usually referred to as the binary choice model. The maximum likelihood estimator (MLE) of $\alpha_0$ was proved to be consistent by \cite{cosslett1983distribution}, and \cite{murphy1999current} prove that the rate of convergence is $O(n^{1/3})$ in the one-dimensional case (that is, when $d=1$). The latter also shows that an appropriately penalized MLE is $\sqrt n$-consistent, but the considered estimator is difficult to implement. \cite{groeneboom2016current} consider several alternative $\sqrt n$-consistent estimators based on a truncated likelihood,  where the truncation parameter has to be chosen by the practitioner.

\subsection{Contents of the paper}

In the monotone single index model, we consider the least squares estimator (LSE) which estimates both the index and the ridge function without the use of a tuning parameter.  We give a characterization of the LSE of $(\Psi_0,\alpha_0)$ under the monotonicity constraint using a profile approach.  Furthermore, letting $g_0(x)=E(Y|X=x)=\Psi_0(\alpha_0^Tx)$, we prove that, under appropriate conditions, the LSE of $g_0$ converges at an $n^{1/3}$-rate in the $L_2$-norm. Then, we consider the LSE of $\alpha_0$ and $\Psi_0$ separately, and also prove their $n^{1/3}$-consistency. The $n^{1/3}$-rate of convergence obtained for the index may be due to our strategy of proof, as we derive this rate from the $n^{1/3}$-rate of the LSE of $g_0$. Thus, sharper rates could potentially be obtained using alternative methods. This is however out of the scope of this work. 

We also study the adaptive properties of our least squares approach.
To our best knowledge, the first adaptive results for the Grenander estimator appear in \citet{pietpyke}, where the asymptotic distribution is obtained exactly.   Later, \citet{sara93} showed 
that the MLE of a decreasing density converges at rate $n^{-1/2} (\log n)^{1/2}$ in Hellinger distance when the true density is uniform. Similar (nearly parametric) results for shape-constrained estimators can be found in 
\citet{bodhi_adapt, royjon_adapt, AEadd1, AEadd2, AEadd3}. Following \citet{bodhi_adapt}, we study the global convergence rate when the truth $\Psi_0$ is piecewise constant.   In Theorem~\ref{thm:adapt_rate2v2}, we show that the rate of convergence to $g_0$ in this case is nearly parametric.

The least squares estimator of the index $\alpha_0$ is computationally intensive, and hence inefficient to compute exactly. Therefore, we also consider alternative estimators of the index taken from earlier literature.   Among them, the so-called  linear estimator, due to \citet{brillinger83}, is especially appealing since it is very easy to implement and converges at the $\sqrt n$-rate to the true index under appropriate conditions.  We then consider ``plug-in" estimators of $g_0:$ we first estimate the index using the first $pn$ data points for some fixed $p\in(0,1)$, then plug the obtained estimator $\tilde\alpha_n$ in the least squares criterion and finally minimize the criterion based on the remaining $(1-p)n$ data points over the space of monotone ridge functions. See Section~\ref{sec: plug} for details.  Combining these two estimators gives an estimator of $g_0,$ and we show that if the rate of convergence of $\widetilde \alpha_n$ is sufficiently fast, then the corresponding estimators of $g_0$ and $\Psi_0$ converge at the $n^{1/3}$-rate.

The paper is organized as follows.  In Section \ref{sec: LSE}, we show existence of the LSE  of $(\Psi_0, \alpha_0)$ and give its characterization.  
Section~\ref{sec: alternative} is devoted to the description of the plug-in approach based on alternative estimators for the index, as well as to the description of such alternative estimators. 
Our main result is given in Theorem \ref{theo: rateL2} in Section~\ref{sec: rate}, where we  establish the $n^{1/3}$-convergence rate of the LSE of $g_0$. Section \ref{sec: rate} also gives the adaptive properties of the LSE. In Section~\ref{sec: separate}, we show under some specified assumptions that  the LSE of $\alpha_0$ and $\Psi_0$  converge separately  at the same rate in the Euclidean norm on $\RR^d$ and the $L_2$-norm on the set of real valued functions respectively, provided that we restrict integration to a bounded subset of the domain of $\Psi_0$. 
 Section~\ref{sec: cvalternative} studies the rate of convergence of the above-mentioned plug-in estimators.   The proof of Theorem \ref{theo: rateL2}  is given  in Section~\ref{sec: proofs}.  Other proofs are deferred to the Supplementary Material.

\section{Existence and characterization of the least squares estimator}\label{sec: LSE}

Assume that we observe an i.i.d. sample $(X_1,Y_1),\dots,(X_n,Y_n)$ from $(X,Y)$ such that
$E(Y|X)=\Psi_{0}(\alpha_{0}^TX)$
almost surely, where both the index $\alpha_0$ and the monotone ridge function $\Psi_0$ are unknown.  To ensure model identifiability (see Section \ref{sec: intro}), $\alpha_0$ is assumed to belong to the $d$-dimensional unit sphere, $\mathcal{S}_{d-1}$ {and the} ridge function $\Psi_0$ is assumed to be non-decreasing on its domain, which contains the range of the linear predictor $\alpha_0^TX$.  
For technical reasons, in what follows we will extend all functions outside their actual support by taking the extension to be constant to the left and right of the endpoints of the original support.

The goal is to find the LSE of $(\Psi_{0},\alpha_{0})$, the minimizer of the  least-squares criterion 
$$h_n(\Psi, \alpha)=\sum_{i=1}^n\left \{ Y_{i} - \Psi(\alpha^T X_i) \right \}^2$$
over $\mathcal M \times \mathcal S_{d-1}$, where $\cal M$ is the class of all non-decreasing functions on $\mathbb R$.  Using a profile least-squares approach, we first minimize $\Psi \mapsto h_n(\Psi, \alpha)$ over $\mathcal M$ for a fixed $\alpha$, and then minimize over $\alpha$.
All proofs for Section \ref{sec: LSE} are given in Section \ref{sec: proofsSection2} of the Supplement.  

\begin{theorem}\label{theo: caractLSE}
For any $\alpha\in\RR^d$, the minimum of $\Psi \mapsto h_n(\Psi, \alpha)$ over $\cal M$ is achieved. 
The minimizer is not unique; it is uniquely defined at the points $\alpha^TX_i$, $i=1,\dots,n$. 
\end{theorem}

Next, we search for $\hat\alpha_n$ that minimizes 
\begin{eqnarray}\label{line:alg2D}
\hat h_n(\alpha):=\min_{\Psi\in{\cal M}}h_n(\Psi, \alpha)
\end{eqnarray}
over $\alpha \in{\cal S}_{d-1}$. The following proposition shows that the minimum is attained on $ {\cal S}^X$, the set of all $\alpha \in{\cal S}_{d-1}$ which  satisfy  $\alpha^TX_i\neq\alpha^TX_j$ for all $i\neq j$ such that $X_i\neq X_j$. This will prove very helpful to provide a characterization of the LSE, see Theorem \ref{theo: caractLSE2} below.

\begin{prop}\label{prop: alphaexists}
{The infimum of $\hat h_n$ over ${\cal S}_{d-1}$ is achieved on $ {\cal S}^X$ and the minimizer is not unique.}
\end{prop}
Combining Theorem \ref{theo: caractLSE} and Proposition \ref{prop: alphaexists}, we prove existence and non-uniqueness  of the LSE. Some notation is needed before giving  a precise characterization of the LSEs.  The characterization uses the fact that (thanks to Proposition \ref{prop: alphaexists}) one can restrict attention to those $\alpha\in{\cal S}^X$ in the minimization process.
Let $x_1,\dots,x_m$  denote the distinct values of $X_1,\dots,X_n$, where $m\in\NN$ is random.  We define
\begin{equation}\label{eq: tilden}
\tilde n_k=\sum_{i=1}^n\mathbb{I}_{X_i=x_k}\mbox{ and }{\tilde y_k}=\frac{1}{\tilde n_k}\sum_{i=1}^nY_i \, \mathbb{I}_{X_i=x_k}
\end{equation}
for all $k=1,\dots,m$.   
Let ${\cal P}^X$ be the set of  all permutations (i.e. orderings) $\pi$ on $\{1,\dots,m\}$ such that there exists an $\alpha\in{\cal S}_{d-1}$  that linearly  induces $\pi$ in the sense that 
\begin{equation}\label{eq: inducepi}
\alpha^Tx_{\pi(1)}<\dots <\alpha^Tx_{\pi(m)}.
\end{equation} 
Note that for each $\alpha\in {\cal S}^X$, the $\alpha^Tx_k$'s are all different from each other and therefore, there exists a unique permutation $\pi$ on $\{1,\dots,m\}$ that is linearly induced by $\alpha$, i.e., that satisfies \eqref{eq: inducepi}.
Then, for each $\pi\in{\cal P}^X$, we denote by $d_1^\pi\leq \ldots\leq d_m^\pi$  the left derivatives of the greatest convex minorant  of the cumulative sum diagram defined by the set of points 
$$\left\{ (0,0), \left(\sum_{j=1}^k \tilde n_{\pi(j)}, \sum_{j=1}^k\tilde n_{\pi(j)} {\tilde y}_{\pi(j)}\right),\,  k =1, \ldots, m \right\}.$$


\begin{theorem}\label{theo: caractLSE2}
The infimum of $(\Psi,\alpha) \mapsto h_n(\Psi, \alpha)$ over ${\cal M}\times{\cal S}_{d-1}$ is achieved. Moreover, if  $(\hat\Psi_n,\hat\alpha_n)$  satisfies the following conditions, then it is a minimizer: 
 \begin{itemize}
\item  $\hat\alpha_n\in{\cal S}^X$ linearly induces  $\hat \pi_n$ that minimizes    $\pi\mapsto \tilde h_n(\pi):=\sum_{k=1}^m \tilde n_{\pi(k)} ({\tilde y}_{\pi(k)}-d_k^\pi)^2$ over ${\cal P}^X$, and
\item  $\hat\Psi_n$ is monotone non-decreasing with $\hat \Psi_{n}(\hat\alpha_n^T x_{\hat\pi_n(k)})=d_{k}^{\hat\pi_n}$.
\end{itemize}

\end{theorem}

To compute a LSE, one can implement the following steps: (1) compute $\tilde n_k$ and $\tilde y_k$
for all $k=1,\dots,m$; (2)
compute
$d_1^\pi, \ldots, d_m^\pi$  for all $\pi$ in the finite set ${\cal P}^X$
using, for example, the pool adjacent violators algorithm  \citep[PAVA]{barlowstatistical}; 
(3) compute $\hat\pi_n$ that minimizes $\tilde h_n(\pi)$ over the finite set ${\cal P}^X$; (4) compute $\hat\alpha_n\in{\cal S}^X$ that linearly induces $\hat\pi_n$; (5)
compute $\hat\Psi_n\in\cal M$ such that $\hat\Psi_n(\hat\alpha_n^Tx_{\hat\pi_n(k)})=d_k^{\hat\pi_n}$ for all $k$ (one can consider for simplicity a piecewise constant function). 

The difficulty with the above line of implementation is that it requires that the set of all linearly inducible permutations ${\cal P}^X$ be computable (steps (2) and (3)).   Also, it requires that given a linearly inducible permutation, one can compute an index in ${\cal S}^X$ that induces the permutation (step (4)).  The cardinality of ${\cal P}^X$ is known to be on the order of  $m^{2(d-1)}$, see \cite{cover1967number}, but we are not aware of an efficient algorithm to implement (2)-(4). 

Therefore, instead of using inducible permutationss, one could use an alternative optimization algorithm; for example, stochastic search was used in \citet[Table 4, page 740]{chen2014}.  When adapted to our setting, the algorithm simplifies as follows: (1)  choose the total number $N$ of stochastic searches to perform and set $k=1$; (2) draw a standard Gaussian vector $Z_k$ in $\RR^d$ and compute $\alpha_k=Z_k/\Vert Z_k\Vert$; (3) compute the ordered distinct values $t_1<\dots<t_L$ of $\alpha_k^TX_i$, $i\in\{1,\dots,n\}$ and also
$$n_l=\sum_{i=1}^n \mathbb{I}_{\alpha_k^TX_i=t_l} \mbox{ and } y_l=\frac{1}{n_l}\sum_{i=1}^nY_i \mathbb{I}_{\alpha_k^TX_i=t_l}$$ for all $l=1,\dots,L$; (4)  compute $d_1\leq \ldots\leq d_L$,  the left derivatives of the greatest convex minorant  of the cumulative sum diagram defined by the set of points 
$$\left\{ (0,0), \left(\sum_{j=1}^l n_{j}, \sum_{j=1}^l n_jy_{j}\right),\,  l =1, \ldots, L \right\}$$
using the PAVA; (5) compute $A_k:=\sum_{l=1}^L n_{l} ({y}_{l}-d_l)^2$, set $k:=k+1$, go to (2) if $k\leq N$ and to (6) otherwise; (6) compute $\hat k$ that minimizes $A_k$ over $k\in\{1,\dots,N\}$. An approximated value of the LSE $(\hat \alpha_n,\hat\Psi_n)$ is then given by $(\alpha_{\hat k},\Psi_{\hat k})$, where using the same notation as in (3) and (4) where $k=\hat k$,  $\Psi_{\hat k}$ is piecewise constant function such that $\Psi_{\hat k}(t_l)=d_l$ for all $l=1,\dots,L$. Note that in the algorithm, the variables $Z_1,\dots, Z_N$ are drawn independently from each other.

 For completeness, in the Supplementary Material, we also give an algorithm to compute the LSE exactly for the special case when $d=2,$ see Section~\ref{sec:exact2D}.

\section{Alternative estimators}\label{sec: alternative}

Alternative estimators can be obtained by combining the above least squares approach with an alternative estimator of the index $\alpha_0$, as detailed in Section~\ref{sec: plug} below.  As can be seen from Section~\ref{sec: plug},  the main difficulty in computing the LSE in the monotone single index model lies in computing an estimator of the  unknown index $\alpha_0$. Hence, we consider below various estimators of $\alpha_0$ from earlier literature on single index models with a non-necessarily monotone ridge function.  For notational convenience, all the considered estimators are denoted by $\tilde\alpha_n$. Among the considered estimators,  the linear estimator of Section~\ref{sec:linear_est} is of particular interest since it
is very easy to compute and converges at the $\sqrt n$-rate in the monotone single index model,  see Theorem \ref{MainTheo} below.

\subsection{Plug-in estimators} \label{sec: plug}
First, randomly split the sample  into two independent sub-samples of respective sizes $n_1$ and $n_2$, where $n_1$ is the integer part of $pn$ for some fixed $p\in(0,1)$ and $n_2=(1-p)n$.  Let $\tilde\alpha_n$ denote some appropriate estimator  of the true index $\alpha_0$ using the $n_1$ data points in the first sub-sample.   Next, we consider the ``plug-in" estimator $\widetilde\Psi_n:=\widehat\Psi^{\tilde\alpha_n}$
of $\Psi_0$, where for all $\alpha$, $\widehat\Psi^{\alpha}$ is the  minimizer of 
\begin{equation}\label{eq: htilde}
\Psi\mapsto \sum_{i \in I_2} \left\{Y_i - \Psi({\alpha}^T  X_i)\right\}^2
\end{equation} over $\Psi \in \mathcal{M}$, where {$\{(X_i,Y_i),i\in I_2\}$} are the observations from the second sub-sample.  Once $\tilde\alpha_n$ is given, the estimator $\widetilde\Psi_n$ is easy to compute using again the PAVA. 
Indeed,   it follows from  \citet[Theorem 1]{barlowstatistical} that 
any $\widetilde\Psi_n\in\cal M$ such that $\widetilde\Psi_n(Z_k)=d_k$ is a minimizer. Here, 
$Z_1<\dots<Z_m$ denote the ordered distinct values of $\tilde \alpha_n^TX_i,i \in I_2$, and 
$d_1\leq \ldots\leq d_m$ are the left derivatives of the greatest convex minorant  of the cumulative sum diagram defined by the set of points 
$$\left\{ (0,0), \left(\sum_{i \in I_2}\mathbb{I}_{\tilde \alpha_n^T X_i\le Z_k}, \sum_{i \in I_2}Y_i\mathbb{I}_{\tilde \alpha_n^T X_i\le Z_k}\right), \,  k =1, \ldots, m \right\}.$$
Below, we consider several estimators $\tilde\alpha_n$ that could be used in this plug-in procedure.

\subsection{The linear estimator}\label{sec:linear_est}

The linear estimator 
goes back to \cite{brillinger83}, who also considered a single index model \eqref{linkPsi} with an unknown, not necessarily monotone ridge function $\Psi_0.$ 
This estimator is exactly what one would use if the regression model were known to be linear. To be precise, based on observations $(X_1,Y_1)$, \dots, $(X_n,Y_n)$ where the $Y_i$s take real values whereas the $X_i$s take values in $\RR^d$, the linear estimator of $\alpha_0$ is defined as follows:
$\tilde\alpha_n=\hat\alpha_n/\|\hat\alpha_n\|$ where here,
\begin{eqnarray}\label{alphan}
\widehat{\alpha}_n = \argmin_{\alpha \in \RR^d }    \sum_{i=1}^n (Y_i  -  \alpha^T ( {X}_i-\bar X_n))^2
\end{eqnarray} 
with  $\bar{X}_n  = n^{-1} \sum_{i=1}^n X_i$.  The linear estimator can therefore be easily computed using standard tools from linear regression. Moreover, it typically converges to the true index $\alpha_0$ at the square-root rate and is asymptotically Gaussian, even if the linearity assumption is not valid.  
Typical assumptions required for these results to hold are that the variables $\Psi_0(\alpha_0^TX)$ and $\alpha_0^TX$ are correlated, and that the conditional expectation of $X$ given $\alpha_0^TX$ is a linear function of $\alpha_0^TX$. The latter condition is met under elliptic symmetry of $X$ (which holds in particular if $X$ is Gaussian, see \cite{Review}), a condition that has been considered for instance by \cite{li1989regression} and
\cite{goldstein2016structured}. It turns out that the condition $Cov(\Psi_0(\alpha_0^TX),\alpha_0^TX)\neq0$ is met in our setting where $\Psi_0$ is monotone and not constant, whence the linear estimator is $\sqrt n$-consistent and asymptotically Gaussian.  The precise statement is given in the following theorem, which is a  close variant to earlier results in the literature on linear estimators. Here, the distribution of $X$ is assumed to be continuous since $\alpha_0$ is not identifiable under a discrete distribution of $X$. The assumption on boundedness of $\Psi_0$ ensures existence of the above covariance. For completeness, the proof
is provided in {Section \ref{sec:ProofMainTheo}} of the Supplementary Material.


\begin{theorem}\label{MainTheo}
Let $(X_1,Y_1), \dots, (X_n,Y_n)$  be an i.i.d. sample from $(X,Y)$ such that $E(Y|X)=\Psi_0(\alpha_0^TX)$ almost surely where $\alpha_0\in{\cal S}_{d-1}$ and $\Psi_0$ is bounded and non-decreasing  such that  there exists a nonempty interval $[a,b]$ in the domain of $\alpha_0^TX$ on which $\Psi_0$  is strictly increasing. Suppose furthermore that $X $ has a continuous elliptically symmetric distribution with finite mean $\mu\in \RR^d$ with a positive definite $d \times d$ covariance matrix $\Sigma$, and  $E(\Vert X\Vert^2Y^2)<\infty$. Then
$\sqrt{n}  \left (\tilde{\alpha}_n - \alpha_0\right) $ converges weakly to a centered $d$-dimensional Gaussian distribution.
\end{theorem}  
Note that by definition of $\widehat{\alpha}_n$, we necessarily have that $
\widehat{\alpha}_n 
                              =  \widehat{\Sigma}_n^{-1}n^{-1}  \sum_{i=1}^n Y_i ( X_i-\bar{X}_n )$, where $\widehat{\Sigma}_n =  n^{-1}  \sum_{i=1}^n  (X_i  -  \bar{X}_n)   (X_i  -  \bar{X}_n)^T$. If the $X_i$'s were known to be centered with identity covariance matrix, one would merely consider the estimator $ n^{-1} \sum_{i=1}^n Y_i  X_i$, which is precisely the estimator considered in Section 1.2 of \cite{plan2016high}. Under the assumptions that (in our notation) $X$ is standard Gaussian and $Y$ is independent of $X$ conditionally on $\alpha^T_0 X$, and $E(\Vert X\Vert^2Y^2)<\infty$, \cite[Proposition 1.1]{plan2016high} shows that their estimator is equal to $\lambda \alpha_0+O_p(\sqrt{d/n})$ in the case $\Vert\alpha_0\Vert=1$, with $\lambda=E(Y\alpha^T_0X)$. As explained in Section 7 of that paper, this result can be generalized to non-standard Gaussian covariates. In fact, if $X$ has a Gaussian distribution with finite mean $\mu\in \RR^d$ and covariance matrix $\Sigma$ such that $\Vert\mu\Vert\leq K$ and all eigenvalues of $\Sigma$ belong to $[K_-,K_+]$ for some positive $K,K_-,K_+$ that do not depend on $d$, then this result can be extended to prove that our $\hat\alpha_n$ is equal to $\lambda^* \alpha_0+O_p(\sqrt{d/n})$ uniformly in $d,n$, with $\lambda^*$ defined as in Section  \ref{sec:ProofMainTheo} of the Supplementary Material. 
This implies that the convergence rate of the linear estimator $\tilde\alpha_n$ depends on the dimension like $\sqrt{d/n}.$ As mentioned by a referee, this rate is optimal as it is the rate one would obtain if the ridge function was equal to the identity.

\subsection{Additional estimators}\label{sec: additional}

In the following we discuss other possible ways of index estimation in the present model.   A well-known estimator in the monotone single index model is the so-called maximum rank correlation (MRC) estimator; see \cite{han1987non}. This estimator is defined as the location of the maximum of $S_n(\alpha)$ over $\alpha\in{\cal S}_{d-1}$ where 
$$S_n(\alpha)=\left(\begin{tabular}{c}$n$\\2\end{tabular}\right)^{-1}\sum_{1\leq i\neq j\leq n}\left[\mathbb{I}_{\{Y_i>Y_j\}}\mathbb{I}_{\{\alpha^TX_i>\alpha^TX_j\}}+\mathbb{I}_{\{Y_i<Y_j\}}
\mathbb{I}_{\{\alpha^TX_i<\alpha^TX_j\}}\right].
$$
Strong consistency of a variant of the MRC estimator is proved in \cite{han1987non} under the assumption that (a) the noise $Y-E(Y|X)$ is independent of $X$, (b) for a given $h\in\{1,\dots,d\}$ the component $h$ of $\alpha_0$ is in absolute value greater than a given $\eta>0$, and (c) the distribution of $X$ behaves \lq\lq nicely\rq\rq. Hence, the variant of the MRC estimator is defined as the location of the maximum of $S_n(\alpha)$ over the set of all $\alpha$'s in ${\cal S}_{d-1}$ whose component $h$ is in absolute value greater than $\eta$. This implies in particular that one would need to know both $h$ and $\eta$, which is quite unrealistic in our opinion.

Building upon the Isotron algorithm of \cite{kalaisastry},  \cite{kakade2011efficient} propose an iterative algorithm in the monotone single index model, called the Slisotron, which finds estimators of the index and monotone ridge function under the additional assumption that the latter is Lipschitz. More precisely, in the updates of the isotonic estimator, the Slisotron algorithm looks for the least squares monotone estimator which is also Lipschitz. 
Slisotron produces estimators for both the index and ridge function, and in view of the current discussion we are interested here in the former.  Theorem~2 of  \cite{kakade2011efficient}  shows that the both true and empirical mean squared errors are of order $n^{-1/3} \log n $ with large probability for some appropriate iteration of Slisotron. This indicates that their  estimator of the index  converges, ignoring the logarithmic factor, at the $n^{1/6}$ rate, which is significantly worse than the cubic rate achieved by our least squares estimator of the index, see Corollary \ref{cor:rate_alpha} below. 

Assuming again that the ridge function is Lipschitz, and assuming also that the response variable takes values in $[0,1]$, \cite{ganti2015learning} provide an estimation method that applies even for the case of high-dimensional  covariates. Their ``SILO" method can be viewed as an extension of the linear estimator (described above in Section \ref{sec:linear_est}) to the high-dimensional case: similar to the linear estimator, the SILO estimator of the index does not take into account the ridge function. Other estimation methods can be found in the compressed sensing literature (some of them are designed for the case of binary response variables), see e.g.   \cite{plan2013robust,plan2016high}.

There are several other alternatives which return an estimate of the index in the single index model with a non necessarily monotone ridge function, and these could also be used here.   For example, one could use kernel-based methods, discussed for example in \citet{hardle1993optimal, chiou2004quasi, hristache2001direct}.   Although these methods yield an estimator which is $\sqrt{n}$-consistent, they do rely on smoothing parameters (the bandwidth) for their estimator.   

\section{Convergence of the LSE for the regression function}\label{sec: rate}
We consider the same setting as in Section \ref{sec: LSE}, however,  we now also assume that $X$ has a continuous distribution. This means that we observe an i.i.d. sample $(X_i, Y_i), i =1, \ldots, n$ from a pair $(X,Y)$ where, with probability one, all the $X_i$'s are different from each other (hence, in the notation of Section \ref{sec: LSE}, $n=m$ and $\tilde n_i=1$ for all $i=1, \dots,n$). It is assumed that $E(Y|X)=\Psi_{0}(\alpha_{0}^TX)$
almost surely, where both the index $\alpha_0$ and the monotone ridge function $\Psi_0$ are unknown, so the regression function is defined by
\begin{equation}\label{eq: g0}
g_0(x)= \Psi_0(\alpha_0^Tx)
\end{equation}
for (almost-) all $x$ in the support of $X$ and its least-squares estimator (LSE)  is given by
\begin{equation}\label{eq: LSEreg}
\hat g_n(x)=\hat \Psi_n(\hat\alpha_n^Tx)
\end{equation}
for (almost-) all $x$ in the support of $X$, where $(\hat\Psi_n,\hat\alpha_n)$ is a LSE  of $(\Psi_0,\alpha_0)$ as studied in Section \ref{sec: LSE}.
For convenience,  we consider below a solution $\hat\Psi_n$ that is left continuous and piecewise constant, with jumps only possible at the points $\hat\alpha_n^TX_i,$ 
for $i=1,\dots,n$. Hence, $\hat\Psi_n$ is uniquely defined whereas $\hat\alpha_n$ is not unique  ($\hat\alpha_n$ denotes an arbitrary minimizer of $\hat h_n$, in the notation of Section \ref{sec: LSE}). 
In this section,  we are interested in the consistency and rate of convergence of $\hat g_n$ in the $L_2$-sense.   

We begin with some notation and assumptions. Let  $\mathcal{X}$ be the support of the random vector of covariates $X$. Let $\PP$ be the joint distribution of $(X, Y)$, $\mathbb{P}_x$ the conditional distribution of $Y$ given $X =x$, and $\mathbb{Q}$ the marginal distribution of $X$. 
Theorem \ref{theo: rateL2} below will be established under the following assumptions.

\begin{list}{}
        {
        \setlength{\topsep}{6pt}
        \setlength{\parskip}{0pt}
        \setlength{\partopsep}{0pt}
        \setlength{\parsep}{0pt}
        \setlength{\itemsep}{3pt}
        \setlength{\leftmargin}{25pt}}

\item[(A1).] $\mathcal{X}$  is a bounded convex set  of $\mathbb{R}^d$,

\item[(A2).] there exists a constant $K_0 > 0$ such that $\vert g_0(x)  \vert \le K_0$ for all $x \in \mathcal{X}$,
\item[(A3).] there exist constants $a_0 > 0$ and $M > 0$ such that for all integers $s\geq 2$  and $x\in\mc X$
\begin{equation}\label{eq: hypexpo}
\int|y|^sd\mathbb P_x(y)\leq a_0s! M^{s-2}, 
\end{equation}

\item[(A4).]  there exist $\overline q>0$ and $\underline q>0$ such that  with respect to the Lebesgue measure,  for all $\alpha\in {\cal S}_{d-1}$,  the variable $\alpha^TX$ has a density that is bounded from above by $\overline q$ and bounded from below by $\underline q$.  

\end{list}

Assumption (A1) ensures that  for all $\alpha\in{\cal S}_{d-1}$, the set $\{\alpha^Tx,\ x\in\cal X\}$, which is the support of the linear predictor $\alpha^TX$ corresponding to $\alpha$,  is convex ({\it i.e.} is an interval). Hence, we consider functions of the form $\Psi(\alpha^TX)$ where $\alpha\in\mathcal S_{d-1}$ and $\Psi$ is a non-decreasing function on its interval of support $\{\alpha^Tx,\ x\in\cal X\}$.

Assumption (A3) ensures that conditionally on $X=x$, the response variable $Y$ is uniformly integrable in $x$. This assumption is clearly satisfied if $Y$ is a bounded random variable. It is also satisfied if the conditional distribution of $Y$ given $X$ belongs to an exponential family, see Proposition~\ref{prop: EF1} in  Appendix \ref{sec: EF} for more details.

Assumption (A4) makes all distributions of $\alpha^TX$, with $\alpha\in{\cal S}_{d-1}$, equivalent to the Lebesgue measure. 
For instance, if  $\QQ$ is the  distribution resulting from  truncating a Gaussian distribution so that it is supported on $\{x,\ \|x\|\leq R\}$, and if  $\lambda_->0$ and $\lambda_+$ are the smallest and largest  eigenvalues of the covariance matrix of the original Gaussian distribution, then we can take $\overline q=(2\pi\lambda_-)^{-1/2}$ and $\underline q=(2\pi\lambda_+)^{-1/2}\exp(-R^2/\lambda_-)$.



The following theorem proves the $n^{1/3}$-rate of convergence of the bundled estimator $\hat g_n$ under the above assumptions, i.e. under the assumption of a continuous design distribution $\QQ$. The case of a discrete distribution with finite support will be considered in a separate paper.  In this case, a $\sqrt n$-rate of convergence can be proved.   We conjecture that in the case when some of the components of $X$ are continuous, and the other ones are discrete, the rate of convergence is still $n^{1/3}$. Another case where the $\sqrt n$-rate of convergence emerges (up to a $\log n$-factor) is when the true ridge function is constant. This case also will be studied elsewhere.


\begin{theorem}\label{theo: rateL2}   With $g_0$ and $\hat g_n$ defined by \eqref{eq: g0} and \eqref{eq: LSEreg} respectively, where $\hat\Psi_n$ is  the same piecewise constant function described above, and under assumptions (A1)-(A4)  we have
\begin{equation}\label{eq: cvg}
\left(\int_{\mathcal{X}} \left(\widehat{g}_n(x)-g_0(x)\right)^2 d\QQ(x)\right)^{1/2}  = O_p(n^{-1/3}).
\end{equation}
\end{theorem}

\begin{rem}
\begin{itemize}\label{rem :subG} 
\item
If instead of assuming (A4) we only assume that there exists a $\overline q>0$ such that with respect to the Lebesgue measure and for all $\alpha\in {\cal S}_{d-1}$,  the variable $\alpha^TX$ has a density bounded above by $\overline q$, then we obtain a rate of convergence $n^{-1/3}(\log n)^{5/3}$ instead of $n^{-1/3}$, see Section 
\ref{sec: mainsteps} below for details.

\medskip

\item
The convergence rate obtained above depends on the dimension $d$. A closer look at the proof reveals that this dependence takes the form of $
O_p(d(1+\sqrt{\overline q R})n^{-1/3}(\log n)^{5/3})$, see Theorem \ref{theo: ratelog} below. Note that the constant $R$ may hide a dependence on $d$ since in the case where $\cal X$ is the $\ell_\infty$-unit ball in $\mathbb{R}^d$ we have $R=\sqrt d$.

\medskip

\item  Suppose that we relax assumptions (A1) and (A4).   That is, instead of assuming (A4), we only assume here that there exists $\overline q>0$ such that  with respect to the Lebesgue measure,  for all $\alpha\in {\cal S}_{d-1}$,  the variable $\alpha^TX$ has a density that is bounded from above by $\overline q$. Moreover, instead of assuming that $X$ has a bounded support, we assume that $X$ has a sub-Gaussian distribution. This means that there exists $\sigma^2>0$ such that for all vectors $u\in {\cal S}_{d-1}$, and all $t\in\RR$, with $T=u^TX$ we have
$$P(T-E(T)>t)\leq\exp(-t^2/(2\sigma^2))\mbox{ and }P(T-E(T)<-t)\leq\exp(-t^2/(2\sigma^2)).$$
Then, for all $\eps>0$ there exists $A>0$ such that with probablity larger than $1-\eps$
 we have
\begin{eqnarray}\label{eq: subG}
\int_{\mathcal{X}} \left(\widehat{g}_n(x)-g_0(x)\right)^2 d\QQ(x)
&\leq & A \sqrt{\overline q}d^{5/4}(\log (n\vee d))^{1/4}n^{-1/3}(\log n)^{5/3},
\end{eqnarray}
see Section \ref{sec: proofsubG} in the supplement for details. If $d$ does not depend on $n$ then this yields a rate of convergence $n^{-1/3}(\log n)^{23/12}$.
\end{itemize}

\end{rem}

\section{Convergence of the separated LSE estimators}\label{sec: separate}

We now derive from Theorem \ref{theo: rateL2} convergence of $\widehat\alpha_{n}$  to $\alpha_{0}$ and $\hat \Psi_{n}$ to $\Psi_{0}$. Moreover, we are interested in the rate of convergence of the two estimators. Convergence can happen only under uniqueness of the limit so first we prove identifiability of $\Psi_0$ and $\alpha_0$ under appropriate conditions.

\subsection{Identifiability of the separated parameters}\label{sec: identifiability}

Let  $(X,Y)$ be a pair of random variables, where $X$ takes values in $\RR^d$ and $Y$ is an integrable real valued random variable such that \eqref{linkPsi} holds for some $\alpha_{0}\in\mathcal S_{d-1}$ and $\Psi_{0}\in{\cal M}.$ 
Identifiability of the  parameter $(\alpha_{0},\Psi_{0})$  means here that if we can find $\beta$ in $\mathcal S_{d-1}$, and  $h$ in ${\cal M}$ such that
$
\Psi_0(\alpha_0^T X) = h(\beta^T X) \ a.s.
$
then $\beta=\alpha_0$ and $h=\Psi_0$ on ${\cal C}_{\alpha_0}={\cal C}_{\beta}$, where for all $\alpha\in{\cal S}_{d-1}$ we set ${\cal C}_\alpha = \{ \alpha^T x, x \in \mathcal{X} \}$ with $\cal X$ being the support of $X$. Although identifiability can be derived from
 \cite{lin2007identifiability} when assuming that $\Psi_{0}$ is non-constant and continuous, for completeness we state below identifiability under a slightly less restrictive assumption, namely left-- (or right--) continuity instead of continuity. A proof can be found in Section \ref{proof: identifiability} in the Supplement.
Since $\cal X$ is convex,  it follows that ${\cal C}_\alpha$ is an interval for any $\alpha$. Moreover, recall that monotone functions on an interval can be extended to monotone functions on~$\RR$.









 \begin{prop}\label{prop: identifiability}
Assume that $\mathcal X$ is convex with at least one interior point.  Assume also that $X$ has a density with respect to Lebesgue measure which is strictly positive on $\mathcal X,$  and 
that~\eqref{linkPsi} holds for some  $\alpha_{0}\in{\cal S}_{d-1}$ and $\Psi_{0}\in{\cal M}$ that is not constant on ${\cal C}_{\alpha_{0}}$, and either left- or right-continuous on ${\cal C}_{\alpha_{0}}$ with no discontinuity point at the boundaries of ${\cal C}_{\alpha_{0}}$. Then, $(\Psi_{0},\alpha_{0})$ is uniquely defined. 
\end{prop}

\subsection{Convergence of the separated estimators}\label{sec: cvSep}

We begin by establishing consistency of $(\widehat{\alpha}_n,\widehat{\Psi}_n)$ where  $\hat\Psi_n$ denotes the left-continuous LSE of $\Psi_0$ extended to $\RR$ and $\hat\alpha_n$ is a minimizer of $\hat h_n$  defined in \eqref{line:alg2D}, see Section \ref{sec: proofCons} in the Supplementary Material for a proof. 

\begin{theorem} \label{theo: consistency}
Assume that assumptions (A1)-(A3) are satisfied and that there exists a $\overline q>0$ such that for all $\alpha\in {\cal S}_{d-1}$, with respect to the Lebesgue measure, the variable $\alpha^TX$ has a density that is bounded above by $\overline q$. 
Assume, moreover,  that  $\Psi_{0}$ is  non-constant and left-continuous with no discontinuity points at the boundaries of  ${\cal C}_{\alpha_{0}}$, and that  $\mathcal X$ has at least one interior point.  Assume also that $X$ has a density with respect to Lebesgue measure which is strictly positive on $\mathcal X$. 
\begin{enumerate}
\item
We then have $\hat\alpha_{n}=\alpha_{0}+o_{p}(1),$ and for all fixed continuity points $t$ of $\Psi_0$ in the interior of ${\cal C}_{\alpha_{0}}$, $\Psi_{n}(t)$ converges in probability to $\Psi_{0}(t)$ as $n\to\infty$. 
\item If, moreover, $\Psi_0$ is continuous, then
\begin{eqnarray}\label{eq: fnhatConsist}
\sup_{t\in I}\vert\hat \Psi_{n}(t)-\Psi_{0}(t)\vert=o_{p}(1)
\end{eqnarray}
for all compact intervals $I\subset\RR$ such that $  K_-<\Psi_{0}(t)<K_+$ for all $t\in I.$ Here,  $K_+$  and $K_-$ denote the largest and smallest values of $\Psi_0$ on ${\cal C}_{\alpha_{0}}$.
\end{enumerate}
\end{theorem}


\begin{table}[!h]
	\centering
	\caption{Values of the empirical covariances matrices for the case $d=2$  of $n^{1/2} (\widehat{\alpha}_n - \alpha_0)$ and   $n^{1/3} (\widehat{\alpha}_n - \alpha_0)$  with entries  $\widehat{\sigma}^2_{11}, \widehat{\sigma}^2_{22}, \widehat{c}_{12} = \widehat{c}_{21}$.   The sample size $n \in \{10^2, 10^3, 10^4, 10^5\}$ and $\alpha_0 = (\cos(\theta_0), \sin(\theta_0))^T$ with $\theta_0 \in \{\pi/4, \pi/3, \pi/2\}$. The obtained estimates were computed based on $100$ replications  and the model  $Y \sim \mathcal{N}((\alpha^T_0 X)^3, 1)$ and $X \sim \mathcal{U}[0,1] \times \mathcal{U}[0,1] $.   }
	\label{TableGauss}
	\vspace{0.5cm}
	\scalebox{0.99}{
		\begin{tabular}{|clr|ccc |}
			\hline
			$\theta_0$ &  Rate & $n$ & $\widehat \sigma^2_{11}$ & $\widehat\sigma^2_{22}$ & $\widehat{c}_{12}$ \\
			\hline
                   &  & & & &   \\
			& $n^{1/2}$ & 100 & 1.1806089 & 1.1672180 & -1.1626274 \\
			& &1000 & 2.2595533 & 2.2470109  & -2.2470183 \\
			& &10000& 4.4150436 & 4.3392369 & -4.3737141   \\
			& &100000&  6.663009  & 6.604052 &  -6.632743\\
			$\pi/4$ &  & & & &  \\
			 &  & & & & \\
                    &  $n^{1/3} $ &100  & 0.2543545  &0.2514695 &  -0.2504805  \\
			&  &1000 & 0.2259553 & 0.2247011 &  -0.2247018\\
			& &10000&0.2049282 & 0.2014095 &  -0.2030098\\
			&  &100000&  0.1435502  & 0.1422800 &  -0.1428981\\
					\hline
\hline
                    &  & & & &  \\
			&  $n^{1/2}$ & 100 & 3.142922 &1.109797 & -1.836072  \\
			&  &1000 & 5.404810 & 1.700261 & -3.011534  \\
			&  &10000& 7.078401 &2.418188  & -4.132608\\
			&  &100000&   7.564632 &  2.512762& -4.359389  \\
			$\pi/3$ &  & & & &  \\
                   &  & & & & \\
			&  $n^{1/3} $ &100  & 0.6771219 &  0.2390985& -0.3955698\\
			&  &1000 &  0.5404810 &  0.1700261  & -0.3011534\\
			&  &10000& 0.3285503 & 0.1122424  &  -0.1918187\\
			&  &100000&  0.16297505  &  0.05413582& -0.09392019 \\
\hline \hline
                   &  & & & &  \\
			&  $n^{1/2}$ & 100 &  6.73971007 &  0.13244680 &  0.06656768 \\
			&  &1000 &  11.633011584  &  0.051489821  &  -0.047274111  \\
			&  &10000&   12.110197275  &   0.0139769985  &  -0.1955363777   \\
			&  &100000&   10.90404143 & 0.000736231  & -0.0112769835 \\
			$\pi/2$ &  & & & &  \\
                   &  & & & & \\
			&  $n^{1/3} $ &100  & 1.45202652 & 0.02853480 & 0.01434157\\
			&  &1000 &1.163301158  &  0.005148982  & -0.004727411\\
			&  &10000&  0.562105564 & 0.0006487548 & -0.0090759947 \\
			&  &100000&  0.2349204513  & 1.586162e-05 &  -2.429552e-04\\
\hline
		\end{tabular}}
\end{table}

Next, we establish rates of convergence for $\widehat{\alpha}_n$ and $\widehat{\Psi}_n$. 
To show that both  $\widehat{\alpha}_n$ and $\widehat{\Psi}_n$ inherit the  $n^{1/3}$ rate of convergence from the joint convergence established for the full estimator $\hat g_n(\, .\, )=\widehat{\Psi}_n(\widehat \alpha_n^T \cdot)$, some additional assumptions are needed.

\begin{list}{}
        {
        \setlength{\topsep}{6pt}
        \setlength{\parskip}{0pt}
        \setlength{\partopsep}{0pt}
        \setlength{\parsep}{0pt}
        \setlength{\itemsep}{3pt}
        \setlength{\leftmargin}{25pt}}
\item[(A5).] 
 There exists an interior point $z_0\in{\cal C}_{\alpha_0}$ such that $\Psi_0$ is continuously differentiable in the neighborhood of $z_0$, with $\Psi_0'(z_0)>0$.  
 
\item[(A6).]   The density of $X, q,$ is continuous on $\mathcal{X}$. 

\end{list}

Let  $\underline c =\inf \mathcal{C}_{\alpha_0}$ and $\overline c  = \sup \mathcal{C}_{\alpha_0}$. Our main result here is the following. It is proved in Section \ref{sec: proofCons} in the Supplement.

\begin{cor}\label{cor:rate_alpha}
Assume that  $\Psi_{0}$ is  non-constant and left-continuous with no discontinuity points at the boundaries of  ${\cal C}_{\alpha_{0}}$, that $\mc X$ has at least one interior point, and that (A1)-(A6) hold. Then,
$\|\hat\alpha_n-\alpha_0\|= O_p(n^{-1/3}).$ If moreover, $\Psi_0$ has a  derivative bounded from above on  ${\cal C}_{\alpha_0}$, then 
\begin{equation}\label{eq: cvpsi}
\left(\int_{\underline c+v_n}^{\overline c-v_n}(\hat \Psi_n(t)-\Psi_0(t))^2dt\right)^{1/2}=O_p(n^{-1/3})
\end{equation}
for all sequences $v_n$ such that $n^{1/3} v_n\to \infty$ and $\underline c+v_n\le \overline c-v_n$.
\end{cor}

\begin{rem}\label{rem:AEconcern2}The above result holds under Assumption (A1) on the support of the covariate $X.$  The result can be made stronger under additional regularity conditions on the support $\mc X$.  For example, when $\mathcal X$ is a ball in $\RR^d$ centered at the origin and of radius~$r$ then the above result holds with $v_n=0.$  Indeed, in this setting the support of the linear predictor $\alpha^TX,$ for any $\alpha,$  is  $[-r,r]$.  Therefore, in the proof, ${\cal C}_{\hat\alpha_n}=[-r,r]$ and hence $v_n=0, \bar c=r$ and $\underline c=-r$ in inequality \eqref{line:result2} of the Supplement.  Notably, \citet{kakade2011efficient} consider this choice of $\mc X$ with $r=1.$   
\end{rem}

\medskip

The $n^{1/3}$-rate obtained  in Corollary \ref{cor:rate_alpha}  for convergence of the LSE $\hat{\alpha}_n$  towards the truth raises the question whether this convergence actually occurs at a faster rate, for example $n^{1/2}$. In order to investigate this question, we have performed simulations for $d=2$ and two different monotone single index models:  the first one is a Gaussian model where  $Y \sim \mathcal{N}((\alpha_0^T X)^3, 1)$, whereas the second one  is a logistic regression model where $Y \sim \text{Bin}(10, \exp(\alpha_0^T X) (1+\exp(\alpha_0^T X))^{-1})$.   In both settings, the two-dimensional covariate $X \sim \mathcal{U}[0,1] \times \mathcal{U}[0,1]$ and $\alpha_0 = (\cos(\theta_0), \sin(\theta_0))^T$ with $\theta_0 \in \{\pi/4, \pi/3, \pi/2\}$.  From each of these monotone single index models we have drawn  $100$ times $n$ i.i.d. pairs $(X_i, Y_i)$ and computed the LSE $\hat{\alpha}_n$  for $n \in \{10^2, 10^3, 10^4, 10^5 \}$. Based on these $100$ 
replications  we computed the empirical estimates for the covariance matrix of $n^{1/3}(\hat{\alpha}_n - \alpha_0)$ and  $n^{1/2}(\hat{\alpha}_n - \alpha_0)$. The main idea behind is that the correct rate of convergence should yield estimates that are more or less stable for large $n$. Our simulation results for the Gaussian and logistic model are reported in Table \ref{TableGauss} and Table \ref{TableLogistic} respectively. For the settings we have chosen, the variances $\widehat{\sigma}^2_{11}$ and $\widehat{\sigma}^2_{22}$ as well as the absolute value of the covariance $\vert \widehat{c}_{12} \vert$ seem to increase with the sample size $n$ if $n^{1/2}$ is the stipulated rate of convergence. This picture is completely reversed for the rate $n^{1/3}$.  
This first investigation can only allow us to conclude (for the chosen models, true monotone link functions and indices) that the convergence of our LSE occurs at a rate that is faster than $n^{1/3}$ and slower than $n^{1/2}$. 
It is clear that more computations are needed to get a better idea about the actual convergence rate for a general dimension $d \ge 2$.  Proving the exact rate of convergence of $\hat \alpha_n$ is an interesting question but goes beyond the scope of this work. We believe that in establishing this exact rate, under suitable assumptions, it is necessary to overcome the difficulty of non-smoothness of  $\hat{\Psi}_n$, the monotone estimator of the true link function  $\Psi_0$ and the fact that $\hat {\alpha}_n$ and $\hat{\Psi}_n$ are intertwined. Consequently, a useful device such as Taylor expansion  cannot be used. Also, when this $\hat{\Psi}_n$ converges at the cubic rate (as it is the case under our assumptions), it is not immediate how to show that $\hat \alpha_n$ converges at a faster rate as both $\Psi_n$ and $\hat \alpha_n$ depend on each other. We intend to investigate these questions in a future work.

 \begin{table}[!ht]
	\centering
	\caption{Values of the empirical covariances matrices for the case $d=2$  of $n^{1/2} (\widehat{\alpha}_n - \alpha_0)$ and   $n^{1/3} (\widehat{\alpha}_n - \alpha_0)$  with entries $\widehat{\sigma}^2_{11}, \widehat{\sigma}^2_{22}, \widehat{c}_{12} = \widehat{c}_{21}$.   The sample size $n \in \{10^2, 10^3, 10^4, 10^5\}$ and $\alpha_0 = (\cos(\theta_0), \sin(\theta_0))^T$ with $\theta_0 \in \{\pi/4, \pi/3, \pi/2\}$. The obtained  estimates were computed based on $100$ replications  and the model  $Y \sim \text{Bin}(10, \exp(\alpha^T_0 X)/(1+ \exp(\alpha^T_0 X))$ and $X \sim \mathcal{U}[0,1] \times \mathcal{U}[0,1] $.   }
	\label{TableLogistic}
	\vspace{0.5cm}
	\scalebox{0.99}{
		\begin{tabular}{|clr|ccc |}
			\hline
			$\theta_0$ &  Rate & $n$ & $\widehat \sigma^2_{11}$ & $\widehat\sigma^2_{22}$ & $\widehat{c}_{12}$ \\
			\hline
                   &  & & & &   \\
			& $n^{1/2}$ & 100 & 1.291103 & 1.330393 & -1.300230 \\
			& &1000 & 2.981480 & 2.977000  & -2.97093 \\
			& &10000& 5.718460 &5.600101 & -5.654135   \\
			& &100000& 7.1398327  &7.1803650 &  -7.1591489 \\
			$\pi/4$ &  & & & &  \\
			 &  & & & & \\
                    &  $n^{1/3} $ &100  & 0.2781598   &0.2866244 & 0.2801261 \\
			&  &1000 & 0.2981480 & 0.2977000 &  -0.2970937\\
			& &10000  &0.2654274  &0.2599337 &  -0.2624417\\
			&  &100000&  0.1538230  & 0.154696 & -0.1542392\\
		   \hline \hline
                     &  & & & &  \\
			&  $n^{1/2}$ & 100 & 3.398950 &  1.082636 & -1.884304 \\
			&  &1000 &  5.250277 & 1.788438  & -3.047276   \\
			&  &10000   & 9.769384 & 3.307230   &   -5.675153\\
			&  &100000&   13.0589943 & 4.42027715 & -7.59624187  \\
			$\pi/3$ &  & & & &  \\
                   &  & & & & \\
			&  $n^{1/3} $ &100  & 0.7322815 & 0.2332469 & -0.4059610\\
			&  &1000 &  0.5250277 &0.1788438  & -0.3047276\\
			&  &10000&  0.4534546 & 0.1535080  &  -0.2634173 \\
			&  &100000&  0.2813475  & 0.09523198 &  -0.16365607\\
\hline
\hline
 &  & & & &  \\
			&  $n^{1/2}$ & 100 & 6.68905651 & 0.12717579 &  -0.04501911 \\
			&  &1000 &  7.582699  &  0.02960793 &    -0.12527002\\
			&  &10000& 9.22021866   &  0.004287974   &  -0.014856708   \\
			&  &100000&   13.35415000  & 0.0008918812    & 0.0076945903 \\
			$\pi/2$ &  & & & &  \\
                   &  & & & & \\
			&  $n^{1/3} $ &100  & 1.441113539 & 0.027399193 & -0.009699074\\
			&  &1000 &  0.7582699   &  0.002960793  & -0.012527002\\
			&  &10000&  0.4279646395 & 0.0001990301 &  -0.0006895873\\
			&  &100000&  2.877064e-01  & 0.0000192150 & 0.0001657749 \\
\hline
		\end{tabular}}
\end{table}	

\section{Convergence of alternative estimators}\label{sec: cvalternative}  
We now consider convergence of plug-in estimators of 
Section \ref{sec: plug}: we randomly split the sample  into two independent sub-samples of respective sizes $n_1$ and $n_2$, where $n_1$ is the integer part of $pn$ for some fixed $p\in(0,1)$ and $n_2=(1-p)n$,  we compute 
an index estimator $\tilde\alpha_{n}$ based on the first sub-sample, and then compute $\tilde\Psi_n$, the minimizer of \eqref{eq: htilde}  over $\Psi \in \mathcal{M}$, where $\alpha=\tilde \alpha_n$ and $\{(X_i,Y_i),i\in I_2\}$ are the observations from the second sub-sample.
 Note that arguing conditionally on the first sub-sample, $\tilde\alpha_{n}$ can be considered as non-random when studying the limiting behavior of 
$\tilde\Psi_n$. In the sequel, we set $\tilde g_n(x)=\tilde\Psi_n(\tilde\alpha_{n}^Tx)$ for all $x\in\RR^d$. We prove below that, provided that $\tilde\alpha_{n}$ converges at the $n^{1/3}$-rate (which is the case of the linear estimator of Section \ref{sec:linear_est} and some estimators from Section \ref{sec: additional} under appropriate assumptions), $\tilde g_n$ also converges at the same rate. The complete proof is given in Subsection \ref{sec: alternativeconsistency} of the Supplementary Material. Below, we implicitly assume that $\alpha_0$ is identifiable.

\begin{theorem} \label{theo: alternativeconsistency}
Assume (A1)-(A4).  Assume, moreover, that  $\Psi_{0}$ is non-constant and Lipschitz continuous, and that 
$\tilde\alpha_{n}=\alpha_0+O_p( n^{-1/3}).$

With $\tilde g_n$ as above we then have 
\begin{equation}\label{eq: alternativecvg}
\left(\int_{\mathcal{X}} \left(\tilde{g}_n(x)-g_0(x)\right)^2 dx\right)^{1/2}  = O_p(n^{-1/3}).
\end{equation}
Furthermore, if (A1)-(A6) hold, and $\Psi_{0}$ has a first derivative that is bounded from above on  ${\cal C}_{\alpha_0}$, then 
\begin{equation}\label{eq: alternativecvpsi}
\left(\int_{\underline c+v_n}^{\overline c-v_n}(\tilde \Psi_n(t)-\Psi_0(t))^2dt\right)^{1/2}=O_p(n^{-1/3})
\end{equation}
for all sequences $v_n$ such that ${n^{1/3} v_n}\to \infty$ and $\underline c+v_n\le \overline c-v_n$.
\end{theorem}

\begin{rem}\label{rem:AEconcern2_v2} 
Similarly to Remark~\ref{rem:AEconcern2},  the result can be made stronger under additional regularity conditions on the support $\mc X$.   Moreover, similar to Remark \ref{rem :subG}, if instead of assuming that $X$ has a bounded support we assume that it has a sub-Gaussian distribution, then the rate of convergence is only inflated by the factor $(\log n)^{5/3}$. 
\end{rem}



\section{Proof of Theorem \ref{theo: rateL2}}\label{sec: proofs}



As the proof of Theorem \ref{theo: rateL2} is quite long and technical, we first give the main ideas of the proof of this theorem in Subsection \ref{sec: mainsteps} below. Here, we give two preparatory lemmas and an intermediate rate theorem (Theorem \ref{theo: ratelog}).  The latter compares to Theorem \ref{theo: rateL2} but with an additional $\log n$ term in the rate of convergence.  The proof of Theorem \ref{theo: ratelog} requires entropy results that are described in Subsection \ref{sec: entropy} and proved in  subsequent subsections. The proof of Theorem \ref{theo: rateL2} is finally completed in Subsection \ref{proof: ratewithoutlog}.

\subsection{The main steps of the proof of Theorem~\ref{theo: rateL2}}\label{sec: mainsteps}

By definition of the LSE, $\widehat g_n$ maximizes the criterion
\begin{eqnarray}\label{eq: Mn}
\mathbb M_ng&:=&\frac 1n\sum_{i=1}^n  \Big \{Y_ig(X_i)-\frac{g(X_i)^2}{2}  \Big\}
\end{eqnarray}
over the set of all functions $g$ of the form $g(x)=\Psi(\alpha^Tx), x \in \mathcal{X}$ with $\alpha\in \mathcal S_{d-1}$ and $\Psi\in\mathcal M$.   It would have been easier to prove Theorem \ref{theo: rateL2} using standard results from empirical process theory if  the LSE  were known to be bounded in probability  by some constant which is independent of $n$. Unfortunately, we do not know whether this  holds true.    Instead, the following lemma can be established (see Subsection \ref{sec: proofBounded} for a proof).

\begin{lemma}\label{lem: LSElog}
We have 
\begin{equation}\label{eq: minmax}
\min_{1\le k\le n}Y_k \leq  \hat g_n(x)\leq \max_{1\le k\le n}Y_k
\end{equation}
for all $x\in\cal X$.
Moreover, under assumptions (A2) and (A3) we have
\begin{eqnarray*}
\sup_{x\in \mathcal{X}}| \widehat g_n (x)| & \leq & \max_{1\le k\le n}|Y_k| \ \ = \ \ O_p (\log n).
\end{eqnarray*}
\end{lemma}

Note that under the more restrictive  assumption that $Y$ is a bounded random variable (so that $\max_k|Y_k|$ is  bounded), we obtain that $\hat g_n$  is also bounded. This is the case for instance in the current status model which, as explained in the introduction, is a special case of the model we consider with  $Y\in\{0,1\}$. For this reason, the arguments developed in \cite{groeneboom2016current}  in the current status model cannot be directly adapted to our setting.

Now it follows from Lemma \ref{lem: LSElog} that, with arbitrarily large probability, $\widehat g_n$ maximizes $\MM_n g$ over the set of all functions $g$ that are bounded in absolute value by $C\log n$ for some appropriately chosen $C > 0$,  and take  the form $g(x)=\Psi(\alpha^Tx), x \in \mathcal{X}$ with $(\alpha, \Psi) \in \mathcal S_{d-1}  \times \mathcal M$. Denote by $\PP_n$ the empirical distribution corresponding to $(X_1,Y_1),\dots,(X_n,Y_n)$, and  let $\hat f_n(x,y)=y\hat g_n(x)-\hat g_n^2(x)/2$ for $x\in\mc X$ and $y\in\RR$. Since $\MM_ng=\PP_nf$ with
\begin{equation}\label{eq: gtof}
f(x,y)=yg(x)-g^2(x)/2
\end{equation}
for all $x\in\mc X,y\in\RR$,  this means that,  with arbitrarily large probability,   $\widehat f_n$ maximizes
$\PP_nf$ over the set of all functions $f$ of the form \eqref{eq: gtof} for some function $g$ that is bounded in absolute value by $C\log n$ and takes  the form $g(x)=\Psi(\alpha^Tx), x \in \mathcal{X}$ with $(\alpha, \Psi) \in \mathcal S_{d-1}  \times \mathcal M$. Hence, classical arguments for maximizers of the empirical process over a class of functions (where $g$ can be assumed to be bounded by $C\log n$) can be used  to compute the rate of convergence of the estimator. This requires bounds for the entropy of the class of functions $f$ of the form \eqref{eq: gtof} together with a basic inequality that makes the connection between the mean of $\MM_ng-\MM_ng_0$ and a distance between $g$ and $g_0$. The entropy bounds are given in Section \ref{sec: entropy} below whereas the basic inequality is given in the following lemma, which is proved in Subsection \ref{proof: basic}.
For each bounded function $g:\mc X\to\RR$, we define $\mathbb{Q}g=\int gd\mathbb Q$ and $\MM g=\PP f$ where $f$ is given by \eqref{eq: gtof} and  $\PP f=\int fd\PP$, which means that $\MM g$ is the expected value of $\MM_ng$:
\begin{equation}\label{eq: M}
\mathbb{M}g=  \int_{\mathcal{X} \times \RR} \Big \{ y g(x)  - \frac{g^2(x)}{2}  \Big\} {d\mathbb P(x,y)}. 
\end{equation}
\begin{lemma}\label{lem: basic}
Let $g:\mc X\to\RR$ with $\mathbb{Q}g^2<\infty$. Then,
$
\mathbb Mg-\mathbb Mg_0 \le  -   D^2(g,g_0)/2$
where
\begin{equation}\label{eq: defD}
D(g,g_0)= \left( \int_{\mathcal{X}}  (g(x) - g_0(x))^2 d\mathbb{Q}(x) \right)^{1/2}.
\end{equation}
\end{lemma}

If classical arguments for maximizers over a class of functions  are used based on the previous basic inequality and Lemma \ref{lem: LSElog} (which allows to restrict attention to functions $g$ that are bounded by $C\log n$, for some large $C>0$ that does not depend on $n$),  the obtained rate of convergence would be  inflated by a logarithmic factor:

\begin{theorem}\label{theo: ratelog}
Assume that assumptions (A1)-(A3) are satisfied and that there exists a constant $\overline q>0$ such that for all $\alpha\in {\cal S}_{d-1}$, with respect to the Lebesgue measure, the variable $\alpha^TX$ has a density which is bounded by $\overline q$. Then,  $D(\hat g_n,g_0)=O_{p}(n^{-1/3}(\log n)^{5/3}).$  More precisely, for all $\eps>0$, there exists $A>0$ that depends only on $\eps,a_0$ and $M$ such that $$P\left(D(\hat g_n,g_0)>Ad(1+\sqrt{\overline q R})n^{-1/3}(\log n)^{5/3}\right)\leq\eps.$$
\end{theorem}

It may seem  superfluous to add Theorem \ref{theo: ratelog}. However, the obtained unrefined rate of convergence will be used to get rid of the additional logarithmic factor.   To explain how this works, 
set $v =  C n^{-1/3} (\log n)^2$ and $K  = C \log n$  for some constant $C > 0$ to be chosen appropriately. Lemma \ref{lem: LSElog} and Theorem \ref{theo: ratelog} are used to show that with a probability that can be made arbitrarily large by choice of $C$, the LSE $\widehat g_n$ is bounded by $K$ while $D(\widehat g_n,g_0)$ is smaller that $v$. Hence, with arbitrarily large probability, the LSE maximizes $\mathbb M_n g$ over the set $\mathcal G_{Kv}$ of all functions $g$ of the form $g(x)=\Psi(\alpha^Tx)$ with $\alpha\in \mathcal S_{d-1}$ and $\Psi\in\mathcal M$ such that $|g(x)|\leq K$ for all $x \in \mathcal{X}$ and 
\begin{equation}\label{eq: defv}
D(g,g_0)\leq v.
\end{equation}
Hence, although  optimization cannot be restricted to a set of functions that are uniformly bounded  in $n$, we can work with a class of functions that are bounded in the $L_2(\mathbb{Q})$-norm. The merit of the latter is that under (A2) and equivalence of $\mathbb{Q}$ with the Lebesgue measure, a function $g\in\mc G_{Kv}$  can be shown to exceed 
$2K_0$ only on a subset of $\mc X$ with Lebesgue measure of maximal order $(v/K_0)^2$.
The fact that considered functions $g$ are bounded by $2K_0$ except on such a small region will balance out the large values that $g$ might have on the same region, and this will prove to be very advantageous in computing the final entropy of the original class of functions.   To estimate this entropy, each function $g(.)   = \Psi(\alpha^T .)$  will be decomposed as follows:
\begin{eqnarray}\label{decomp}
g=(g-\bar g)+\bar g
\end{eqnarray}
where $\bar g$ is the  truncated version of $g$ defined by 
\begin{eqnarray}\label{eq:barg}
\bar g(x)=\begin{cases}g(x)&\mbox{if }|g(x)|\leq 2K_0\\
2K_0&\mbox{if  }g(x)> 2K_0\\
-2K_0&\mbox{if }g(x)<- 2K_0.
\end{cases}
\end{eqnarray}
The set of all possible $\bar g$ forms now a class of bounded functions on which standard arguments from empirical processes theory apply. On the other hand, the  differences $g-\bar g$ form a set of functions whose supremum norm increases with $n$ and which, by the discussion above,  take the value zero except on regions of a very small size. Those two classes of functions will be treated with different arguments. The assumption that $\alpha^TX$ has a density bounded from above will be used to compute entropy bounds for the former class (see Lemma \ref{lem: entropy} and the preceding comment) whereas the assumption of a density bounded away from zero will be used to compute entropy bounds for the latter class (see Lemma \ref{lem: entropyg-barg} and the preceding comment).  Below are some entropy results required in the proof of Theorems~\ref{theo: ratelog} and \ref{theo: rateL2}. The complete proofs of the theorems are given in Subsections~\ref{sec: proofRate} and \ref{proof: ratewithoutlog} whereas the proofs of the lemmas are given in Subsections \ref{sec: proofBounded} to  \ref{proof: Bernstein}.

\subsection{Entropy results}\label{sec: entropy}

We begin with some notation. For any class of functions $\cal F$ equipped with a norm $\|\cdot\|$, and $\eps>0$, we denote by 
$H_B(\eps, \mc F, \|\cdot\|)$ the corresponding bracketing entropy:  
\begin{eqnarray*} 
H_B(\eps, \mc F, \|\cdot\|)=\log N_B(\eps, \mc F, \|\cdot\|)
\end{eqnarray*}
where $N_B(\eps, \mc F, \|\cdot\|)=N$ is the smallest number of pairs of functions $(f_1^U,f_1^L)$, \ldots, $(f_N^U,f_N^L)$ such that all $\|f_j^L- f_j^U\|\leq \eps$ and for each $f\in \mc F,$ there exists a $j  \in \{1, \ldots, N \}$ such that $f_j^L \leq f \leq f_j^U.$  Moreover, assuming that $X$ has a bounded support $\mc X$, we set 
$R = \sup_{x \in \mathcal{X}}  \Vert x \Vert $
where $\|\,.\,\|$ denotes the Euclidean norm in $\RR^d$.
We then have
\begin{equation}\label{eq: R}
 \vert \alpha^Tx \vert \le R\quad for\ all\ x\in\mc X
 \end{equation}
for all $\alpha\in{\cal S}_{d-1}$, using the Cauchy-Schwarz inequality. 
In the rest of the paper,  we will use the following notation
\begin{list}{--}
        {
        \setlength{\topsep}{6pt}
        \setlength{\parskip}{0pt}
        \setlength{\partopsep}{0pt}
        \setlength{\parsep}{0pt}
        \setlength{\itemsep}{3pt}
        \setlength{\leftmargin}{20pt}}
\item $\|\,.\,\|_{\PP}$ and $\|\,.\,\|_{\mathbb Q}$ are the $L_2$-norms corresponding to respectively $\PP$ and $\mathbb Q$: 
$\|f\|_{\PP}^2=\int_{\mc X\times \RR} f^2(x,y)d\PP(x,y)$ and $\|g\|_{\mathbb Q}^2=\int_{\mc X} g^2(x)d\mathbb Q(x)$ for all $f:{\cal X}\times\mathbb R\to\mathbb R$ and $g:{\cal X}\to\mathbb R$.
\item $\mc M_{K}$ is the class of all nondecreasing functions on $\mathbb R$ that are bounded in absolute value by $K$
\item $\mathcal{G}_{K}$ is the class of functions $g(x)  = \Psi(\alpha^T x),    \ x \in \mathcal{X}$ where $ \alpha \in \mathcal{S}_{d-1}, \Psi \in \mc M_{K}$,
\item $\mathcal{F}_K$ is the class of functions $f$ of the form \eqref{eq: gtof}, $ x  \in \mathcal{X}, y  \in \RR$, and $g \in \mathcal{G}_K$,
\item $\mathcal{G}_{Kv}$ is the class of functions $g\in  \mathcal{G}_K$  satisfying  the condition \eqref{eq: defv},
\item $\mathcal{F}_{Kv}$ is the class of functions $f$ of the form \eqref{eq: gtof}, $ x  \in \mathcal{X}, y  \in \RR$, and $g \in \mathcal{G}_{Kv}$,

\item $\overline{\mathcal{G}}_{Kv}$ is the class of functions $g - \bar{g}$ , where $g \in \mathcal{G}_{Kv}$ and $\bar{g}$ is  given in (\ref{eq:barg}),
\item $\overline{\mathcal{F}}_{Kv}$ is the class of functions $f - \bar{f}$, where $f$ takes the form \eqref{eq: gtof} for some $g \in \mathcal{G}_{Kv}$ and $\bar{f}(x, y)  = y \bar{g}(x)  - \bar{g}^2(x)/2, \ x  \in \mathcal{X}, y  \in \RR$.
\end{list}

Our starting point is the  following result, which follows from Theorem 2.7.5 in \cite{vanweak}.

\begin{lemma}\label{lem:entropyMono}    There exists a universal constant $A > 0$ such that
\begin{eqnarray*}
 H_B(\eps, \mc M_{K}, \|\,.\,\|_{Q})\leq\frac {A K}\eps 
\end{eqnarray*}
for all $\eps>0$, $K>0$, and all probablity measures $Q$ on $\mathbb R$, where $ \|\,.\,\|_{Q}$ is the $L_2$-norm corresponding to $Q$: $ \|\Psi\|^2_{Q}=\int \Psi^2d Q$ for all $\Psi:\mathbb R\to\mathbb R$.
\end{lemma}

The next result, which follows from Lemma 22 of \cite{feige2002}, gives a bound on the minimal number of subsets with diameter at most $\eps,$ say,  into which $\mathcal{S}_{d-1}$ can be divided. Here, the diameter of some given subset $\mathcal{A}  \subset \mathcal{S}_{d-1}$,  is given by  $\sup_{(x, y) \in \mathcal{A}^2}  \Vert x - y \Vert$.

\begin{lemma}\label{lem:partsphere}
Fix $\eps \in (0, \pi/2)$,  and let  $N(\eps, \mathcal{S}_{d-1})$ be the number of subsets of equal size with diameter at most $\eps$ into  which $\mathcal{S}_{d-1}$ can be partitioned. Then,  there exists a universal constant $A > 0,$ such that 
 $N(\eps, \mathcal{S}_{d-1})  \le \left(A/\eps\right)^d.$
\end{lemma}

In what follows, we assume that the assumptions (A1)-(A3) are satisfied.
The next step is to  use the results above to construct $\eps$-brackets for the classes $\mathcal{G}_{K}$, $\mathcal{F}_{K}$, $\overline{\mathcal{G}}_{Kv}$ and $\overline{\mathcal{F}}_{Kv}$. We begin with the classes $\mc G_{K}$ and $\mathcal{F}_{K}$. In the next lemma, we assume that $\alpha^TX$ has a bounded density on a bounded support for all $\alpha$. This assumptions is used in the proof to show that with $Q$ the distribution of $\alpha^TX$ with arbitrary $\alpha\in{\cal S}_{d-1}$, and $\Psi\in{\cal M}_K$, there exists a constant C such that
$$\int_\RR(\Psi(t+u)-\Psi(t-u))dQ(t)\leq Cu.$$

\begin{lemma}\label{lem: entropy} Assume that the assumptions (A1)-(A3) are satisfied and that there exists $\overline q>0$ such that for all $\alpha\in {\cal S}_{d-1}$, with respect to the Lebesgue measure, the variable $\alpha^TX$ has a density that is bounded by $\overline q$. Let $K>\eps  >0$. There exists a universal constant $A_1 > 0$ such that 
\begin{eqnarray*}
H_B\Big(\eps, \mc G_{K}, \|\,.\,\|_{\mathbb{Q}}\Big)\leq\frac {A_1K d(1+\sqrt{\overline q R})}{\eps}.
\end{eqnarray*}
Moreover, if $K>1$ then there exists $A_2 > 0$  depending only on $a_0$ such that
\begin{eqnarray*}
H_B\Big(\eps, \mc F_{K}, \|\,.\,\|_{\PP}\Big)\leq\frac {A_2K^2 d(1+\sqrt{\overline q R})}{\eps}.
\end{eqnarray*}
\end{lemma}

The next lemma will be used to control the differences $g(X)-\overline g(X)=h(\alpha^TX )$, $g\in\mc G_{Kv}$. Here, we assume that for all $\alpha$, the variable $\alpha^TX$ has a density that is bounded away from zero on its support. The assumption is used to show that under \eqref{eq: defv}, $h=0$ except on a set whose Lebesgue measure is at most $\underline q^{-1}K_0^{-2}v^2$, see \eqref{eq: whylb} below. Since the distribution of $\alpha^TX$ is also assumed to have a bounded density with respect to the Lebesgue measure, this implies that the probablity that $h(\alpha^TX )\neq 0$ is of maximal order $v^2$, for all $\alpha$, leading to a sharp bound for $\bar{ \mc G}_{Kv}$ and then $\bar{ \mc F}_{Kv}$.

\begin{lemma}\label{lem: entropyg-barg}  Assume that the assumptions (A1)-(A4) are satisfied. Let $\eps>0$ and $v>0$. There exists a  constant $A_1  > 0$ depending only on $K_0$, $\overline q$, $\underline q$ and $R$ such that 
\begin{eqnarray*}
H_B \Big(\eps, \bar{ \mc G}_{Kv}, \|\,.\,\|_\QQ \Big)\leq  \ \frac{A_1 Kv}{\eps}  + d \log\left( \frac{A_1 K^2}{\eps^2} \right)
\end{eqnarray*}
for all $K>\eps$ such that $Kv>\eps K_0\sqrt{2R\underline q}$. Moreover,  there exists $A_1 > 0$ depending only on $K_0$, $\overline q$, $\underline q$, $R$ and $a_0$ such that for all  $K>1\vee\eps$ such that $K^2v>\eps K_0\sqrt{2R\underline q}$, we have
\begin{eqnarray*}
H_B\Big(\eps, \bar{ \mc F}_{Kv}, \|\,.\,\|_{\PP} \Big)\leq    \frac{A_1 K^2v}{\eps} + d \log\left( \frac{A_1K^4}{\eps^2} \right).
\end{eqnarray*}  
\end{lemma}

The next lemma will be needed to give entropy bounds in the Bernstein norm. We recall that the Bernstein norm of some function $f$ with respect to $\PP$ is defined by 
\begin{eqnarray*}
\Vert f \Vert_{B, \mathbb{P}}  =  \Big(2 \PP( e^{\vert f \vert }  - 1 - \vert f \vert )  \Big)^{1/2}.
\end{eqnarray*}  
Although not technically a norm, it is typically referred to as such in the literature \citep[page 324]{vanweak}, and we do not stray from this in what follows.  With $\tilde{C} > 0$ a constant appropriately chosen,  Lemma \ref{Bernsteinentropy} below derives from Lemmas \ref{lem: entropy} and \ref{lem: entropyg-barg}  upper bounds on the  bracketing number of the classes 
\begin{eqnarray}\label{eq: tildeF}
\widetilde{\mathcal{F}}_{Kv}  :=\Big \{(f - f_0) \tilde{C}^{-1}, f \in \mathcal{F}_{Kv}  \Big \}  
\end{eqnarray}
where $f_0(x,y)=yg_0(x)-g_0^2(x)/2$ and 
\begin{eqnarray}\label{eq: tildebarF}
 \widetilde{\bar{ \mc F}}_{Kv}:= \Big \{ f
\tilde{C}^{-1}, f \in \bar {\mathcal{F}}_{Kv} \Big \}   
\end{eqnarray}
with respect to the Bernstein norm.  This will enable us to use Lemma 3.4.3 of \cite{vanweak}  which does not require  the class of functions of interest to be bounded. 

\begin{lemma}\label{Bernsteinentropy} Assume that the assumptions (A1)-(A3) are satisfied and that there exists $\overline q>0$ such that for all $\alpha\in {\cal S}_{d-1}$, with respect to the Lebesgue measure, the variable $\alpha^TX$ has a density that is bounded by $\overline q$. Let $\eps>0$ and $v>0$.   Let $M$ be the same constant from the moment condition (\ref{eq: hypexpo}) of Assumption (A3). Let $\tilde{C} = 4M K^2$ such that $ K \ge (2K_0) \vee 2$. Then, there exist constants $A_1>0$ and $A_2 > 0$ that depend on $a_0$ and $M$ only  such that 
\begin{eqnarray*}
H_B\Big(\eps, \widetilde{\mathcal{F}}_{Kv}, \Vert \cdot \Vert_{B, \PP} \Big)\leq\frac {A_1d(1+\sqrt{\overline qR}) }{\eps}, \  \ \textrm{and}  \ \  \Vert \tilde{f}  \Vert_{B, \PP}   \le A_2 v
\end{eqnarray*}
for  all $\tilde f  \in \widetilde{\mathcal{F}}_{Kv}$.  If moreover, the assumption (A4) is fulfilled, then there exist constants $A_1>0$ and $A_2 > 0$ depending  only on $K_0$, $R$, $a_0$, $M$, $\bar{q}$ and $\underline q$  such that 
\begin{eqnarray*}
H_B\Big(\eps, \widetilde{\bar{ \mc F}}_{Kv}, \Vert \cdot \Vert_{B, \PP} \Big)\leq  \frac{A_1 v}{\eps} + d \log\left( \frac{A_1 }{\eps^2} \right) , \  \ \textrm{and} \ \ \Vert  \tilde{f} \Vert_{B, \PP}   \le A_2 v
\end{eqnarray*}
for  all $\tilde{f} \in \widetilde{\bar{ \mc F}}_{Kv}$, provided that $K^2v>A_2\eps$ and $K>\eps$.
\end{lemma}

Note that the condition $\eps < 2$ guarantees that $K > \eps$ since $K\ge 2$. Also, we point out that the constants $A_1$ and $A_2$ may not be  the same ones as in Lemma \ref{lem: entropyg-barg} but we can always increase their respective values so that they are suitable for Lemma \ref{Bernsteinentropy}.

\subsection{Proof  of Lemma \ref{lem: LSElog} } \label{sec: proofBounded}
For a fixed $\alpha\in{\cal S}_{d-1}$, let $\widehat\Psi_n^{\alpha}$ be a minimizer of 
$h_n(\Psi,\alpha)$ over $\Psi \in \mathcal{M}$.  It follows from Theorem 1 in  \cite{barlowstatistical} that $\widetilde\Psi_n(Z_k)=d_k$ for $k =1, \ldots, m$, where
$Z_1<\dots<Z_m$ are the ordered distinct values of $\alpha^TX_1,\dots,\alpha^TX_n$, and 
$d_1\leq \ldots\leq d_m$ are the left derivatives of the greatest convex minorant  of the cumulative sum diagram defined by the set of points 
$$\left\{ (0,0), \left(\sum_{i=1}^n\mathbb{I}_{ \alpha^T X_i\le Z_k}, \sum_{i=1}^nY_i\mathbb{I}_{\alpha^T X_i\le Z_k}\right), k =1, \ldots, m \right\}.$$
Hence we have
$$\min_{1\le k\le m}\frac{\sum_{i=1}^nY_i\mathbb{I}_{\alpha^T X_i\le Z_k}}{\sum_{i=1}^n\mathbb{I}_{ \alpha^T X_i\le Z_k}}\leq  \hat \Psi_n^\alpha(\alpha^TX_j)\leq \max_{0\le k\le m-1}\frac{\sum_{i=1}^nY_i\mathbb{I}_{\alpha^T X_i> Z_k}}{\sum_{i=1}^n\mathbb{I}_{ \alpha^T X_i> Z_k}}, $$
with $Z_0=-\infty$,
for all $j=1,\dots,n$. Therefore, 
$\min_{1\le i\le n}Y_i\leq  \hat \Psi_n^\alpha(\alpha^TX_j)\leq \max_{1\le i\le n}Y_i$
for all $\alpha\in\mathcal S_{d-1}$ and $j=1,\dots,n$. The inequalities  in \eqref{eq: minmax} follow since by definition, $\hat g_n(x)$ takes the form $\hat \Psi_n^\alpha(\alpha^Tx)$ for all $x$, where $\alpha$ is replaced by the LSE $\hat\alpha_n\in\mathcal S_{d-1}$.

Now we prove that under the assumptions (A2) and (A3),
\begin{eqnarray}\label{eq: bound log}
\max_{1\leq i\leq n}|Y_i|=O_p(\log n).
\end{eqnarray}
For an integer $s \ge 2$, it follows from convexity of the function $z\mapsto |z|^s$ on $\RR$ that 
\begin{eqnarray*}
E[\vert Y - g_0(x) \vert^s   | X =x ]  & \le  &  2^{s-1}  \Big (E[\vert Y \vert^s | X=x]    + |g_0(x)|^s \Big)  \\
& \le & 2^{s-1}  \Big(s! a_0 M^{s-2}  +  K^s_0  \Big)    
\le   s! b_0 (M')^{s-2}
\end{eqnarray*}
with $b_0 = 2 (a_0 + K^2_0)$ and $M' = 2 (M \vee K_0)$. Now, using  Lemma 2.2.11 of \cite{vanweak} with $n=1$ and $\tilde{Y}= Y - g_0(x)$ and after integrating out with respect to $d\mathbb{Q}$, we obtain
\begin{eqnarray*}
P\left(\vert Y-g_0(X) \vert > t  \right )  \le 2 \exp\left(\frac{-t^2}{2 (2 b_0 + M't) }\right)
\end{eqnarray*}
for all $ t > 0$.   Hence, with $t = C \log n$ such that $K_0  <  C \log (n) /2 $ we have that  
\begin{eqnarray*}
P\left( \max_{ 1 \le i \le n}   \vert Y_i  \vert  > C \log  n  \right )
&  \le&   \sum_{i =1}^n P \left ( \vert Y_i-g_0(X_i) \vert  > C \log (n)/2  \right)  \\
& \le& 2n \exp \left(\frac{-C^2 \log n}{4 \left(4 b_0 (\log n)^{-1}+  M'C\right) }\right)
\end{eqnarray*}
which converges to $0$ as $n \to \infty$ provided that $C$ is sufficiently large so that
$8M'C<C^2$. Lemma \ref{lem: LSElog} follows.   \qed

\subsection{Proof of Lemma \ref{lem: basic}}\label{proof: basic}
Since $E[Y|X=x]  = g_0(x)$ we have
\begin{eqnarray*}
\notag
\mathbb Mg-\mathbb Mg_0
&=&\int_{\mathcal{X}} \left\{g_0(x)(g(x)-g_0(x))-\frac{g(x)^2}{2}+\frac{g_0(x)^2}{2}\right\}d\mathbb Q(x) \\
&=&-\frac{1}{2} \int_{\mathcal{X}}  (g(x)  - g_0(x))^2  d\mathbb{Q}(x),
\end{eqnarray*}  
which proves Lemma \ref{lem: basic}. \qed

\subsection{Proof of Lemma \ref{lem: entropy}}  \label{proof: entropy}
Let  $\eps_\alpha=\eps^2K^{-2} \in (0,1).$ 
By Lemma~\ref{lem:partsphere},  $\mc S_{d-1} $  can be covered by $N$ neighborhoods with diameter at most $\eps_\alpha$ where
$N \le   (A  \eps_\alpha^{-1})^d$ 
with $A>0$ a universal constant.  Let $\{\alpha_1, \ldots, \alpha_N\}$ denote elements of each of these neighborhoods. 
Now, consider an arbitrary $g\in\mc G_{K}$.  Then, $g(x)  = \Psi(\alpha^T x), \ x \in \mathcal{X},$ for some $\Psi  \in \mathcal{M}_{K}$ 
and $\alpha \in \mathcal{S}_{d-1}$. 
We can find $i  \in \{1, \ldots, N \}$ such that $\Vert \alpha  - \alpha_i \Vert \le \eps_\alpha$. Then, using the monotonicity of $\Psi$ together with the Cauchy-Schwarz inequality  we can write  for all $x\in  \mathcal{X}$ that \begin{eqnarray*}
g(x) &=& \Psi(\alpha_i^T x + (\alpha-\alpha_i)^T x)
\leq\Psi(\alpha_i^T x+\eps_\alpha R) 
\end{eqnarray*}
 and $
g(x)\geq\Psi(\alpha_i^T x-\eps_\alpha R) $.    
Then,  Lemma \ref{lem:entropyMono} implies that with $N'=\exp( A K\eps^{-1})$, at the cost of increasing $A$, 
we can find brackets  $[\Psi^L_j,   \Psi^U_j]$  covering the class of functions $\mathcal{M}_{K}$ such that 
\begin{eqnarray*}
\int_{\mathbb R}  (\Psi^U_j(t)  - \Psi^L_j(t) )^2  d Q^-_i(t)   \le \eps^2\quad and\quad \int_{\mathbb R}  (\Psi^U_j(t)  - \Psi^L_j(t) )^2  d Q^+_i(t)   \le \eps^2
\end{eqnarray*}
for $j=1, \ldots, N'$, where $Q_i^-$ and $Q_i^+$ respectively denote the distribution of $\alpha^T_iX-\eps_\alpha R$ and $\alpha^T_iX+\eps_\alpha R$. Now returning to  $g$, and using that $\Psi\in\mathcal{M}_{K}$, we can see that 
\begin{eqnarray}\label{eq: encadPsi}
\Psi^L_j(\alpha_i^T x-\eps_\alpha R)   \le  g(x)  \le \Psi^U_j(\alpha_i^T x+\eps_\alpha R) 
\end{eqnarray}
for some $j=1,\dots,N'$ and  all $x \in \mathcal{X}$. 
We will show that there exists  $B > 0$  depending only on $\overline q, R$ such that  the new bracket $[\Psi^L_j(\alpha_i^T x-\eps_\alpha R), \Psi^U_j(\alpha_i^T x+\eps_\alpha R) ],   \ x \in \mathcal{X}$ satisfies 
\begin{eqnarray}\label{eq: brackg}
&&\Bigg(\int_{\mc X} \Big(\Psi_{j}^U(\alpha_i^T x + \eps_\alpha R)-\Psi_{j}^L(\alpha_i^T x-\eps_\alpha R)\Big)^2 d\mathbb Q(x)\Bigg)^{1/2}
\leq  B \eps.
\end{eqnarray}
 It follows from the Minkowski inequality that the left-hand side in \eqref{eq: brackg} is at most
\begin{eqnarray}\label{eq: Mink} \notag
&& \Bigg(\int_{\mathcal{X}} \Big(\Psi(\alpha_i^T x  - \eps_\alpha R)-\Psi^L_j(\alpha_i^T x - \eps_\alpha R)\Big)^2d\mathbb Q(x) \Bigg)^{1/2} \notag \\
&& \qquad +\Bigg(\int_{\mathcal{X}} \Big(\Psi^U_j(\alpha_i^T x+\eps_\alpha R)-\Psi(\alpha_i^T x+\eps_\alpha R)\Big)^2d\mathbb Q(x) \Bigg)^{1/2} \notag \\
&&\qquad\qquad +\Bigg(\int_{\mathcal{X}}  \Big(\Psi(\alpha_i^T x+\eps_\alpha R)-\Psi(\alpha_i^T x-\eps_\alpha R)\Big)^2d\mathbb Q(x)\Bigg)^{1/2}.
\end{eqnarray}
We have
\begin{eqnarray*}
\int_{\mathcal{X}} \Big(\Psi(\alpha_i^T x  - \eps_\alpha R)-\Psi^L_j(\alpha_i^T x - \eps_\alpha R)\Big)^2d\mathbb Q(x)  & =  &  \int_{\mathbb R}  \left(\Psi(t)  - \Psi^L_j(t)  \right)^2  dQ_i^-(t)  \\
& \le & \eps^2
\end{eqnarray*}
and a similar bound is found for the square of the second integral in (\ref{eq: Mink}).  Hence, the left-hand side in \eqref{eq: brackg} is less than or equal to
\begin{eqnarray*}
2\eps +\Bigg(\int_{\mathbb R}  \Big(\Psi(t+\eps_\alpha R)-\Psi( t-\eps_\alpha R)\Big)^2d Q_i(t)\Bigg)^{1/2}
\end{eqnarray*}
where $Q_i$ is the distribution of $\alpha_i^TX$.  
By monotonicity of $\Psi$ and the fact that it is bounded in absolute value by $K$, 
we can write 
\begin{eqnarray*}
\int_{\mathbb R}  \Big(\Psi(t+\eps_\alpha R)-\Psi( t-\eps_\alpha R)\Big)^2d Q_i(t)
&\le&2K \int_{\mathbb R} \Big(\Psi(t+\eps_\alpha R)-\Psi( t-\eps_\alpha R)\Big)d Q_i(t)\\
&\le&
    2K\overline q \int_{-R}^{R}  \Big(\Psi(t + \eps_\alpha R)  - \Psi(t - \eps_\alpha R)\Big) dt. 
\end{eqnarray*}
with $\overline q$ an upper bound of the density of $Q_i$, that is supported on $[-R,R]$, with respect to the Lebesgue measure. 
This is at most
\begin{eqnarray*}
2K  \overline q\left(  \int_{R - \eps_\alpha R}^{R+\eps_\alpha R} \Psi(t)  dt  -  \int_{-R-\eps_\alpha R}^{-R+\eps_\alpha R}  \Psi(t)  dt  \right)  
 \le 8 \overline q R \eps^2,
\end{eqnarray*}

using that $\eps_\alpha  =  \eps^2/K^2$. Hence,
\begin{eqnarray}\label{eq: Entalpha}
\Bigg(\int_{\mathcal{X}}  \Big(\Psi(\alpha_i^T x+\eps_\alpha R)-\Psi(\alpha_i^T x-\eps_\alpha R)\Big)^2d\mathbb Q(x) \Bigg)^{1/2}  \le (8 \overline q R)^{1/2}\eps.
\end{eqnarray}
Combining the inequalities above, we get the claimed inequality in (\ref{eq: brackg})  with 
$B=2+(8 \overline q R)^{1/2}$. It follows that 
\begin{eqnarray}\label{eq: entropyGNN'}
H_B(B\eps, \mc G_{K}, \|\,.\,\|_{\mathbb Q}  ) &\leq  &  \log N+\log N'  \\ \notag
&\le  & d \log \left( A K^2\eps^{-2}  \right)  +  A   K\eps^{-1} \\ \notag
& \le &   K\eps^{-1}\Big(d A^{1/2}  + A \Big) 
\end{eqnarray}
since $\log x\leq \sqrt x$ for all $x>0$. The first assertion of  Lemma \ref{lem: entropy} follows.

To prove the second assertion,   we need to build brackets for the class of  functions $(x,y)\mapsto yg(x)$, $x\in\mc X,y\in\RR$, and then for the class of functions $g^2,$  with $g \in \mathcal G_{K}$. In the following, we denote the former class by $\mathcal{G}_1$ and the latter by $\mathcal{G}_2$  with which we begin. Note that $g^2(x)=  s(x) = \Psi^2(\alpha^Tx) = h(\alpha^Tx)$ for some function $h$ that is either monotone non-decreasing, monotone non-increasing or $U$-shaped depending on the sign of $\Psi$.  Hence, the function $h$ can be always decomposed into the difference of two monotone functions that are bounded by $K^2$. If $K^2>\eps$ (which holds for all $\eps>0$ and $K>\eps$ such that $K>1$), we can use similar arguments as above  to conclude that there exists a  universal constant $B_0 > 0$ such that 
\begin{equation}\label{eq: entG2bis}
\displaystyle H_B\Big(\eps, \mc G_2, \|\,.\,\|_{\QQ} \Big)\leq \frac {B_0K^2 d(1+\sqrt{\overline q R})}{\eps}.
\end{equation}
Using the fact that any element $s \in \mc G_2$ satisfies 
\begin{eqnarray*}
\int_{\mathbb{R}} \int_{\mathcal{X}}  s^2(x)  d\PP(x, y)  &=  & \int_{\mathcal{X}}  s^2(x) d\QQ(x),
\end{eqnarray*}
it follows that  
\begin{equation}\notag
\displaystyle H_B\Big(\eps, \mc G_2, \|\,.\,\|_{\PP}  \Big)\leq \frac {B_0K^2 d(1+\sqrt{\overline q R})}{\eps}.
\end{equation}
Now we turn to  $\mc G_1$. With $N=N_B(\eps, \mc G_{K},  \|\,.\,\|_{\QQ})$, we will denote by $\{(g_i^L, g_i^U),i\in\{1,\dots,N\}\}$  a  cover of $\eps$-brackets for $ \mathcal{G}_{K}$. For all $i =1, \ldots, N$, define
\begin{equation}\label{eq: bracketG1}
k_i^U(x,y)=\begin{cases}
yg_i^U(x)&\mbox{if }y\geq 0\\
yg_i^L(x)&\mbox{if }y\leq 0
\end{cases},
\qquad
k^L_i(x,y)=\begin{cases}
yg^L_i(x)&\mbox{if }y\geq 0\\
yg^U_i(x)&\mbox{if }y\leq 0
\end{cases}.
\end{equation}
Now, take $g\in \mathcal{G}_{K}$ and let  $i  \in \{1, \ldots, N\}$ such that $g_i^L\leq g\leq g_i^U$. Then, we have 
$k^L_i(x,y)\leq yg(x) \leq k_i^U(x,y)$
so that $\{(k_i^L, k^U_i), \ i\in\{1,\dots,N\}\}$ form a bracketing cover for $\mc G_{1}$. We will now compute its size. We have that
\begin{eqnarray*}
  \int_{\RR\times \mathcal{X}} \left(k_i^U(x, y)- k^L_i(x, y) \right)^2 d\mathbb P(x, y)  &=& \int_{\RR\times\mathcal{X}}  y^2\times \left(g_i^U(x)- g^L_i(x)\right)^2 d\mathbb P(x,y)\\
&\leq &2 a_0\int_{\mathcal{X}} \left(g_i^U(x)- g^L_i(x)\right)^{2} d\mathbb Q(x)\\
&\le& 2 a_0  \ \eps^2
\end{eqnarray*}
where $a_0$ is taken from \eqref{eq: hypexpo}. Hence,
\begin{eqnarray} \notag \label{eq: entropyG1}
H_B\Big(\sqrt{2a_0}\eps,  \mc G_{1}, \|\,.\,\|_{\PP} \Big)&\leq& H_B\Big(\eps, \mc G_{K},  \|\,.\,\|_{\QQ} \Big)\\
&\le &\frac{A_1 K d(1+\sqrt{\overline q R})}{\eps},
\end{eqnarray}
using the first assertion of the lemma.  But for all $\eps>0$, we have
\begin{equation}\label{eq: 4.8}
H_B(\eps, \mc F_{K}, \|\,.\,\|_{\PP})\leq H_B(\eps/2, { \mc G}_{1},\|\,.\,\|_{\PP})+H_B(\eps, { \mc G}_{2}, \|\,.\,\|_{\PP})
\end{equation}
and hence we obtain the second assertion of the lemma, which completes the proof.
\qed

\subsection {Proof of Lemma \ref{lem: entropyg-barg}} 
Let  $g\in\mc G_{Kv}$ and $\Psi\in{\cal M}_K,\alpha\in{\cal S}_{d-1}$  such that $g(x)=\Psi(\alpha^T x)$. 
We shall show below that $\vert \Psi(t) \vert \le 2K_0$ except on a region of small size.  To do that,
we shall use the condition \eqref{eq: defv} together with the triangle inequality to get
\begin{eqnarray*}
\int_{\mathcal{X}} \mathds{1}_{\{|\Psi(\alpha^Tx)|>2K_0\}}d\mathbb Q(x)&\leq&\int_{\mathcal{X}}\mathds{1}_{\{|\Psi(\alpha^Tx)-\Psi_0(\alpha_0^Tx)|>K_0\}}d\mathbb Q(x)\\
&\le&\int_{\mathcal{X}} \left(\frac{\Psi(\alpha^Tx)-g_0(x)}{K_0}\right)^2  d\mathbb Q(x)
\le v^2K^{-2}_0.
\end{eqnarray*}
With $a,b$ the boundaries of the interval $\{\alpha^Tx,x\in\mc X\}$ and $Q_\alpha$  the distribution of $\alpha^TX$ we have
\begin{eqnarray*}
\int_{\mathcal{X}}\mathds{1}_{\{|\Psi(\alpha^Tx)|>2K_0\}}d\mathbb Q(x) &=& \int_{a}^{b}\mathds{1}_{\{|\Psi(t)|>2K_0\}}dQ_\alpha(t)\\
&\geq&\underline q\int_{a}^{b}\mathds{1}_{\{|\Psi(t)|>2K_0\}}dt.
\end{eqnarray*}
 Combining the two preceding displays, we conclude that
\begin{eqnarray}\label{eq: whylb}
\int_{a}^{b}\mathds{1}_{\{|\Psi(t)|>2K_0\}}dt\leq   D_2 v^2
\end{eqnarray}
where $\displaystyle D_2  =\underline q^{-1}K^{-2}_0  $.   By monotonicity of $\Psi$, this means that $|\Psi(t)|\leq 2K_0$ for all $t$ in the interval
$[a+ D_2 v^2,b-D_2 v^2].$
Now, from \eqref{eq:barg}, $g(x) - \bar{g}(x)$ takes the form of $h(\alpha^T  x)$   where $h\in{\cal M}_K$ is such that 
\begin{equation}\notag
h(t) =\begin{cases}
0&\mbox{ if }|\Psi(t)|\le 2K_0\\
\Psi(t)  + 2K_0&\mbox{ if }\Psi(t)  < -2 K_0\\
\Psi(t)  - 2K_0&\mbox{ if }\Psi(t) > 2K_0.\end{cases}
\end{equation}
Hence, for $t\in[a,b]$ we can have $h(t)\neq0$ only for $t\in[a, a+ D_2 v^2]\cup[ b-D_2 v^2,b]$.
Consider $\{\alpha_1, \ldots, \alpha_N\}$ the grid providing a $\eps_\alpha$-cover for $\mathcal{S}_{d-1}$ with   $\eps_\alpha = \eps^2/K^2$ and $N \le (A \eps^{-1}_\alpha)^d$, see Lemma \ref{lem:partsphere}. Similar to the proof of Lemma \ref{lem: entropy}, with $\alpha_i$ such that $\|\alpha-\alpha_i\|\leq\eps_\alpha$, we then have
\begin{equation}\label{eq: encadg}
h(\alpha_i^T x-\eps_\alpha R)\leq g(x)-\bar g(x)\leq h(\alpha_i^T x+\eps_\alpha R)
\end{equation}
for all $x  \in \mathcal{X}$, where   $h$ is considered on the  interval $[a_i,b_i]$ where 
$$a_i=\inf\{\alpha_i^Tx-\eps_\alpha R,x\in\mc X\}\quad\mbox{and}\quad b_i=\sup\{\alpha_i^Tx+\eps_\alpha R,x\in\mc X\}.$$ 
Note that the support of $\alpha_i^TX$, $\alpha_i^TX-\eps_\alpha R$ and $\alpha_i^TX+\eps_\alpha R$ are all included in $[a_i,b_i]$, and we have $|a-a_i|\le 2\eps_\alpha R$ and $|b-b_i|\le2\eps_\alpha R$. From what precedes, 
for $t\in[a_i,b_i]$  we can have $h(t)\ne 0$ only for $t\in \mathcal{I}_{i,1}\cup\mathcal{I}_{i,2}$
where
$$\mathcal{I}_{i,1}   =  \Big[  a_{i}, a_i  +D_2v^2+2\eps_\alpha R \Big]  \quad\mbox{and}\quad   \mathcal{I}_{i,2}=  [b_{i} - D_2v^2-2\eps_\alpha R, b_{i}  \Big ]$$
have length at most $2D_2v^2$ under the assumption that $Kv>\eps K_0\sqrt{2R\underline q}$.
Hence, we only need  to construct brackets for the class of monotone functions on  $[a_i,b_i]$ that are bounded by $K$ and constant equal to zero outside $\mathcal{I}_{i,1}\cup\mathcal{I}_{i,2}$.   This can be done by using Lemma \ref{lem:entropyMono} with $Q$ denoting the uniform distribution on  $\mathcal{I}_{i,1} \cup
\mathcal{I}_{i,2} $: it follows from that lemma that with $N_i \le \exp(2A\sqrt{D_2} Kv/ \eps)  $, we can find brackets $(h_j^L,h_j^U),j=1,\dots,N_i$ such that every function in the class belongs to $[h_j^L, h_j^U]$ for some $j$, and
\begin{eqnarray}\label{eq: bracketh}\notag
\int_{\RR}(h_j^L(t)-h_j^U(t))^2dt&\leq &
4D_2v^2\int_{\mathcal{I}_{i,1}\cup \mathcal{I}_{i,2}}(h_j^L(t)-h_j^U(t))^2dQ(t)\\ 
&\leq&\eps^2
\end{eqnarray}
for all $j$.
Note that we have omitted writing the dependence on $i$ for the functions in the brackets. 

Let $j\in\{1,\dots,N_i\}$ such that
$h^L_{j}\le h\le h^U_{j}$
on $[a_i,b_i]$.
By \eqref{eq: encadg} we have
\begin{eqnarray*}
b^L(x) \equiv  \ \ h_j^L(\alpha_i^T x-\eps_\alpha R)\leq g(x)-\bar g(x)\leq h_j^U(\alpha_i^T x+\eps_\alpha R) \ \  \equiv b^U(x)
\end{eqnarray*}
for all $x\in\mc X$ and it remains to compute the size of the obtained brackets.
By the Minkowski inequality, with $Q_i$ the distribution of $\alpha^T_iX$ we have 
\begin{eqnarray*}
 \Vert b^U-b^L \Vert_{\QQ}
&=&\bigg(\int_{\RR}  \Big(  h^U_{j}(t + \eps_\alpha  R)  - h^L_{j}( t - \eps_\alpha  R)  \Big)^2  dQ_i(t) \bigg)^{1/2}  \\
&\leq&\bigg(\overline q\int_{a_i+\eps_\alpha  R}^{b_i-\eps_\alpha  R}  \Big(  h^U_{j}(t + \eps_\alpha  R)  - h^L_{j}( t - \eps_\alpha  R)  \Big)^2  dt \bigg)^{1/2}  \\
& \le& \bigg(\overline q\int_{a_i+\eps_\alpha  R}^{b_i-\eps_\alpha  R}  \Big(  h^U_{j}( t + \eps_\alpha  R)  - h_j^U( t - \eps_\alpha  R)  \Big)^2  dt \bigg)^{1/2}  \\
&&  \qquad + \bigg(\overline q\int_{a_i+\eps_\alpha  R}^{b_i-\eps_\alpha  R}   \Big(  h^U_{j}(t - \eps_\alpha  R)  - h_j^L(t - \eps_\alpha  R)  \Big)^2  dt \bigg)^{1/2}  \\
\end{eqnarray*}
Since $h_j^L$ and $h_j^U$  can be choosen monotone on $[a_i,b_i]$ and bounded in absolute value by $K$ we conclude from \eqref{eq: bracketh} that
\begin{eqnarray*}
 \Vert b^U-b^L \Vert_{\QQ} &\leq & \bigg(2K\overline q\int_{a_i}^{b_i-2\eps_\alpha  R}  \Big(  h_j^U( t + 2\eps_\alpha  R) - h^U_{j}( t)  \Big)  dt \bigg)^{1/2}  + \sqrt{\overline q}\eps\\
& \leq & \left(8\eps_\alpha RK^2 \overline q \right)^{1/2}  + \sqrt{\overline q}\eps.
\end{eqnarray*} 
Since by definition, $\eps_\alpha=\eps^2/K^2$,  this means that we  have  $C\eps$-brackets 
where $C$ depends on $\overline q$ and $R$ only. 
Hence,
\begin{eqnarray*}
H_B\Big(C\eps, \bar{ \mc G}_{Kv},  \Vert \cdot \Vert_{\QQ}  \Big)  & \leq  &     \log(N_i) + \log(N)
\ \ \le \ \   \frac{2A\sqrt{D_2} K v}{\eps}   + d \log \frac{A  K^2}{\eps^2} 
\end{eqnarray*}
and the first assertion of the lemma follows.

To prove the second assertion, recall that  
$(f-\bar f)(x,y)=y(g(x)-\bar g(x))-\frac 12(g^2(x)-\bar g^2(x)).$
We need  then to build brackets for the functions $(x,y)\mapsto y(g(x)-\bar g(x))$, and those for the functions $x \mapsto g^2(x)-\bar g^2(x).$ 
The construction of the brackets goes along the same line as the construction of the brackets for the classes $\mathcal{G}_1$ and $\mathcal{G}_2$  in the proof of Lemma \ref{lem: entropy} above, where for the latter class, we use the fact that  all functions in the class take the form $h(\alpha^Tx)$ where $h$ is either monotone or U-shaped, and vanishes when $|g(\alpha^Tx)|\le 2K_0$, which implies that the function vanishes except on at most two intervals of maximal length $2D_2v^2$.

Hence, we can find  $A_1>0$ and $A_2 > 0$ depending on $K_0$, $\overline q$, $\underline q$, $R$ and $a_0$  such that  
\begin{eqnarray*}
H_B\Big(\eps, \bar{ \mc F}_{Kv},  \|\,.\,\|_{\PP} \Big)  \le  \frac{ A_1 K^2v}{\eps}  + d \log\frac{A_1K^4}{\eps^2} 
\end{eqnarray*}
for $K>1\vee\eps$ such that $K^2v>A_2\eps$. 
This 
completes the proof of Lemma \ref{lem: entropyg-barg}. 
\qed

\subsection{Proof of Lemma \ref{Bernsteinentropy}}  \label{proof: Bernstein}
We start by noting that entropy bound on the class $\widetilde{\mc F}_{Kv}$ is smaller than the entropy bound for the class $\widetilde{\mc F}_{K}$ as a consequence of inclusion of the former class in the latter. We will now show that the upper bound with respect to the Bernstein norm for the class $\widetilde{\mc F}_{Kv}$ is of the claimed order. 
Let $N_1  =  N_B(\eps , \mathcal{G}_K, \|\,.\,\|_{\QQ}   )  $ and $N_2   =  N_B(\eps , \mathcal{G}_2, \|\,.\,\|_{\QQ} )$, where
$\mathcal{G}_2$ is the class of functions $\{x \mapsto  g^2(x), g \in \mathcal{G}_{K} \}$.
Consider brackets  $[g^L_j, g^U_j], \ j=1, \ldots, N_1$  covering $\mathcal{G}_K$ and $[s^L_i, s^U_i],  \ i =1, \ldots,    N_2$ covering $\mathcal{G}_2$.
Note that $g^L_i$ and $g^U_i$ can be always taken to be bounded by $K$, because otherwise we can take instead $g^L_i \vee (-K) $ and $g^U_i \wedge K$.  The same thing holds for $s^L_i$ and $s^U_i$ which can be taken to be bounded by $K^2$. Let $f  \in \mathcal{F}_{Kv}$  and $\tilde{C} > 0$ a fixed constant to be chosen later.  Then, there exists $(i, j ) \in \{1, \ldots, N_1 \}  \times \{1, \ldots, N_2  \}$ such that $f^L_{i, j}\le f\le f^U_{i, j}$ and 
\begin{eqnarray*}
f^L_{i, j}(x, y)=
\begin{cases}
 y  g^L_j(x)  - \frac{1}{2}  s^U_i(x), \ \  \  \textrm{if $y \ge 0$}  \\
 y  g^U_j(x)  - \frac{1}{2}  s^U_i(x), \ \  \  \textrm{if $y < 0$}
\end{cases}\mbox{\ ;\ }
f^U_{i, j}(x, y)=
\begin{cases}
 y  g^U_j(x)  - \frac{1}{2}  s^L_i(x), \ \  \  \textrm{if $y \ge 0$}  \\
 y  g^L_j(x)  - \frac{1}{2}  s^L_i(x), \ \  \  \textrm{if $y < 0$}.
\end{cases}
\end{eqnarray*}
Now note that for a given function $h$  such that $h^k$ is $\PP$-integrable for all $k \ge 2$, we can write  $\Vert h \Vert^2_{B, \PP} =  2 \sum_{k=2}^\infty (\int  \vert h \vert^k  d\PP )/k!$.  Hence,  by convexity of $x\mapsto x^k$ for $k\ge 2$, we have 
\begin{eqnarray*}
\left \Vert \frac{f^U_{i, j}  -   f^L_{i, j}}{\tilde C} \right\Vert^2_{B, \PP} \le 2 \sum_{k=2}^\infty \frac{2^{k-1}}{k!  \tilde{C}^k }   \int_{\mathcal{X}\times\RR}  \bigg\{\vert  y  \vert^k \Big|g^U_j(x)   -  g^L_j(x)\Big|^k   +  \frac{1}{2^k}  \Big| s^U_i(x)  -  s^L_i(x) \Big|^k   \bigg \}  d \mathbb{P}(x, y)  .
\end{eqnarray*}
Integrating first with respect to $y$ using the assumption (A3) yields that the right-hand side is at most
\begin{eqnarray*}
&&2  \sum_{k=2}^\infty \frac{2^{k-1}}{k!  \tilde{C}^k }   \bigg\{  a_0  M^{k-2}  k!  \times (2K)^{k-2}  \int_{\mathcal{X}} |g^U_j   -  g^L_j|^2 d\mathbb{Q}  +  \frac{(2K^{2})^{(k-2)}}{2^k} \int_{\mathcal{X}}  |s^U_j   -  s^L_j|^2 d\mathbb{Q} \bigg \}  \\
&& \le  2  \sum_{k=2}^\infty \frac{2^{k-1}}{k!  \tilde{C}^k }   \bigg\{  a_0  M^{k-2}  k!  \times (2K)^{k-2}  \eps^2 + \frac{(2K^{2})^{(k-2)}}{2^k }  \eps^2  \bigg \}. 
\end{eqnarray*}
Hence,
\begin{eqnarray*}
\Big \Vert (f^U_{i, j}  -   f^L_{i, j}) \tilde{C}^{-1}  \Big \Vert^2_{B, \PP} 
& \le& \frac{2 }{\tilde{C}^2} \Bigg (2 a_0  \sum_{k=2}^\infty \left(\frac{4 M K}{\tilde{C}}\right)^{k-2}    +   \frac{1}{4}\sum_{k=2}^\infty \left(\frac{2K^2}{\tilde{C}}\right)^{k-2}  \frac{1}{(k-2)!}  \Bigg ) \  \eps^2,
\end{eqnarray*}
using  the fact that $k! \ge 2 (k-2)! $ for all $k \ge 2$. We conclude that
\begin{eqnarray*}
\Big \Vert (f^U_{i, j}  -   f^L_{i, j}) \tilde{C}^{-1}  \Big \Vert^2_{B, \PP} & \le& 
\frac{2 }{\tilde{C}^2}  (2 a_0 \vee 1/4) \ \Big(  \frac{1}{1-4MK/\tilde{C}}  +  e^{2K^2/\tilde{C}}  \Big)  \  \eps^2.
\end{eqnarray*} 
Since $K\ge 2$, the choice  $\tilde{C}  = 4 M K^2$ yields
\begin{eqnarray}\label{eq: B}\notag
 \Vert (f^U_{i, j}  -   f^L_{i, j}) \tilde{C}^{-1} \Vert^2_{B, \PP}   &\le   &  \frac{2}{\tilde{C}^2}  \Big(2 a_0 \vee 1/4\Big)   \ \Big(  \frac{K}{K-1}  +  e^{(2M)^{-1}}  \Big)  \eps^2    \\ 
& \le &  
B^2 K^{-4}\eps^2
\end{eqnarray}
where $B$ depends on $ a_0$ and $M$ only. Using \eqref{eq: entG2bis} and the first assertion of Lemma \ref{lem: entropy}, this means that there exists a universal constant $A_2$ such that
\begin{eqnarray}\label{eq: ??}
H_B \Big(\frac{B\eps}{K^2}, \widetilde{ \mc F}_{Kv}, \|\,.\,\|_{B,\PP} \Big)&\leq& \log N_1+\log N_2  \ \ \leq \ \ \frac{A_2  K^2d(1+\sqrt{\overline qR})}{\eps}.
\end{eqnarray}
This in turn implies that 
\begin{eqnarray*}
H_B \Big(\eps, \widetilde{ \mc F}_{Kv}, \|\,.\,\|_{B,\PP} \Big)& \leq&   \frac{A_2 B K^2 d(1+\sqrt{\overline qR})}{\eps K^2}  \ \  = \ \ \frac{A_1 d(1+\sqrt{\overline qR})}{\eps}
\end{eqnarray*}
as claimed in the statement of the lemma. 


To show the second claim, we will use again the series expansion of the Bernstein norm. Similar as above, using that $g_0$ is bounded  by $K_0\le K$, and that  for arbitrary $\tilde{f}  = (f - f_0 ) \tilde{C}^{-1} \in \widetilde{\mathcal{F}}_{Kv}$, the corresponding $g \in \mathcal{G}_{Kv}$ satisfies \eqref{eq: defv}, we obtain 
\begin{eqnarray}\label{eq: gtoBernsteinf}\notag
\Vert \tilde{f}  \Vert^2_{B, \PP}
& \le  &    \sum_{k=2}^\infty \frac{2^{k}}{k!  \tilde{C}^k }   \bigg \{  a_0  M^{k-2}  k!  (2K)^{k-2} \int_{\mathcal{X}}  (g  -  g_0)^2 d\mathbb{Q}  +  \frac{(2K^{2})^{(k-2)}}{2^k} \int_{\mathcal{X}}  (g^2   -  g^2_0)^2 d\mathbb{Q}\bigg \}  \\ \notag
& \le  &    \sum_{k=2}^\infty \frac{2^{k}}{k!  \tilde{C}^k }   \bigg\{  a_0  M^{k-2}  k!   (2K)^{k-2} \int_{\mathcal{X}} (g   -  g_0)^2 d\mathbb{Q} +\frac{2(2K^{2})^{k-1}}{2^k} \int_{\mathcal{X}}  (g  -  g_0)^2 d\mathbb{Q} \bigg \}  \\ \notag
&  \le&\bigg( \frac{4 a_0}{\tilde{C}^2}  \sum_{k=2}^\infty \left(\frac{4MK}{\tilde{C}}\right)^{k-2}   +  \frac{(4K)^2}{\tilde{C}^2} \sum_{k=2}^\infty \left(\frac{4K^2}{\tilde{C}}  \right)^{k-2}  \frac{1}{k!}  \bigg)  v^2  \\
& \le   & \bigg(  \frac{a_0}{2 M^2 K^4 }  + \frac{1}{4M^2K^2}  e^{1/M}  \bigg)   v^2, 
\end{eqnarray}
using that $K \ge 2$. The second claim follows.

Using the same arguments as above in combination with the entropy bound for $\bar{ \mc F}_{Kv}$  obtained in Lemma \ref{lem: entropyg-barg} 
we can show that 
\begin{eqnarray*}
H_B \Big(\eps, \widetilde{\bar{ \mc F}}_{Kv}, \|\,.\,\|_{B,\PP} \Big)\leq   \frac{A_1 v}{\eps} + d \log\left( \frac{A_1 }{\eps^2} \right) 
\end{eqnarray*}
at  the cost of increasing the constant $A_1$.  To show the second assertion for the elements of $\widetilde{\bar{ \mc F}}_{Kv}$, we can use again the same arguments as for the class  $\widetilde{\mathcal{F}}_{Kv}$. Indeed, the condition  $K \ge 2K_0 \vee 2$ implies that  $\max(\vert g \vert, \vert \bar{g}\vert ) \le K$  since $\vert \bar{g} \vert \le 2K_0$. Moreover, with $\tilde{C} = 4MK^2$ we get for any element $\tilde{f}  \in \widetilde{\bar{ \mc F}}_{Kv}$ that
\begin{eqnarray}\notag \label{eq: majBern}
\Vert \tilde{f} \Vert^2_{B, \PP}  
& \le  &  \left(  \frac{a_0}{32 M^2 }  + \frac{1}{16M^2}  e^{1/M}  \right)  \int_{\mathcal{X}} \left(g(x)   -  \bar{g}(x)\right)^2 d\QQ(x) \\
& \le & 4 \left(  \frac{a_0}{32M^2 }  + \frac{1}{16M^2}  e^{1/M}  \right)  v^2
\end{eqnarray}
where in the last line we used the fact that by convexity of $x \mapsto x^2$,
\begin{eqnarray}\label{ineqg-barg}
\int_{\mathcal{X}} \left(g(x)   -  \bar{g}(x)\right)^2 d\QQ(x)  &\le  &  2 \int_{\mathcal{X}} ( g(x)  - g_0(x))^2 d\QQ(x)  + 2 \int_{\mathcal{X}} ( \bar{g}(x)  - g_0(x))^2 d\QQ(x), \notag\\
&  \le  &  4 \int_{\mathcal{X}} ( g(x)  - g_0(x))^2 d\QQ(x),
\end{eqnarray}
as $ ( \bar{g}  - g_0)^2   \le( g  - g_0)^2 $ by  definition of $\bar{g}$.  This completes the proof of Lemma \ref{Bernsteinentropy}.  \hfill $\Box$

\subsection{Proof of Theorem \ref{theo: ratelog}}\label{sec: proofRate}

In the sequel, we assume that  the assumptions of the theorem hold.   Below, we give  a uniform bound for the centered process $\MM_n-\MM$, with $\MM_n$ and $\MM$ as in \eqref{eq: Mn} and \eqref{eq: M} respectively. In the sequel, the notation $\lesssim$ means ``is bounded up to an absolute constant". Moreover, the capital $E$ denotes outer expectation in cases when we consider expectation of a random variable which we have not proved to be measurable.

\begin{prop}\label{prop:centproc}
Let $K  > 2\vee(2K_0)$. 
Then, for all $v \in (0, 2 K]$, there exists $A>0$ that depends only on $a_0$ and $M$ such that 
\begin{eqnarray}\label{eq: prop61}
\sqrt nE\bigg[\sup_{g \in \mathcal{G}_{Kv}} \Big \vert (\mathbb{M}_n - \mathbb{M})g-(\mathbb{M}_n - \mathbb{M})g_0  \Big \vert   \bigg] \leq A  d(1+\sqrt{\overline q R})\phi_n(v) 
\end{eqnarray}
where 
$
\phi_n(v)   =    v^{1/2}  K^{5/2} (1 +  K^{1/2}v^{-3/2}n^{-1/2} ).
$
\end{prop}

\begin{proof} \  
Define for $\eta > 0$ and fixed $v \in (0, 2K]$
\begin{eqnarray*}
J(\eta)  =  \int_0^\eta  \sqrt{1  +  H_B\Big(\eps, \widetilde{\mathcal{F}}_{Kv}, \|\,.\,\|_{B, \PP} \Big) } d\eps  
\end{eqnarray*}  
where we recall that $\|\,.\,\|_{B, \PP} $  is the Bernstein norm, and  $\widetilde{\mathcal{F}}_{Kv}$ is defined in \eqref{eq: tildeF}
with  $\tilde{C} = 4M K^2$.
By Lemma \ref{Bernsteinentropy}, there exists a constant $A_2 > 0$ depending only on $a_0$ and $M$  such that $\Vert \tilde{f} \Vert_{B, \PP}  \le A_2 v$ for all $\tilde f\in\widetilde{\mathcal{F}}_{Kv} $. It follows from Lemma 3.4.3  of \cite{vanweak}  (using the notation of that book)  that 
\begin{eqnarray*}
 E[\Vert \mathbb{G}_n \Vert_{\widetilde{\mathcal{F}}_{Kv} }]   \lesssim   J\big(A_2 v\big) \left( 1 +  \frac{J\big(A_2 v \big)}{A^2_2 v^2 \sqrt n}  \right) 
\end{eqnarray*}
where by Lemma \ref{Bernsteinentropy} and the inequality $\sqrt{u + v}  \le \sqrt{u } + \sqrt{v}$ for $u, v \ge 0$ we have that 
\begin{eqnarray*}
J(\eta)  \le \int_0^\eta  \sqrt{1 + \frac{A_1d(1+\sqrt{\overline q R})}{\eps}}  d\eps \le  \eta  + 2 (A_1d(1+\sqrt{\overline q R}))^{1/2}  \eta^{1/2}
\end{eqnarray*}
for all $\eta>0$.  Note that  $ v \in (0, 2K]$  implies that $v  \le v^{1/2}  2^{1/2}  K^{1/2}$, and hence  
$$ J\big(A_2 v \big)  \le \big(A_2  2^{1/2}  K^{1/2}  +  2 (A_1 d(1+\sqrt{\overline q R}))^{1/2} A^{1/2}_2 \big) v^{1/2}  \le A_3 (d(1+\sqrt{\overline q R}))^{1/2}K^{1/2} v^{1/2}$$
 using that $K \ge 1$, where $A_3$ depends only on $a_0$ and $M$.  Hence, by  definition of  $\widetilde{\mathcal{F}}_{Kv}$, which has the same entropy has the class ${\mc F}_{Kv}-f_0=\{f-f_0,f\in{\mc F}_{Kv}\}$, we obtain
\begin{eqnarray*}
 \sqrt nE  \bigg[\sup_{g \in \mathcal{G}_{Kv}} \Big \vert (\mathbb{M}_n - \mathbb{M})(g)-(\mathbb{M}_n - \mathbb{M})(g_0)  \Big \vert   \bigg]    &= & E  [\Vert \mathbb{G}_n \Vert_{\mathcal{F}_{Kv}-f_0} ]  \\
& = &  4MK^2 E  [\Vert \mathbb{G}_n \Vert_{\widetilde{\mathcal{F}}_{Kv}-f_0}]  \lesssim     \phi_n(v)
\end{eqnarray*}
which completes the proof of Proposition \ref{prop:centproc}.
\end{proof}

Now we are ready to give the proof of Theorem \ref{theo: ratelog}.

\begin{proof}[Proof of Theorem  \ref{theo: ratelog}]  \ 
In the sequel, we consider $K=C\log n$ for some $C>0$ that does not depend on $n$, and $v\in(0,2K]$.  It follows from Proposition \ref{prop:centproc} that for all $n$ sufficiently large, we have \eqref{eq: prop61}
where
$$\phi_n(v)=(\log n)^{5/2}\sqrt{v}\left(1+(\log n)^{1/2}v^{-3/2}n^{-1/2}\right)$$
and $A$ depends only on $a_0$, $M$ and $C$.
Since with $D$ taken from \eqref{eq: defD}, we have $D(g,g_0)\leq \|g\|_\infty+\|g_0\|_\infty\le 2K$ for sufficiently large $n$ and all $g\in\mc G_{K}$, the above inequality holds for all $v>0$.
Furthermore, $\widehat g_n$ maximizes $\MM_n g$  over the set of all functions $g$ of the form $g(x)=\Psi(\alpha^Tx), x \in \mathcal{X}$ with $\alpha\in \mathcal S_{d-1}$ and $\Psi$ a non-decreasing function on $\RR$, and it follows from Lemma \ref{lem: LSElog} that with arbitrarily large probablity by choice of $C$, $\widehat g_n$ maximizes $\MM_n g$ over the restricted set $\mc G_{K}$. Hence, we can use Lemma \ref{lem: basic} and Proposition \ref{prop:centproc} above, together with Theorem 3.2.5 in \cite{vanweak} with $\alpha=1/2$ and  $r_n\sim n^{1/3}(\log n)^{-5/3}$, to conclude that $D(\hat g_n,g_0)=O_p^*(n^{-1/3}(\log n)^{5/3}),$ which completes the proof of Theorem \ref{theo: ratelog}.
\end{proof}

\subsection{Proof of Theorem \ref{theo: rateL2}} \label{proof: ratewithoutlog}
Assuming that   (A1)-(A4) hold, we give  a second uniform bound for  $\MM_n-\MM$. The bound is sharper than the one obtained in Proposition \ref{prop:centproc} for the case where all functions $g$ in the considered class of functions satisfy \eqref{eq: defv} for some $v\leq (\log n)^2  n^{-1/3}$. As before, the notation $\lesssim$ means ``is bounded up to an absolute constant".

\begin{prop}\label{lem:centproc2}
Let  $K = C \log n$ for some fixed $C>0$, $v  \in(0, (\log n)^2  n^{-1/3}]$ and $
\phi_n(v)    =    v^{1/2}  \left(1 +  v^{-3/2} n^{-1/2}  \right)
$. 
Then for $n $ large enough we have that 
\begin{eqnarray*}
\sqrt n E   \bigg[\sup_{g \in \mathcal{G}_{Kv} } \Big \vert (\mathbb{M}_n - \mathbb{M})g-(\mathbb{M}_n - \mathbb{M})g_0  \Big \vert   \bigg] \leq A \phi_n(v)
\end{eqnarray*}
where $A$ depends only on $R,a_0,M,\overline q,\underline q$ and $K_0$.
\end{prop}


\begin{proof}
 Assume $n$ large enough so that $K\geq (2K_0)\vee2$ and   use (\ref{decomp}) to write that the expectation on the left hand side of the previous display is bounded above by
\begin{eqnarray*}
\sqrt n E   \bigg[\sup_{g \in \mathcal{G}_{Kv} } \Big \vert (\mathbb{M}_n - \mathbb{M})g-(\mathbb{M}_n - \mathbb{M}) \overline g  \Big \vert   \bigg]  
+ \sqrt n E   \bigg[\sup_{g \in \mathcal{G}_{Kv} } \Big \vert (\mathbb{M}_n - \mathbb{M}) \overline g-(\mathbb{M}_n - \mathbb{M})g_0  \Big \vert   \bigg].  \notag
\end{eqnarray*}
Hence, in the notation of \cite{vanweak}
\begin{eqnarray}\label{decompexp}
  \sqrt n E    \bigg[\sup_{g \in \mathcal{G}_{Kv}} \Big \vert (\mathbb{M}_n - \mathbb{M})g-(\mathbb{M}_n - \mathbb{M})g_0  \Big \vert   \bigg]  \le    E  [\Vert \mathbb{G}_n \Vert_{\mathcal{F}_1 }]  +     E  [\Vert \mathbb{G}_n \Vert_{\mathcal{F}_2 }] 
\end{eqnarray}   
where $\mathcal{F}_1  = \bar{\mc{F}}_{Kv}$ and $\mathcal{F}_2$ is the class of  functions $\bar{f}  - f_0$ such that $\bar f\in{\mc F}_{(2K_0)v}$. 
To give a bound for the first term on the right-hand side, consider $v' =  (\log n)^2 n^{-1/6} \gg\sqrt v$. It follows from Lemma \ref{Bernsteinentropy} that for all $\eps\in(0,A_2^{-1}v]$, we have
\begin{eqnarray}\label{entropybound}
H_B\Big(\eps, \widetilde{\bar{ \mc F}}_{Kv'}, \Vert \cdot \Vert_{B, \PP} \Big)&\leq&\frac{A_1 v'}{\eps} + d \log\left( \frac{A_1 }{\eps^2} \right) \ \ \le \ \  A_1 (1 + d)  \frac{v'}{\eps}
\end{eqnarray}
provided that $n$ is sufficiently large, where we used the fact that $\log(x)  \le \sqrt x$ for all $x > 0$ for the second inequality.
Since  the class $\widetilde{\mc F}_1: = \widetilde{\bar{\mc{F}}}_{Kv}$ is included in $\widetilde{\bar{\mc{F}}}_{Kv'}$, 
its $\eps$-bracketing entropy can be also bounded above by (\ref{entropybound})  for all $\eps \in (0, A_2^{-1} v]$. Using again the inequality $\sqrt{x  + y }  \le \sqrt{x} + \sqrt{y}  $ for all $x, y \ge 0$ we can write  
\begin{eqnarray*}
J_1\big( A_2 v \big)  &:=  & \int_0^{A_2 v}  \sqrt{1  + H_B\Big(\eps, \widetilde{\mc F}_1, \Vert \cdot \Vert_{B, \PP}  \Big)} d \eps  \\
& \le &  A_2 v +  2 \ \Big(A_1 (1+ d)   v'\Big)^{1/2}   (A_2 v)^{1/2} \le A_3  (v'v)^{1/2} 
\end{eqnarray*}
using that $v<v'$ and $K>1$, where $A_3 > 0$ is a constant depending on $a_0$, $M$ and $d$. Lemma \ref{Bernsteinentropy} implies that $\Vert \tilde{f} \Vert_{B, \PP} \le A_2 v $ for all $\tilde{f}  \in \widetilde{\mc F}_1$. Invoking Lemma 3.4.3 of \cite{vanweak} allows us to write that  
\begin{eqnarray*}
 E[\Vert \mathbb{G}_n \Vert_{\widetilde{\mathcal{F}}_1 }]  & \lesssim & J_1\big( A_2 v \big)  \left(1 +  \frac{J_1\big(A_2 v\big)}{A^2_2  v^2 \sqrt n}  \right)   \ \ \leq \ \   A_3   (v'v)^{1/2}  \left(  1 +  \frac{ (v')^{1/2}}{v^{3/2} \sqrt{n}}  \right) 
\end{eqnarray*}
at the cost of increasing $A_3$. Now, using the definition of $\widetilde{\mathcal{F}}_1 $, we have that 
\begin{eqnarray*}
 E[\Vert \mathbb{G}_n \Vert_{\mathcal{F}_1 }]  &= &  4MK^2  E[\Vert \mathbb{G}_n \Vert_{\widetilde{\mathcal{F}}_1 }]   \leq A_3  v^{1/2}  \left(1 +  \frac{1}{v^{3/2}\sqrt{n} }  \right)  
\end{eqnarray*}
at the cost of increasing $A_3$.
This gives a bound for the first term on the right hand side of \eqref{decompexp}.

To deal with the second term, we  apply Lemma \ref{Bernsteinentropy} to the class $\widetilde{\mc F}_2   = \widetilde{\mathcal{F}}_{(2K_0)v}$ with $K= 2K_0$. Here, $\tilde{C} = 4 M K^2_0$  is independent of $n$, and  
$J_2(A_2 v) \le A_3 v^{1/2}$ for some $A_3>0$ that does not depend on $n$, where $J_2$ is defined in the same manner as $J_1$ with $\widetilde{\mc F}_1$ replaced by $\widetilde{\mc F}_2$. By Lemma 3.4.3 of \cite{vanweak}, we have 
\begin{eqnarray*}
 E[\Vert \mathbb{G}_n \Vert_{\mathcal{F}_2} ] & = &  16 M K^2_0  E[\Vert \mathbb{G}_n \Vert_{\widetilde{\mc F}_2}]  \leq A_3     v^{1/2}  \left(  1 +  \frac{1}{v^{3/2} \sqrt n} \right)
 \end{eqnarray*}
at the cost of increasing $A_3$.
Combining the calculations developed for both classes together with (\ref{decompexp})  gives the claimed form of the entropy bound.  
\end{proof}

\begin{proof}[Proof of Theorem \ref{theo: rateL2}]  
Theorem \ref{theo: ratelog} implies that with a probability that can be made arbitrarily large, the LSE $\widehat g_n$ belongs to $\mathcal{G}_{Kv}$ with $K=C\log n$ and $v=(\log n)^2n^{-1/3}$ for some $C>0$ that does not depend on $n$. The result follows now from Theorem 3.2.5 of \cite{vanweak}  with $\alpha=1/2$ and $r_n\sim n^{1/3}$. 
\end{proof}

\section*{Acknowledgements} Part of this work was done while the third author was at the University Paris Nanterre as an invited visiting researcher.  She thanks Paris Nanterre for the support. The contribution of the second author has been conducted as part of the project Labex MME-DII (ANR11-LBX-0023-01).

\begin{center}
\huge{Supplement to``Least squares estimation in the monotone single index model"}
\end{center}
\section{Proofs and other results from Section \ref{sec: LSE}}\label{sec: proofsSection2}

\subsection{Proof of Theorem \ref{theo: caractLSE}}
We need some notation.  Fix $\alpha\in{\cal S}_{d-1}$ and let $m_\alpha$ be the (random) number of distinct values among $\alpha^T X_1, \ldots, \alpha^T X_n$. Let 
  $Z _1 < \ldots < Z _{m_{\alpha}}$ 
be the corresponding ordered statistics. 
Define

$$n _{k}=\sum_{i=1}^n\mathbb{I}_{\alpha^T X_{i}=Z_{k} }\quad \mbox{and}\quad  
 t _{k}=\frac1{n_{k} }\sum_{i=1}^nY_{i} \mathbb{I}_{\alpha^T X_{i}=Z_{k} },$$
for all $k=1,\dots,m_\alpha$.
Note that to alleviate the notation, we do not make it explicit that $Z_k,n_k,t_k$ depend on $\alpha$.
For all $\Psi\in\cal M$ we  then have
\begin{equation}\label{eq: hnPsi}
h_n(\Psi, \alpha)=\sum_{k=1}^{m_\alpha}n_k \left \{ t_k  - \Psi(Z_k ) \right \}^2
+\sum_{i=1}^nY_{i}^2-\sum_{k=1}^{m_\alpha}n_k  (t_k )^2.
\end{equation}
This means that minimizing $h_n(\Psi,\alpha)$ with respect to $\Psi$ amounts to minimizing
$$\Psi\mapsto \sum_{k=1}^{m_\alpha}n_k \left \{ t_k  - \Psi(Z_k ) \right \}^2$$
over $\cal M,$ which in turn amounts to minimizing
$\sum_{k=1}^{m_{\alpha}}n_{k} \left \{ t_{k}   - \eta_{k} \right \}^2$
over the set of all real numbers $\eta_{1}\leq\dots\leq \eta_{m_\alpha}$, where we set $\eta_{k}=\Psi(Z_{k} )$. It follows from Theorem 1.1 in \cite{barlowstatistical} that the minimum is achieved at a unique $(\eta_{1},\dots, \eta_{m_\alpha})$, which completes the proof of Theorem \ref{theo: caractLSE}.
\qed

 \subsection{Proof of Proposition \ref{prop: alphaexists}}\label{sec: exists}
For arbitrary $(i,j)$ with $X_i\neq X_j$, the total space $\RR^d$ is separated into three disjoint parts: 
(1)  the hyperplane $H_{ij}$ of all vectors $\alpha\in\RR^d$  such that $\alpha^T(X_{i}-X_{j})=0$,
(2) the open half-space $S_{ij}^+$ of all $\alpha\in \RR^d$ such that $\alpha^TX_{i}>\alpha^TX_{j}$,
and (3) the open half-space $S_{ij} ^-$ of all $\alpha\in \RR^d$ such that $\alpha^TX_{i}<\alpha^TX_{j}$. We call a subset $\mc R\subset\RR^d$ a {\it maximal region} if $\mc R\cap H_{ij}$ is empty for all $i\neq j$ such that $X_i\neq X_j$, but with $\overline {\mc R}$ the closure of $\mc R$, $\overline{\mc  R}\backslash \mc R$ is included in the (finite) union of all hyperplanes $H_{ij}$.   Hence, a maximal region is an intersection of half-spaces.
We will prove that for arbitrary $i\neq j$ such that $X_i\neq X_j$, and $\alpha \in H_{ij}\cap{\cal S}_{d-1}$, we can find 
$\alpha_k\in {\cal S}^X$ with
\begin{equation}\label{eq: MaxRegion}
\hat h_n(\alpha)\geq \hat h_n(\alpha_k).
\end{equation}

 We can find a maximal region $\mc R$ and a sequence $(\alpha_{k})_{k\in\NN}$ such that $\alpha_k\in \mc R\cap{\cal S}_{d-1}$ for all $k$ and
$\alpha_{k}\to\alpha$ as $k\to\infty$. With $x_1,\dots,x_m$ (where $m\in\NN$ is random) the distinct values of $X_1,\dots,X_n$,  there exists a (unique) permutation $\pi_{\mc R}$ such that
$$\alpha_{k}^Tx_{\pi_{\mc R}(1)}<\alpha_{k}^Tx_{\pi_{\mc R}(2)}<\dots<\alpha_{k}^Tx_{\pi_{\mc R}(m)}$$
for all $k$.
Denote by $l_{1}$ and $l_{2}$ the indices such that $X_i=x_{\pi_{\mc R}(l_{1})}$ and $X_j=x_{\pi_{\mc R}(l_{2})}.$
We have $\alpha^TX_i= \alpha^TX_j$ since $\alpha\in H_{ij}$ and therefore, letting $k\to\infty$ yields
\begin{equation}\label{eq: order}
\alpha^Tx_{\pi_{\mc R}(1)}\leq \dots\leq\alpha^Tx_{\pi_{\mc R}(l_{1})}=\dots=\alpha^Tx_{\pi_{\mc R}(l_{2})} \leq \dots\leq \alpha^Tx_{\pi_{\mc R}(m)}.
\end{equation}
Define
$$\tilde n_{k}=\sum_{i=1}^n\mathbb{I}_{X_{i}=x_{k}}\quad \mbox{and}\quad  
{\tilde  y_{k}}=\frac1{\tilde  n_{k}}\sum_{i=1}^nY_{i} \mathbb{I}_{X_{i}=x_{k}},$$
for all $k=1,\dots,m$.
 Rearranging the terms in \eqref{eq: hnPsi} yields
\begin{equation}\label{eq: h_nalpha} 
h_n(\Psi, \alpha_k)=\sum_{k=1}^{m}\tilde  n_{\pi_{\mc R}(k)}\left \{ {\tilde  y}_{\pi_{\mc R}(k)} - \Psi( \alpha_k^Tx_{\pi_{\mc R}(k)}) \right \}^2+\sum_{i=1}^nY_{i}^2-\sum_{k=1}^{m}\tilde n_k ({\tilde y_k})^2
\end{equation}
for all $\Psi\in\cal M$. Minimizing $h_{n}(\Psi,\alpha_{k})$ over $\Psi\in{\cal M},$ we get
\begin{eqnarray}\label{supalphak}
\hat h_n(\alpha_{k})
=\inf_{\eta_1\leq\dots\leq\eta_{m}} \sum_{k=1}^{m}\tilde n_{\pi_{\mc R}(k)}\left \{{\tilde y}_{\pi_{\mc R}(k)}- \eta_k \right \}^2
+\sum_{i=1}^nY_{i}^2-\sum_{k=1}^{m}\tilde n_k ({\tilde y}_k)^2
\end{eqnarray}
for all $k$ whereas because of \eqref{eq: order}, 
\begin{eqnarray*}
\hat h_n(\alpha)=\inf_{\eta_1\leq\dots\leq \eta_{l_{1}}=\dots=\eta_{l_{2}} \leq\dots \leq\eta_{m}}\sum_{k=1}^{m}\tilde n_{\pi_{\mc R}(k)}\left \{{\tilde y}_{\pi_{\mc R}(k)}- \eta_k \right \}^2
+\sum_{i=1}^nY_{i}^2-\sum_{k=1}^{m}\tilde n_k ({\tilde y}_k)^2.
\end{eqnarray*}
The above infimum is taken over a restricted set as compared to \eqref{supalphak}, so we conclude that \eqref{eq: MaxRegion} holds
for all $k$. Hence, the infimum of $\hat h_n$ over ${\cal S}_{d-1}$ is equal to the infimum of $\hat h_n$ over  the intersection of ${\cal S}_{d-1}$ with the (finite) union of all possible maximal regions in $\RR^d$. This is precisely ${\cal S}^X$, which completes the proof of the first claim.

Now, it follows from \eqref{supalphak} that for all $\alpha_k\in  {\cal S}^X$, $\hat h_n(\alpha_{k})$  depends only on the ordering $\pi_{\cal R}$ induced by $\alpha_k$. As the number of inducible orderings is finite, minimizing $\alpha\mapsto \hat h_n(\alpha)$ over ${\cal S}^X$ amounts at minimizing a function over a finite set, so that the minimum is achieved. This completes the proof of
 Proposition \ref{prop: alphaexists}. \qed

\subsection{Proof of Theorem \ref{theo: caractLSE2}} Similar to \eqref{supalphak}, for all $\pi\in{\cal P}^X$ and $\alpha\in{\cal S}^X$ connected with \eqref{eq: inducepi}, we have
$$\hat h_n(\alpha)
=\inf_{\eta_1\leq\dots\leq\eta_{m}} \sum_{k=1}^{m}\tilde n_{\pi(k)}\left \{{\tilde y}_{\pi(k)}- \eta_k \right \}^2
+\sum_{i=1}^nY_{i}^2-\sum_{k=1}^{m}\tilde n_k ({\tilde y}_k)^2.$$
It follows from Theorem 1.1 in \cite{barlowstatistical} that the minimum is achieved at the unique $(\eta_1,\dots,\eta_m)=(d_1^{\pi},\dots,d_m^{\pi})$ whence
\begin{eqnarray*}
\hat h_n(\alpha)
&=&\tilde h_n(\pi)
+\sum_{i=1}^nY_{i}^2-\sum_{k=1}^{m}\tilde n_k ({\tilde y}_k)^2.
\end{eqnarray*}
If $(\hat\Psi_n,\hat\alpha_n)$ satisfies the conditions of the theorem, we conclude that for all $\alpha\in{\cal S}^X$ and $\pi\in{\cal P}^X$ that satisfies \eqref{eq: inducepi},
\begin{eqnarray*}
h_n(\hat\Psi_n,\hat\alpha_n)&=& \tilde h_n(\hat\pi_n)
+\sum_{i=1}^nY_{i}^2-\sum_{k=1}^{m}\tilde n_k ({\tilde y}_k)^2\\
&\leq& \tilde h_n(\pi)
+\sum_{i=1}^nY_{i}^2-\sum_{k=1}^{m}\tilde n_k ({\tilde y}_k)^2\\
&=&\hat h_n(\alpha)=\inf_{\Psi\in{\cal M}} h_n(\Psi,\alpha)
\end{eqnarray*}
which completes the proof.
\qed

\subsection{Algorithm to compute the LSE when $d=2$} \label{sec:exact2D}
Below, we give an algorithm to compute the LSE exactly for the special case when $d=2.$   Recall that the number of orderings grows like $m^{2(d-1)}.$  Thus, even if the following approach could be extended to $d>3,$ this would probably not the best approach for computational efficiency.  For $d=3$ and $m=100,$ there are over 24 million possible orderings.  


\bigskip
\noindent\textbf{Algorithm:}
\begin{enumerate}
\item Enumerate all pairs of the covariates $\{X_1, \ldots, X_n\}$ and calculate the unit orthogonal vector to the difference vector; remove all duplicates (resulting in $v_1 \ldots, v_K$, say).  Convert the unit vectors from Cartesian coordinates to polar coordinates, with $\beta_i$ denoting the angle from the positive horizontal axis to vector $v_i$.  This results in  $\beta_1, \ldots, \beta_K .$  Place these in order, and calculate the midpoint of each difference: $\alpha_1, \ldots, \alpha_K .$ 
\item For each $\alpha_i$, compute $\widehat h_i = \widehat h_n(\alpha_i)$ as in \eqref{line:alg2D}.   This can be done using, for example, the PAVA.  
\item Return $\widehat \alpha_n$  corresponding to the minimizer of $\widehat h_i, 1 \leq i \leq K.$   
\end{enumerate}

\section{Additional proofs}\label{sec: appendix}

\subsection{Proof of Theorem \ref{MainTheo}}  \label{sec:ProofMainTheo}
By definition of $\widehat{\alpha}_n$, we necessarily have that 
\begin{eqnarray}\label{eq: sigma}
\widehat{\Sigma}_n\widehat{\alpha}_n 
                             & = & \frac{1}{n}  \sum_{i=1}^n Y_i ( X_i-\bar{X}_n )  \to_p     Cov(X, Y) 
\end{eqnarray} 
as $n\to\infty$, where $\widehat{\Sigma}_n =  n^{-1}  \sum_{i=1}^n  (X_i  -  \bar{X}_n)   (X_i  -  \bar{X}_n)^T$. This implies that 
\begin{eqnarray*}
\widehat{\Sigma}_n\widehat{\alpha}_n  &\to_p  &   E[(X-\mu)  Y]\\
  & = &  E[(X-\mu)  \Psi_0(\alpha_0^T  X )   ]  \\
& = &  E \left [ \Psi_0(\alpha_0^T  X)  E[(X-\mu)  |  \alpha_0^T  X] \right ].
\end{eqnarray*} 
Now, using the known property of elliptically symmetric random variables, see e.g. \cite[page 319, comment following Condition 3.1]{MR1137117} we should have that $E[(X-\mu)  |  \alpha_0^T  X  ] $ is  linear in $ \alpha_0^T  X$ and therefore,
 $$E[(X-\mu)  |  \alpha_0^T  X  ]  = \alpha_0^T  (X-\mu)   b $$
 (using $E[X]=\mu$) where $b$ is a vector that has to satisfy
\begin{eqnarray*}
E[(X-\mu)   \alpha_0^T  (X-\mu)]  =  var(\alpha_0^T  X)  b 
\end{eqnarray*}
or equivalently (using $var(X) = \Sigma$), 
$b =  \Sigma\alpha_0(\alpha^T_0  \Sigma  \alpha_0)^{-1}$.
We conclude that 
\begin{eqnarray*}
\widehat{\Sigma}_n\widehat{\alpha}_n  &\to_p  & 
 \lambda^*\Sigma \alpha_0
\end{eqnarray*}
where $\lambda^*  =  cov\big(\Psi_0(\alpha^T_0  X), \alpha^T_0 X \big)/\alpha^T_0  \Sigma \alpha_0$.
By the law of large numbers, $\hat\Sigma_n$ converges in probability to $\Sigma$, so we obtain if $\Sigma$ is invertible that $\hat\Sigma_n$ is invertible with probabiity that tends to one, with inverse $\hat\Sigma_n^{-1}$ that converges in probability to $\Sigma^{-1}$ and therefore,
\begin{eqnarray}\label{eq: cvLE}
\widehat{\alpha}_n  \to_p   
 \lambda^*\alpha_0.
\end{eqnarray}
Now,  let $Z$ be an independent copy of $X$. Since $\Psi_0$ is  strictly increasing on an interval we have
\begin{eqnarray*}
0 < E\Big[\big(\Psi_0(\alpha^T_0  Z)  - \Psi_0(\alpha^T_0  X) \big) \big(\alpha^T_0  Z  -  \alpha^T_0  X\big)  \Big]   = 2 \ cov\big(\Psi_0(\alpha^T_0  X), \alpha^T_0 X\big),
\end{eqnarray*} whence $\lambda^*>0$. Combining \eqref{eq: cvLE} with the continuous mapping theorem then yields
$$\tilde\alpha_n=\frac{\widehat{\alpha}_n}{\|\widehat{\alpha}_n\|}  \to_p   
\frac{ \lambda^*\alpha_0}{\lambda^*\|\alpha_0\|}=\alpha_0
$$
since by assumption, $\|\alpha_0\|=1$.  This completes the proof of the first assertion.

To prove the second assertion, we write $C  =  cov(X, Y)$. Combining \eqref{eq: sigma} and \eqref{eq: cvLE} yields that with $\alpha^* =\lambda^* \alpha_0 $, we have
 $\alpha^* = \Sigma^{-1}  C
$ 
and therefore,
\begin{eqnarray*}
\sqrt{n} (\widehat{\alpha}_n  - \alpha^*)  &=  &  \sqrt n \left[ \frac{\widehat{\Sigma}^{-1}_n}{n} \sum_{i=1}^n Y_i (X_i -  \bar{X}_n)  - \alpha^*   \right]  \\
 & = & \widehat{\Sigma}^{-1}_n \sqrt n \left[\  \frac{1}{n} \sum_{i=1}^n Y_i(X_i -  \bar{X}_n)  -   C   -  \big(\widehat{\Sigma}_n \alpha^*  -  \Sigma \alpha^* \big)  \right]  \\
  & = & \widehat{\Sigma}^{-1}_n \sqrt n \left[  \frac{1}{n} \sum_{i=1}^n (X_i -  \mu) ( Y_i  - E(Y))   -   C   - \big(  \widehat{\Sigma}_n \alpha^*   -  \Sigma \alpha^*   \big)   \right] \\
&&  - \widehat{\Sigma}^{-1}_n \sqrt n   (\bar{X}_n - \mu)  (\bar{Y}_n - E(Y))  \\
& = &\frac{\widehat{\Sigma}^{-1}_n}{ \sqrt n }\sum_{i=1}^n \left[\ (X_i -  \mu) ( Y_i  - E(Y))   -    (X_i - \mu)  (X_i  -\mu)^T \alpha^*    - \left( C   -  \Sigma \alpha^*  \right)  \right]  \\
&&  - \widehat{\Sigma}^{-1}_n \sqrt n   (\bar{X}_n - \mu)  (\bar{Y}_n - E(Y))  + \widehat{\Sigma}^{-1}_n \sqrt n(\bar{X}_n - \mu)(\bar{X}_n - \mu)^T \alpha^*  \\
& = & \frac{\widehat{\Sigma}^{-1}_n}{ \sqrt n }\sum_{i=1}^n \left[\ (X_i -  \mu) ( Y_i  - E(Y)-(X_i  -\mu)^T \alpha^*)    - \left( C   -  \Sigma \alpha^*  \right)  \right]  \\
&& + o_p(1).
\end{eqnarray*}
Hence, it follows from the central limit theorem that
\begin{eqnarray*}
\sqrt{n} (\widehat{\alpha}_n  - \alpha^*)  \to_d \mathcal{N}\bigg(0, \Sigma^{-1}  \Gamma \Sigma^{-1}\bigg)
\end{eqnarray*}
where $\Gamma$ is the dispersion matrix of $(X -  \mu) ( Y - E(Y)-(X -\mu)^T\alpha^*).$
Now  consider the  functions $h(x)  =  x/\Vert x \Vert$ and $h_i(x)  = x_i /\Vert x \Vert$ for $x\in\RR^d$. Let $H$ be the gradient matrix of $h$. Then,  $H_{ij} = \partial h_i/\partial x_j  =  (\Vert x \Vert^2  - x^2_i)/\Vert x \Vert^3$ if $j=i$  and $ -x_i x_j/ \Vert x \Vert^3  $ otherwise.  Since $\|\alpha^*\|=\lambda^*$, we obtain that $H^*$, the gradient matrix evaluated at $\alpha^*$  is given by   $$H^* = \frac{1}{\vert \lambda^* \vert }  \big(I_d - \alpha_0 \alpha^T_0\big).$$  Using  the $\delta$-method we conclude from the two preceding displays that $\sqrt n(\tilde\alpha_n-\alpha_0)$ converges in distribution to a centered Gaussian distribution with dispersion matrix 
\begin{eqnarray*}
V = \frac{1}{(\lambda^*)^2}  \big(I_d - \alpha_0 \alpha^T_0\big)   \Sigma^{-1}  \Gamma \Sigma^{-1}\big  (I_d - \alpha_0 \alpha^T_0  \big),
\end{eqnarray*}
which completes the proof.
 \hfill $\Box$

\subsection{Proof of \eqref{eq: subG}}\label{sec: proofsubG}
For all $R>2\|E(X)\|$ we have
\begin{eqnarray*}
P(\|X\|>R)&\leq&P(\|X-E(X)\|>R/2)\\
&\leq&\sum_{i=1}^dP(|X_i-E(X_i)|>R/(2\sqrt d))
\end{eqnarray*}
where we denote here by $X_i$ the $i$-th entry of $X$. Since $X$ has a sub-Gaussian distribution this implies that
\begin{eqnarray*}
P(\|X\|>R)&\leq&2\sum_{i=1}^d\exp\left(-\frac{R^2}{8d\sigma^2}\right)\\
&\leq &2d\exp\left(-\frac{R^2}{8d\sigma^2}\right).
\end{eqnarray*}
Since $g_0$ is bounded, we can assume without loss of generality that its supremum norm is bounded above by $\log n$. Hence, combining the previous display with Lemma  \ref{lem: LSElog} yields  that for all $\eps>0$, there exists $A>0$ that depends only on $\eps$, $a_0$ and $M$ such that with probability larger than $1-\eps$,
\begin{eqnarray*}
\int_{\overline{\mathcal{X}}_R} \left(\widehat{g}_n(x)-g_0(x)\right)^2 d\QQ(x)
&\leq & A(\log n)^2\QQ(\overline{\mathcal{X}}_R)\\
&\leq& 2dA(\log n)^2\exp\left(-\frac{R^2}{8d\sigma^2}\right).
\end{eqnarray*}
where $\overline {\mathcal{X}}_R=\{x,\ \|x\|>R\}$. If we consider $R$ such that
$$R= 4\sigma \sqrt{d(\log d+\log n)}$$
we then obtain that
\begin{eqnarray*}
\int_{\overline{\mathcal X}_R} \left(\widehat{g}_n(x)-g_0(x)\right)^2 d\QQ(x)
&\leq& 2dA(\log n)^2\exp\left(-2(\log d+\log n)\right)\\
&\leq& 2d^{-1}A(\log n)^2n^{-2}\\
&\leq& 2A(\log n)^2n^{-2}
\end{eqnarray*}
with probability larger than $1-\eps$. In particular,
\begin{eqnarray*}
\int_{\overline{\mathcal X}_R} \left(\widehat{g}_n(x)-g_0(x)\right)^2 d\QQ(x)
=O_p(n^{-2}).
\end{eqnarray*}
Now, with the above choice of $R$, denote by $\QQ_R$ the distribution given by $\QQ_R(E)=\QQ(E\cap{\mc X}_R)/\QQ({\mc X}_R)$ for all events $E$, where ${\mathcal{X}}_R=\{x,\ \|x\|\leq R\}$. Theorem \ref{theo: ratelog} yields that for all $\eps>0$, there exists $A>0$ that depends only on $\eps$, $a_0$ and $M$ such that
\begin{eqnarray*}
\int_{{\mathcal X}_R} \left(\widehat{g}_n(x)-g_0(x)\right)^2 d\QQ_R(x)&\leq& dA\QQ({\mc X}_R) (1+\sqrt{\overline q R})n^{-1/3}(\log n)^{5/3}\\
&\leq& dA (1+\sqrt{\overline q R})n^{-1/3}(\log n)^{5/3}
\end{eqnarray*}
with probablity larger than $1-\eps$. It follows that for all $\eps>0$, there exists $A>0$ that depends only on $a_0$ and $M$ such that \eqref{eq: subG} holds with probablity larger than $1-\eps$.  \hfill $\Box$

\subsection{Proof of Proposition \ref{prop: identifiability}}   \label{proof: identifiability}

  Consider vectors $\alpha$ and $\beta$ in $\mathcal S_{d-1}$, and  non-constant functions $f\in{\cal M}$ and $h\in{\cal M}$ that satisfy 
$f(\alpha^T X) = h(\beta^T X) \ a.s.$
 and are both left-continuous (the right-continuous case can be treated likewise) with no discontinuity point at the boundary of their domain. To prove Proposition \ref{prop: identifiability}, it suffices to show that in such a case, we necessarily have $\alpha=\beta$ and $f=h$ on ${\cal C}_\alpha={\cal C}_\beta$. We prove below that we indeed have $\alpha=\beta$ and $f=h$.

By assumption we have
  $f(\alpha^Tx)=h(\beta^Tx)$ for almost all $x\in\cal X$
  in the Lebesgue sense. Using left-continuity of both $f$ and $h$, we conclude that the above equality holds for all $x$ in the interior of $\cal X$.  If we could prove that $\alpha=\beta$, this would imply that 
  $f=h$ on the interior of ${\cal C}_{\alpha}={\cal C}_{\beta}$. By continuity of both $f$ and $h$  at the boundaries of their domain, this would imply that $f=h$ on  ${\cal C}_{\alpha}={\cal C}_{\beta}$.

  Hence, it suffices to show that $\alpha=\beta$.  To show that $\alpha=\beta$,  we first notice that because of the convexity of $\cal X$, for small enough $L>0$ we can find an open ball ${\cal B}$ with radius $L$ included in $\cal X$ on which  $x\mapsto f(\alpha^Tx)$ is not constant. We then have
\begin{equation}
\label{identifiabilitybis}
f(\alpha^Tx)=h(\beta^Tx)\mbox{ for all }x\in\cal B.
\end{equation}
Without loss of generality (possibly replacing $f(z)$ by $f(z-\alpha^Tx_{0})$ and $h(z)$ by $h(z-\beta^Tx_{0})$ with $x_{0}$ being the center of the ball), we assume that $\cal B$ is the open ball with center $x_{0}=0$ and radius $L$.

Assume $\beta\not\in\{\alpha,-\alpha\}$ (which implicitly assumes that $d\geq 2$).   We will show that this yields a contradiction.  The vectors $\beta$ and $\alpha$ are linearly independent so it follows from the Cauchy-Swcharz inequality, where the equality case is excluded,  that $\alpha^T\beta< 1$.  With $v = \beta - \alpha$, we then have
$v^T\alpha=\beta^T\alpha-1< 0.$ Hence, for all $a\in[0,L)$ we have 
$$f(a)= f(\alpha^T (a\alpha))=h(\beta^T (a\alpha))=h(\alpha^T (a\alpha)+v^T(a\alpha))\leq h(a),$$
using \eqref{identifiabilitybis} combined to the monotonicity of $h$. Likewise, $v^T\beta>0$ and therefore,
$$h(a)= h(\beta^T (a\beta))=f(\alpha^T (a\beta))=f(\beta^T (a\beta)-v^T(a\beta))\leq f(a)$$
for all $a\in[0,L)$, whence  $h(a)= f(a)$ for all $a\in[0,L)$. Similarly, $f(-a)=h(-a)$ for all $a\in[0,L)$, whence $f=h$ on $(-L,L)$. Combining this with \eqref{identifiabilitybis} we arrive at
\begin{equation}\label{eq: falphabeta}
f(\alpha^T x)=f(\beta^Tx)\mbox{ for all }x\in\cal B.
\end{equation}
Since $x\mapsto f(\alpha^Tx)$ is not constant on ${\cal B}$,  there exists  a point $b\in(-L,L)$ of strict increase of $f$. The ball $\cal B$ can be chosen in such a way that $b\neq 0$. Either we have $f(b+ \eps) > f(b)$  for all 
$\varepsilon\in(0,L-b)$,  or   we have $f(b-\varepsilon)<f(b)$ for all 
$\varepsilon\in(0,L+b)$. We assume that  $f(b+ \eps) > f(b)$ for all 
$\varepsilon\in(0,L-b)$. The other case can be handled likewise. In the case where  $b> 0$, with $x=(b+\varepsilon)\beta$, we have $x\in\cal B$ and $\alpha^T x\leq b$ for sufficiently small $\varepsilon$, using that $\alpha^T\beta<1$.   Hence, $f(\alpha^Tx)\leq f(b)<f(b+\varepsilon)= f(\beta^Tx)$ by monotonicity of $f$, which  yields a contradiction with  \eqref{eq: falphabeta}. In the case $b<0$, consider $\varepsilon$ sufficiently small so that $b+\varepsilon<0$. Then, with $x=b\alpha$ we  have $x\in\cal B$ and $\beta^Tx\geq b+\varepsilon$  for sufficiently small  $\varepsilon$. Hence,  $f(\alpha^Tx)= f(b)<f(b+\varepsilon)\leq f(\beta^Tx)$ which again, yields a contradiction. This means that $\beta\in\{\alpha,-\alpha\}$.

Assume $\beta=-\alpha$. We will show that this yields again a contradiction.   For all $a\in[0,L)$ we have 
$$f(a)=f(\alpha^T(a\alpha))=h(\beta^T(a\alpha))=h(-a),$$
using \eqref{identifiabilitybis} with $\beta=-\alpha$. This means that $f(a)= h(-a)$. Likewise,  $h(a)= f(-a)$. By monotonicity of $h$ we then have
$$f(a)= h(-a)\leq h(a)= f(-a).$$
As $f$ is non-decreasing, this means that $f(a)=f(-a)$ for all $a\in[0,L)$. Hence, $f$ is constant on $(-L,L)$, which yields a contradiction. This means that $\beta\neq-\alpha$.  We have proved that $\beta\in\{\alpha,-\alpha\}$, hence  $ \alpha=\beta$. This completes the proof of Proposition \ref{prop: identifiability}.
\hfill $\Box$

\subsection{Proof of Theorem \ref{theo: consistency}} \label{sec: proofCons}

Since convergence in probability is equivalent to the property that each subsequence has a further subsequence along which the convergence holds with probability one,  Theorem \ref{theo: ratelog} allows us to assume in what follows without loss of generality (possibly arguing along subsequences) that
\begin{eqnarray}\label{eq: cvL2}
\lim_{n\to\infty}\int\left(\hat \Psi_{n}(\hat\alpha_{n}^Tx)-\Psi_{0}(\alpha_0^Tx)\right)^2d\mathbb{Q}(x)=0
\end{eqnarray}
with probability one. We will show that \eqref{eq: cvL2}  implies {that
$\hat\alpha_{n}$ converges to $\alpha_{0}$}.

 In order to use compactness arguments, we consider a truncated version of $\hat \Psi_{n}$, where we recall that $\hat\Psi_n$ denotes the LSE extended monotonically to the whole real line: 
we consider $\bar \Psi_{n}$ such that
\begin{eqnarray*}
\bar \Psi_{n}(t)=\begin{cases}\hat \Psi_{n}(t)&\mbox{if }\hat \Psi_{n}(t)\in (K_{-},K_{+})\\
K_{+}&\mbox{if }\hat \Psi_{n}(t)\geq K_{+}\\
K_{-}&\mbox{if }\hat \Psi_{n}(t)\leq K_{-}
\end{cases}.
\end{eqnarray*}
where   $K_+$ denotes the largest value of $\Psi_0$ whereas  $K_-$ denotes the smallest value of $\Psi_0$.

We argue along paths, which means that $\omega$ is considered as fixed here and such that \eqref{eq: cvL2} holds.
For all $n$, the function $\bar \Psi_{n}$ is monotone, left-continuous (because $\hat \Psi_{n}$ itself is left-continuous) and bounded by $\max\{K_{+},-K_{-}\}$ in supremum norm. Hence, similar to Lemma 2.5 in \cite{aad}, each subsequence of $(\bar \Psi_{n})_{n\geq 0}$ possesses a further subsequence along which $\bar \Psi_{n}$ converges pointwise to a left-continuous monotone function $m_0$ say, at each of its points of continuity. The function $m_{0}$ possesses right-limits at every point. Because $\hat\alpha_{n}$ belongs to the compact set $\mathcal S_{d-1}$ for all $n$, we can extract a further subsequence along which $\hat\alpha_{n}$ converges in the $\RR^d$-Euclidean distance to some vector $a_{0}\in \mathcal S_{d-1}.$ For simplicity, we denote the general term of the subsequence by $(\bar \Psi_{n}, \hat\alpha_{n})$.   
We aim to show that $m_{0}=\Psi_{0}$  and that $a_{0}=\alpha_{0}$. For this task, consider the $L_{2}$-distance
\begin{eqnarray}\label{eq: dist m0 f0}
\int\left(m_{0}(a_{0}^Tx)-\Psi_{0}(\alpha_{0}^Tx)\right)^2d\mathbb{Q}(x)\leq 3I_{n,1}+3I_{n,2}+3I_{n,3}
\end{eqnarray}
where 
\begin{eqnarray*}
I_{n,1}=\int\left(m_{0}(a_{0}^Tx)-m_{0}(\hat\alpha_{n}^Tx)\right)^2d\QQ(x),&&
I_{n,2}=\int\left(\bar \Psi_{n}(\hat\alpha_{n}^Tx)-\Psi_{0}(\alpha_{0}^Tx)\right)^2d\QQ(x)\\
I_{n,3}&=&\int\left(m_{0}(\hat\alpha_{n}^Tx)-\bar \Psi_{n}(\hat\alpha_{n}^Tx)\right)^2d\QQ(x).
\end{eqnarray*}
We will show that $I_{n,j}$ tends to zero as $n\to\infty$ for $j=1,2,3$ to conclude that the $L_{2}$-distance on the left-hand side of \eqref{eq: dist m0 f0} equals zero. To deal with $I_{n,1}$, we use the fact that because $m_{0}$ is monotone, the set of its  discontinuity points is countable and hence has $\mathbb{Q}$-measure zero. This means that $I_{n,1}$ can be viewed as an integral over the set of continuity points of $m_{0}$. At the continuity points of  $m_{0}$ we have $m_{0}(\hat\alpha_{n}^Tx)\to m_{0}(a_{0}^Tx)$ as $n\to\infty$
so it follows from the dominated convergence theorem that 
$I_{n,1}$ converges to zero as $n\to\infty$. Next, it follows from the definition of $\bar \Psi_{n}$ that
$$I_{n,2}\leq \int\left(\hat \Psi_{n}(\hat\alpha_{n}^Tx)-\Psi_{0}(\alpha_{0}^Tx)\right)^2d\mathbb{Q}(x).$$
Hence, with \eqref{eq: cvL2} we conclude that $I_{n,2}\to 0$.
Finally, with  {$Q_{n}$} the distribution of $\hat\alpha_{n}^TX$, where $X$ is independent of the data points $(X_1,Y_1),\dots,(X_n,Y_n)$, we have
$$I_{n,3}=\int\left(m_{0}(t)-\bar \Psi_{n}(t)\right)^2{dQ_{n}(t).}
$$
{Because} $m_{0}$ is monotone, the set of its discontinuity points is countable and hence has Lebesgue measure zero{, and hence $Q_n$-measure zero}. Because of the convergence of $\bar \Psi_{n}$ at each continuity point of $m_{0}$ we conclude from the dominated convergence theorem that 
$I_{n,3}$ converges to zero as $n\to\infty$. This means that the three terms on the right-hand side of \eqref{eq: dist m0 f0} tends to zero as $n\to\infty$ and therefore,
$$\int\left(m_{0}(a_{0}^Tx)-\Psi_{0}(\alpha_{0}^Tx)\right)^2d\mathbb{Q}(x)=0.$$
Possibly modifying $m_{0}$ so that its restriction to ${\cal C}_{\alpha_0}$ has no discontinuity point at the boundaries of ${\cal C}_{\alpha_{0}}$, which does not modify the value of the above integral, we conclude from Proposition \ref{prop: identifiability} that $a_{0}=\alpha_{0}$ and $m_{0}=\Psi_{0}$ with possible exception at the boundaries of ${\cal C}_{\alpha_0}$. From the calculations above, it follows that 
$\hat\alpha_{n}$ converges to $\alpha_{0}$ and $\bar \Psi_{n}$ converges pointwise to $\Psi_{0}$ {at each continuity point of $\Psi_0$ on the interior of ${\cal C}_{\alpha_0}$, with probability one. The first claim of the theorem follows.}

{Next, let} $I$ be such that $K_-<\Psi_{0}(t)<K_+$ for all $t\in I$. For almost all $\omega$, we then have $\bar \Psi_{n}=\hat \Psi_{n}$ on $I$ for all large enough $n$ 
and therefore,
$$\lim_{n\to\infty}\sup_{t\in I}\vert\hat \Psi_{n}(t)-\Psi_{0}(t)\vert=0$$ 
{with probability one.}
Uniformity of the convergence follows from continuity of $\Psi_{0}$ together with monotonicity of the functions involved. Theorem \ref{theo: consistency}  follows. \hfill{$\Box$}

\subsection{Proof of Corollary~\ref{cor:rate_alpha}} \label{proof: rate_alpha}

\begin{proof}
Borrowing, from \cite[Lemma 5.7, page 418]{murphy1999current} we have that for any random variable $X$, if $(E[g_1(X) g_2(X)])^2 \leq c E[g_1^2(X)]E[g_2^2(X)]$ for $c\leq 1,$ then 
\begin{eqnarray}\label{line:murphy}
E[(g_1(X)+g_2(X))^2] &\geq& (1-\sqrt{c})(E[g_1(X)^2]+E[g_2(X)^2]).
\end{eqnarray}
Note that in the expectations, {with $\PP_X$ denoting the distribution of $X$,} $
E[g_1(X)]={\int g_1(x) d\PP_X(x)}$ and so on - that is, the expectations should be viewed as short-hand notation for the integral in the $x$-variable.   

Let $\cal A$ denote a subset of $\cal X$ to be chosen later such that $\QQ(\mc A)>0$. In order to adapt the above result to our problem, we write 
\begin{eqnarray*}
\int_{{\cal A}}( \hat g_n(x) - g_0( x))^2d\mathbb Q(x)
&=& \int_{\cal A} \left(g_1(x)+g_2(x)\right)^2 d\QQ(x) \\
& = & E[(g_1(X_{\cal A})+g_2(X_{\cal A}))^2]\times\QQ({\cal A})
\end{eqnarray*}
where  $g_1(x)= \hat \Psi_n(\hat\alpha_n^T x) - \Psi_0(\widehat\alpha_n^T x)= \widetilde g_1(\widehat \alpha_n^T x)$, $g_2(x)=\Psi_0(\widehat\alpha_n^T x) - g_0( x)$ and $X_{\cal A}$ denotes here a random variable with density function  $x\mapsto q(x)\mathbb I_{\cal A}(x)/\QQ({\cal A})$. To alleviate the notation, in what follows we simply write $X$ instead of $X_{\cal A}.$ 
We then have
\begin{eqnarray*}
E[g_1(X)g_2(X)]^2 &=& E[\widetilde g_1(\widehat \alpha_n^T X)g_2(X)]^2\\
&=& E[\widetilde g_1(\widehat \alpha_n^T X) E[g_2(X)|\widehat \alpha_n^T X]]^2\\
&\leq& E[\widetilde g_1^2(\widehat \alpha_n^T X)]E\left[E[g_2(X)|\widehat \alpha_n^T X]^2\right], 
\end{eqnarray*}
by the Cauchy-Schwarz inequality. Hence,
\begin{eqnarray*}
E[g_1(X)g_2(X)]^2
& \le & c_n E[g_1^2(X)]E[g_2^2(X)].
\end{eqnarray*}
where 
\begin{eqnarray*}
c_n &=& \frac{E\left[\left(\Psi_0(\widehat \alpha_n^T X) -E[\Psi_0(\alpha_0^T X)|\widehat \alpha_n^T X]\right)^2\right]}{E\left[\left(\Psi_0(\widehat \alpha_n^T X) -\Psi_0(\alpha_0^T X)\right)^2\right]}.
\end{eqnarray*}
Using \eqref{line:murphy} with $c=c_n$, we conclude that 
\begin{eqnarray}\label{eq: conval}\notag
\int_{{\cal A}}(\hat g_n( x) - g_0( x))^2d\mathbb Q(x)&\geq & (1-\sqrt{c_n}) \left\{E\left[(\widehat \Psi_n (\widehat \alpha_n^T X) - \Psi_0(\widehat \alpha_n^T X))^2\right]\right.\\
&&\hspace{.3cm}\left.+E\left[(\Psi_0 (\widehat \alpha_n^T X) - \Psi_0(\alpha_0^T X))^2\right]\right\}\times \QQ({\cal A}).
\end{eqnarray} 
Note that $c_n$ depends on  $\widehat \alpha_n$ and is therefore random in the data.  We will prove below that there exists some real number $c\in(0,1)$ such that from any subsequence, we can extract a further subsequence along which
$
\limsup_{n\to\infty} c_n \leq c <1
$
with probability one. This means that 
\begin{eqnarray}\label{eq: cnOp}
(1-\sqrt{c_n})^{-1}=O_p(1),
\end{eqnarray}
since convergence in probability is equivalent to the property that each subsequence has a further subsequence along which the convergence holds with probability one.

Consider an arbitrary subsequence. Define $r_n =  \|\hat\alpha_n-\alpha_0\|$ and $\gamma_n =  (\widehat \alpha_n-\alpha_0)/r_n.$ Because $\gamma_n$ belongs to the compact set ${\cal S}_{d-1}$, we can extract a subsequence along which $\gamma_n$ converges in probability to a limit, $\gamma$ say. By Theorem~\ref{theo: consistency},  we can extract a further subsequence (that we still index by $n$ to alleviate the notation) along which $\hat\alpha_n,\gamma_n$ converge to $\alpha_0,\gamma$ with probability one.  This means that  we can extract a subsequence along which both $\hat\alpha_n(\omega)\to\alpha_0$ and $\gamma_n(\omega)\to\gamma(\omega)$ for almost all paths $\omega$. We next argue with such a path $\omega$ fixed. Hence, $\hat\alpha_n$ and $\gamma_n$ can be considered as non random and we search for bounds that do not depend on the chosen path $\omega$. We denote by $x_0$ a point in $\cal X$ such that $z_0=\alpha_0^Tx_0$ where $z_0$ is taken from Assumption (A5) and we consider $\eps>0$ such that $\Psi_0$ is continuously differentiable over ${\cal V}:=[z_0-2\eps,z_0+2\eps]$ with a derivative that is bounded both from above and away from zero on $\cal V$. Note that the derivative is uniformly continuous on the compact set $\cal V$. Furthermore, we denote by $\cal A$ the Euclidean ball with center $x_0$ and radius $\eps$. Note that for large enough $n$, we then have $\alpha_0^Tx\in\cal V$ and $\hat\alpha_n^Tx\in\cal V$ for all $x\in\cal A$, by the Cauchy-Schwarz inequality.

We have that 
\begin{eqnarray}\label{eq: taylor}
\Psi_0(\alpha^T_0 x)  = \Psi_0(\widehat{\alpha}^T_n x)  + \Psi'_0(\widehat{\alpha}^T_n x) (\alpha_0-\widehat{\alpha}_n)^T x  +  {o(r_n)}
\end{eqnarray}
uniformly for all $x\in\cal A$ since $|(\widehat{\alpha}_n - \alpha_0)^T x|\leq r_n \|x\|$ where  $\|x\|\leq \|x_0\|+\eps$. 
Hence,
\begin{eqnarray*}
&&\hspace*{-2cm} E\left[\left(\Psi_0(\widehat \alpha_n^T X) -E[\Psi_0(\alpha_0^T X)|\widehat \alpha_n^T X]\right)^2\right]\\
&=& E\left[\left(E\left.\left[\left\{\Psi'_0(\widehat{\alpha}_n^T X) (\widehat \alpha_n-\alpha_0)^T X + {o(r_n)} \right\} \right|\widehat \alpha_n^T X\right]\right)^2\right]\\
&=& E\left[\left(E\left.\left[\Psi'_0(\widehat{\alpha}_n^T X) (\widehat \alpha_n-\alpha_0)^T X\right|\widehat \alpha_n^T X\right]\right)^2\right] + {o(r_n^2)}+E_{n,1}
\end{eqnarray*}
where we recall that $X$ denotes here $X_{\cal A}$, a random variable supported on $\cal A$, and
\begin{eqnarray*}
\vert E_{n,1}\vert   &=& 2o(r_n) \vert E\left[\Psi_0'(\widehat{\alpha}_n^T X)(\widehat \alpha_n-\alpha_0)^T X)\right]  \vert =o(r_n^2),
\end{eqnarray*}
again by (A5).  Similarly, we have $$E\left[\left(\Psi_0(\widehat \alpha_n^T X) -\Psi_0(\alpha_0^T X)\right)^2\right]=E\left[\left(\Psi_0'(\widehat \alpha_n^T X) (\widehat \alpha_n-\alpha_0)^TX\right)^2\right] + o(r_n^2).$$   
Combining these calculations, we arrive at
\begin{eqnarray*}
c_n 
&=&  \frac{E\left[\left(\Psi_0'(\widehat \alpha_n^T X) \gamma_n^TE\left[X|\widehat \alpha_n^T X\right]\right)^2\right]
+o(1)}
{E\left[\left(\Psi_0'(\widehat \alpha_n^T X) \gamma_n^TX\right)^2\right] 
+o(1)}.
\end{eqnarray*}
Lemma~\ref{ConvCondExp} {below} shows that $E[X|\widehat \alpha_n^TX] {\longrightarrow} E[X|\alpha_0^TX]$  almost surely.  This, 
along with  continuity of $\Psi'_0$ and the Lebesgue dominated convergence theorem (since $|X| \leq \|x_0\|+\eps$ almost surely), implies that $c_n$ converges with
\begin{eqnarray*}
 \lim_{n \to \infty}  c_n
&= & \frac{E\Big[(\Psi'_0(\alpha^T_0 X) \gamma^T E[X|\alpha^T_0  X])^2 \Big]}{E\Big[\big(\Psi'_0(\alpha^T_0 X) \gamma^T X \big)^2\Big]} \\
&=& \frac{\gamma^T E\left[(\Psi_0'(\alpha_0^T X))^2 E[X|\alpha_0^T X]E[X|\alpha_0^T X]^T\right] \gamma}{\gamma^T E\left[(\Psi_0'(\alpha_0^T X))^2 XX^T\right] \gamma}.
\end{eqnarray*} 

Now, $\|\widehat \alpha_n\|^2 = \|\alpha_0+r_n \gamma_n\|^2$ and therefore, $1 = ||\alpha_0||^2 + r_n^2 + 2r_n\langle \alpha_0, \gamma_n\rangle,$ which implies that 
$
\langle \alpha_0, \gamma_n\rangle = -r_n/2 \to 0
$
since $\|\alpha_0\|^2=1$.  Since $\gamma=\lim_{n\to\infty}\gamma_n$, this implies  that $\langle \alpha_0, \gamma\rangle=0$ and therefore, $\lim_{n \to \infty} c_{n} \leq  c$ where 
\begin{eqnarray*}
 c=\sup_{\|\gamma\|=1, \langle\gamma, \alpha_0\rangle=0}  \frac{\gamma^T E\left[(\Psi_0'(\alpha_0^T X))^2 E[X|\alpha_0^T X]E[X|\alpha_0^T X]^T\right] \gamma}{\gamma^T E\left[(\Psi_0'(\alpha_0^T X))^2 XX^T\right] \gamma}
\end{eqnarray*}
 does not depend on the chosen path $\omega$. It remains to prove that $c<1$. Now, note that 
\begin{eqnarray*}
E\left[(\Psi_0'(\alpha_0^T X))^2 XX^T\right] &=& E\left[(\Psi_0'(\alpha_0^T X))^2 E[X|\alpha_0^T X]E[X|\alpha_0^T X]^T\right]\\
&& \hspace{0.5cm}+E\left[(\Psi_0'(\alpha_0^T X))^2 (X-E[X|\alpha_0^T X])(X-E[X|\alpha_0^T X])^T\right].
\end{eqnarray*}
Since  $\Psi_0'(\alpha_0^TX)$ is bounded below, our goal is now  to  show that  the null-space of the matrix $E\left[( X-E[X|\alpha_0^T X])( X-E[X|\alpha_0^T X])^T\right]$ is spanned by~$\alpha_0$ only, as this will imply that $c <1.$

To this end, consider any $\gamma_0$ perpendicular to $\alpha_0$ with $\|\gamma_0\|=1.$ 
Let $A_0$ denote the matrix with first row $\alpha_0^T$ and second row $\gamma_0^T$ and  let $Z=A_0 X.$ Since $X$ has an everywhere-positive density, so does $Z.$ To see this, take $Z'=A'_0X$ where $A_0'$ is a $d\times d$ invertible matrix such that $A_0'$ has its first two rows equal to $A_0$.  (Such a matrix exists since $\gamma_0$ and $\alpha_0$ are perpendicular and hence linearly independent.) Then $Z'$ has a density with respect to the Lebesgue measure which can be explicitly calculated using the Jacobian formula and its marginal gives the density of $Z$, $f_Z$ say.  
Now, 
\begin{eqnarray*}
\gamma_0^T E\left[(X-E[X|\alpha_0^T X])(X-E[X|\alpha_0^T X])^T\right] \gamma_0 &=& E\left[(\gamma_0^T X-E[\gamma_0^TX|\alpha_0^T X])^2\right].
\end{eqnarray*}
This equals zero iff $\gamma_0^T X = E[\gamma_0^TX|\alpha_0^T X]$   almost surely, or, $Z_2=E[Z_2|Z_1]$ almost surely.  However, this means that the distribution of $Z$ is concentrated on a one-dimensional subspace of $Z,$ which cannot hold since $Z$ has an {everywhere-positive} Lebesgue density. 
This finally shows that $\lim_{n\to\infty} c_n\leq c <1$.

Then, if follows from \eqref{eq: conval} 
that we can find $c'>0$ and $c''>0$ independent on $n$ and $\omega$ such that for large $n$,
\begin{eqnarray*}
\int_{{\cal A}}(\hat g_n( x) - g_0( x))^2d\mathbb Q(x)&\geq & c' \int_{{\cal A}}(\Psi_0 (\widehat \alpha_n^T x) - \Psi_0(\alpha_0^T x))^2d\mathbb Q(x)\\
&\ge& c'' \int_{{\cal A}}((\widehat \alpha_n-\alpha_0)^T x)^2d\mathbb Q(x)\\
&\ge& c'' \|\hat\alpha_n-\alpha_0\|^2\inf_{\beta\in{\cal S}_{d-1}}\int_{\cal A}(\beta^T x)^2d\mathbb Q(x),
\end{eqnarray*}
by definition of $\cal A$. The infimum above does not depend on the chosen path $\omega$ and is achieved at a point $\beta_0$, say, by continuity of  the function
$\beta\mapsto \int_{\cal A}(\beta^T x)^2d\mathbb Q(x)$
on the compact set ${\cal S}_{d-1}$. Hence, 
\begin{eqnarray*}
\int_{{\cal A}}(\hat g_n( x) - g_0( x))^2d\mathbb Q(x)&\ge& c'' \|\hat\alpha_n-\alpha_0\|^2\int_{\cal A}(\beta_0^T x)^2d\mathbb Q(x).
\end{eqnarray*}
The integral on the right-hand side is strictly positive since $\QQ$ has a density function that is everywhere-positive positive on $\cal A$. This means that there exists $K>0$ such that from each subsequence, we can extract a further subsequence along which
\begin{eqnarray*}
\|\hat\alpha_n-\alpha_0\|^2&\leq& K \int_{{\cal A}}(\hat g_n( x) - g_0( x))^2d\mathbb Q(x)\\
&\leq & K \int_{{\cal X}}(\hat g_n( x) - g_0( x))^2d\mathbb Q(x)
\end{eqnarray*}
for large $n$, with probability one. Since the right-hand side is of order $O_p(n^{-2/3})$ by Theorem \ref{theo: rateL2}, this implies that $\|\hat\alpha_n-\alpha_0\|= O_p(n^{-1/3}).$

In the following, we assume that $\Psi_0$ has a derivative bounded from above on ${\cal C}_{\alpha_0}$ and we extend $\Psi_0$ to the whole real line in such a way that the extension has a bounded derivative on $\RR$.
{With $\underline q$ a lower bound for the density of $\hat\alpha_n^TX$ (note that by assumption, one can consider $\underline q$ that is non random and does not depend on $n$) we have
\begin{eqnarray}\notag
\int_{\mathcal{X}} (\widehat \Psi_n (\widehat \alpha_n^T x) - \Psi_0(\widehat \alpha_n^T x))^2d\mathbb Q(x) &\ge   &  \underline{q} 
\int_{\mathcal{C}_{\widehat{\alpha}_n}} (\widehat \Psi_n (t) - \Psi_0(t))^2 dt\\
&\geq & \underline q   {\int_{\underline c+v_n}^{\overline c-v_n}}(\widehat \Psi_n (t) - \Psi_0(t))^2 dt   \label{line:result2}
\end{eqnarray}}
with probability that tends to one, using the definition of $v_n$ and $\hat\alpha_n-\alpha_0=O_p(n^{-1/3})$. 
On the other hand, 
\begin{eqnarray*}
\int_{\mathcal{X}} (\widehat \Psi_n (\widehat \alpha_n^T x) - \Psi_0(\widehat \alpha_n^T x))^2d\mathbb Q(x) 
&\leq & 2 \int_{\mathcal{X}} (\widehat \Psi_n (\widehat \alpha_n^T x) - \Psi_0(\alpha_0^T x))^2d\mathbb Q(x) \\
&&\quad +2\int_{\mathcal{X}} ( \Psi_0 ( \alpha_0^T x) - \Psi_0(\widehat \alpha_n^T x))^2d\mathbb Q(x)  
\end{eqnarray*}
where by Theorem \ref{theo: rateL2}, the first integral on the right hand side is of order $O_p(n^{-2/3})$. For the second integral, denoting by $K$ an upper bound for the sup-norm of $\Psi_0'$ on $\RR$ we have
\begin{eqnarray*}
\int_{\mathcal{X}} ( \Psi_0 ( \alpha_0^T x) - \Psi_0(\widehat \alpha_n^T x))^2d\mathbb Q(x) &\leq& 
K \int_{\mathcal{X}} ((  \alpha_0-\widehat \alpha_n)^T x)^2d\mathbb Q(x)   \\
&\le& K\|\alpha_0-\alpha_n\|^2 R^2 =O_p(n^{-2/3})
\end{eqnarray*}
where $R$ is taken from \eqref{eq: R}. Combining, we conclude that \eqref{eq: cvpsi} holds.
This proves the second result.
\end{proof}

\subsection{Proofs for Subsection \ref{sec: cvSep}}

\begin{lemma}\label{ConvCondExp}
Let $X$ be a random variable having a density function $q$ with respect to Lebesgue measure on a bounded subset of $\RR^d$, and let $\alpha_n$ be a non-random sequence in ${\cal S}_{d-1}$ that converges to some $\alpha_0$ as $n\to\infty$. If $q$ is continuous and bounded on its domain, we then have
\begin{eqnarray*}
E[X| \alpha^T_n  X] &{\longrightarrow}&  E[X | \alpha^T_0 X]
\end{eqnarray*}
with probability one, as $n\to\infty$.
\end{lemma}

\begin{proof}
  In the following, $\alpha_{0,j}$ and ${\alpha}_{n,j}$ denotes the $j$-th component of $\alpha_0$ and  ${\alpha}_n$ respectively. Without loss of generality, we can assume that $\alpha_{0,1} \ne 0$, and hence ${\alpha}_{n,1} \ne 0$ provided that $n$ is large enough.  Consider the transformation  $Z_1  = {\alpha}^T_n  X$ and  $Z_j =  X_j, \ j \in  \{2, \ldots, d \}$. Let $Z=(Z_1, \ldots, Z_d)$. 
Simple calculations show that  
the density of $Z$ is given by
\begin{eqnarray*}
f_Z(z)  = q\left( \frac{z_1 -{\alpha}_{n,2} z_2 - \ldots  - {\alpha}_{n,d} z_d}{{\alpha}_{n,1}}, z_2, \ldots, z_d \right)  \ \frac{1}{{\alpha}_{n,1}}.
\end{eqnarray*}
This yields the density of the conditional distribution of $X_j $ given that $  {\alpha}^T_n  X  =t$, ${h}_{nj}(x_j |t)$, for $j \in \{2, \ldots, d \}$:
\begin{eqnarray*}
{h}_{nj}(x_j |t) = \frac{\int q\left( \frac{t -{\alpha}_{n,2} \ z_2 - \ldots  - {\alpha}_{n,d} \ z_d}{{\alpha}_{n,1}}, z_2, {\ldots,z_{j-1},x_j,z_{j+1}}\ldots, z_d \right)dz_2 \ldots dz_{j-1}dz_{j+1}\ldots dz_d }{\int q\left( \frac{t -{\alpha}_{n,2} \ z_2 - \ldots  - {\alpha}_{n,d} \ z_d}{{\alpha}_{n,1}}, z_2, \ldots, z_d \right)dz_2 \ldots  dz_d}.
\end{eqnarray*}
where the domains of integrations in the numerator and denominator are  \linebreak $\{(x_2,\dots,x_{j-1}, x_{j+1},\dots,x_d): x\in\mc X\}$ and $\{(x_2,\dots,x_d): x\in\mc X\}$ respectively.
$$
E[X_j | {\alpha}^T_n  X ]  = \int x_j {h}_{nj}(x_j |{\alpha}^T_n  X)  dx_j. 
$$
Using convergence of ${\alpha}_n$ to $\alpha_0$ {and the assumptions on $q$}, 
it follows that 
$
{h}_{nj}(x_j |t)$ converges to  
$h_{0j}(x_j | t)$
for all $t$, where $h_{0j}(\cdot | t)$ is defined in the same manner as $h_{nj}(\cdot|t)$ with $\alpha_n$ replaced by $\alpha_0$. Hence, $h_{0j}(\cdot|t)$ is the conditional density of $X_j $ given that $ \alpha_0^T X =t$.  Using again the Lebesgue dominated convergence theorem, since $X$ is supported on a bounded subset of $\RR^d$,  we then have that 
 $$
E[X_j | {\alpha}^T_n  X ]  \to E[X_j | \alpha^T_0  X ]
$$
almost surely for all $j \in \{2, \ldots, d \}$. For $j=1$, note that 
\begin{eqnarray*}
E[X_1 | {\alpha}^T_n  X ]  &= & \frac{1}{{\alpha}_{n,1}}\left( {\alpha}^T_n  X  -  \sum_{j=2}^d {\alpha}_{n,j} E[X_j | {\alpha}^T_n  X ]\right) \\
& \to & \frac{1}{\alpha_{0,1}} \left(\alpha^T_0  X  -  \sum_{j=2}^d \alpha_{0,j} E[X_j | \alpha^T_0  X ]\right)  = E[X_1 | \alpha^T_0  X ]
\end{eqnarray*}
which completes the proof.  \end{proof}

\subsection{Proof of Theorem \ref{theo: alternativeconsistency}}\label{sec: alternativeconsistency}
By analogy with the proof of Theorem \ref{theo: ratelog}, with $D$ taken from \eqref{eq: defD} we first prove that 
\begin{equation}\label{eq: Dtildelog}
D(\tilde g_n,g_0)=O_p(n^{-1/3}(\log n)^{5/3}).
\end{equation}

For notational convenience, we write  $n$ instead of $n_2$. Moreover, by analogy with the notation in Section \ref{sec: entropy}, we denote by 
\begin{list}{--}
        {
        \setlength{\topsep}{6pt}
        \setlength{\parskip}{0pt}
        \setlength{\partopsep}{0pt}
        \setlength{\parsep}{0pt}
        \setlength{\itemsep}{3pt}
        \setlength{\leftmargin}{20pt}}
\item
${\tilde{\mathcal{G}}_K}$  the class of functions $g(x)  = \Psi(\tilde\alpha_{n}^T x),    \ x \in \mathcal{X}$ where $\Psi \in {\mc M_{K}}$,
\item ${\tilde{\mathcal{G}}_{Kv}}$  the class of functions $g\in  {\tilde{\mathcal{G}}_K}$  satisfying  the condition \eqref{eq: defv}.
\end{list}
We again denote $R = \sup_{x \in \mathcal{X}}  \Vert x \Vert $, so that \eqref{eq: R} holds 
for all $\alpha\in{\cal S}_{d-1}$.

Let $K=C\log n$ for some $C>0$ that does not depend on $n$. Since ${\tilde{\mathcal{G}}_{Kv}}$ is a subset of ${{\mathcal{G}}_{Kv}}$ for arbitrary {$v>0$}, it follows from 
Proposition \ref{prop:centproc} that for all $v \in (0, 2 K]$, and by analogy to  \eqref{eq: Mn}, with $\MM_n$ defined by
$$\mathbb M_ng:=\frac 1{n_2}\sum_{i\in I_2}  \Big \{Y_ig(X_i)-\frac{g(X_i)^2}{2}  \Big\}$$
{there exists $A>0$  that depends only on $R$, $a_0$, $M$ and $\overline q$ such that}
\begin{eqnarray}\label{eq: prop611}
\sqrt nE\bigg[\sup_{g \in{\tilde{\mathcal{G}}_{Kv}}} \Big \vert ({\mathbb{M}_n - \mathbb{M})(g)-(\mathbb{M}_n - \mathbb{M})(g_0) } \Big \vert   \bigg] {\leq A}  \phi_n(v) 
\end{eqnarray}
where 
$
\phi_n(v)   =    v^{1/2}  (\log n)^{5/2} (1 +  (\log n)^{1/2}v^{-3/2}n^{-1/2} ).
$
Since  $D(g,g_0)\leq \|g\|_\infty+\|g_0\|_\infty\le 2K$ for sufficiently large $n$ and all $g\in\tilde{\mc G}_{K}$, we have 
$$\tilde{\mc G}_{{Kv}}=\tilde{\mc G}_{{K}}=\tilde{\mc G}_{{K(2K)}} $$
for all $v>2K$, so that
$$\sqrt nE\bigg[\sup_{g \in \tilde{\mathcal{G}}_{{Kv}}} \Big \vert ({\mathbb{M}_n - \mathbb{M})(g)-(\mathbb{M}_n - \mathbb{M})(g_0) }  \Big \vert   \bigg] {\leq A} \phi_n(2K) {\leq 2A\sqrt C} (\log n)^3$$
 for all $v>2K$ {and sufficiently large $n$}. Hence,  the above inequality \eqref{eq: prop611} holds for all $v>0$.
Furthermore, $\tilde g_n$ maximizes $\MM_n g$  over the set of all functions $g$ of the form $g(x)=\Psi(\tilde\alpha_n^Tx), x \in \mathcal{X}$ with $\Psi$ a non-decreasing function {which} implies, as in Lemma \ref{lem: LSElog}, that
$$
\min_{{k\in I_2}}Y_k \leq  \tilde g_n(x)\leq \max_{{k\in I_2}}Y_k.
$$
Therefore,  with arbitrarily large probability and appropriate choice of $C$, $\tilde g_n$ maximizes $\MM_n g$ over the restricted set $\tilde{\mc G}_{{K}}$. We will prove below that
\begin{equation}\label{eq: tildegTOg}
\MM_n(\tilde g_n)\geq \MM_n(g_0)-O_p(n^{-2/3}).
\end{equation}
Hence, we can use Lemma \ref{lem: basic} above and Theorem 3.2.5 in \cite{vanweak}, with $\alpha=1/2$ and  $r_n\sim n^{1/3}(\log n)^{-5/3}$, to conclude that \eqref{eq: Dtildelog} holds.

We now prove \eqref{eq: tildegTOg}. With  $\tilde g_0(x)=\Psi_0(\tilde\alpha_n^Tx)$ for all $x\in\RR^d$, it follows from the definition of $\tilde g_n$ that $\MM_n(\tilde g_n)\geq \MM_n(\tilde g_0)$. Moreover,
\begin{eqnarray*}
\MM_n g_0-\MM_n \tilde g_0&=&\frac{1}{{2n_2}}\sum_{{i \in I_2}}\left(\Psi_0(\tilde\alpha_n^TX_i)-\Psi_0(\alpha_0^TX_i)\right)^2\\
&&-\frac1{{n_2}}\sum_{{i \in I_2}}(Y_i-\Psi_0(\alpha_0^TX_i))
\left(\Psi_0(\tilde\alpha_n^TX_i)-\Psi_0(\alpha_0^TX_i)\right).
\end{eqnarray*}
Hence, it follows from  the assumption that $\Psi_0$ is $L$-Lipschitz that 
\begin{eqnarray*}
\MM_n g_0-\MM_n \tilde g_0&\leq &\frac{L^2}{{2n_2}}\sum_{{i \in I_2}}\left((\tilde\alpha_n-\alpha_0)^TX_i\right)^2\\
&&-\frac1{{n_2}}\sum_{{i \in I_2}}{(Y_i}-\Psi_0(\alpha_0^TX_i))
\left(\Psi_0(\tilde\alpha_n^TX_i)-\Psi_0(\alpha_0^TX_i)\right).
\end{eqnarray*}
It follows from the Cauchy-Schwarz inequality that the first term on the right hand side is bounded from above by
$$\frac{L^2R^2}{2}\|\tilde\alpha_n-\alpha_0\|^2,
$$
which is of the order $O_p(n^{-2/3})$ by assumption. Conditionally on the first sub-sample, that was used to construct $\tilde\alpha_n$, the second term on the right hand side is a mean of i.i.d. centered variables whose common variance is bounded from above by
$$L^2R^2E\left[(Y-\Psi_0(\alpha_0^TX))^2\right]\|\tilde\alpha_n-\alpha_0\|^2.
$$
Under our assumptions, $E\left[(Y-\Psi_0(\alpha_0^TX))^2\right]$ is finite so we conclude that conditionnaly on the first sub-sample, the second term is of order $O(n^{-1/2} \Vert \tilde{\alpha}_n - \alpha_0 \Vert ) = O(n^{-2/3})$ with arbitrarily large probability, whence 
$\MM_ng_0-\MM_n \tilde g_0\le O_p(n^{-2/3})$. Combining this with the fact that $\MM_n(\tilde g_n)\geq \MM_n(\tilde g_0)$ completes the proof of \eqref{eq: tildegTOg} and hence, the proof of \eqref{eq: Dtildelog}.


It remains to get rid of the log factor. Consider again $K=C\log n$ for some $C>0$ that does not depend on $n$. Since $\tilde{\mathcal{G}}_{{Kv}}$ is a subset of ${\mathcal{G}}_{{Kv}}$ for arbitrary {$v>0$}, it follows from 
Proposition \ref{lem:centproc2} that for all $v \in (0, (\log n)^2n^{-1/3}]$ and for $n$ large enough 
we have
{
\begin{eqnarray}\label{eq: prop61}
\sqrt nE\bigg[\sup_{g \in \tilde{\mathcal{G}}_{Kv}} \Big \vert (\mathbb{M}_n - \mathbb{M})(g)-(\mathbb{M}_n - \mathbb{M})(g_0)  \Big \vert   \bigg] \leq A  \phi_n(v) 
\end{eqnarray}
where $A$ depends only on $R,a_0,M,\overline q,\underline q$ and $K_0$ and}
$
\phi_n(v)   =    v^{1/2} (1 +  v^{-3/2}n^{-1/2} ).
$
It follows from \eqref{eq: Dtildelog} that with a probability that can be made arbitrarily large, the  estimator $\tilde g_n$ belongs to ${\tilde{\mathcal{G}}_{Kv}}$ with $K=C\log n$ and $v=n^{-1/3}(\log n)^2$ for some $C>0$ that does not depend on $n$. The result \eqref{eq: alternativecvg} follows now by applying again Theorem 3.2.5 of \cite{vanweak}  with $\alpha=1/2$ and $r_n\sim n^{1/3}$ combined to \eqref{eq: tildegTOg}.  

The proof of \eqref{eq: cvpsi} only uses that $\hat\alpha_n=\alpha_0+O_p(n^{-1/3})$ so the same  arguments lead to \eqref{eq: alternativecvpsi} since $\tilde\alpha_n=\alpha_0+O_p(n^{-1/3})$ by assumption. \qed

\subsection{Some properties of exponential families}\label{sec: EF}

\begin{prop}\label{prop: EF1}
Let  $Y$ be an integrable random variable having a density with respect to a dominating measure $\lambda$ on $\RR$ of the form
\begin{eqnarray}
y\mapsto h(y, \phi) \exp\left\{\frac{y \ell(\mu) - B( \ell(\mu))}{\phi}\right\} 
\end{eqnarray}
where $\mu$ is the mean, $\phi$ is a dispersion parameter,  $ \ell$ is a real valued function with a strictly positive first derivative on a non void open interval $(a,b)\subset\RR$, and $h$ is a normalizing function. We then have
\begin{enumerate}
\item
 $B$ is infinitely differentiable with $B'= \ell^{-1}$ on $( \ell(a), \ell(b))$, and $ \ell$ is infinitely differentiable. 
\item
If $\ell(\mu)$ belongs to a compact interval wich is stricly included in the domain of $B$, then we can find positive numbers $a_0$ and $M$ that depend only on that compact interval such that $E|Y|^s\leq a_0 s! M^{s-2}$ for all integers $s\geq 1$.
\end{enumerate}

\end{prop}

\begin{proof}
Setting $\eta= \ell(\mu)$, the density takes the form
 $$y\mapsto h(y, \phi) \exp\left\{\frac{y \eta - B(\eta)}{\phi}\right\}. 
$$
Since it integrates to one with respect to the dominating measure $\lambda$, we have
$$\exp\left\{\frac{B(\eta)}{\phi}\right\} =\int h(y, \phi) \exp\left\{\frac{y \eta}{\phi}\right\} d\lambda(y)$$
for all $\eta\in ( \ell(a), \ell(b))$. It follows from standard  properties of the Laplace transform that the right-hand side of the previous display is infinitely differentiable as a function of $\eta$ on $\RR$, so we conclude that $B$ is infinitely differentiable on $( \ell(a), \ell(b))$. Moreover, we can differentiate and interchange derivation and integration to obtain that on $( \ell(a), \ell(b))$,
\begin{eqnarray*}
\frac{\partial}{\partial\eta}\exp\left\{\frac{B(\eta)}{\phi}\right\} &=&\int h(y, \phi) \frac{\partial}{\partial\eta}\exp\left\{\frac{y \eta}{\phi}\right\} d\lambda(y)\\
&=&\int h(y, \phi) \frac{y }{\phi}\exp\left\{\frac{y \eta}{\phi}\right\} d\lambda(y)\\
&=&\frac{E(Y)}{\phi}\exp\left\{\frac{B(\eta)}{\phi}\right\} .
\end{eqnarray*}
Hence, $E(Y)=B'(\eta)$. Going back to the parameter $\mu=E(Y)$, we conclude that
$\mu=B'( \ell(\mu))$
for all  $\mu\in(a,b)$, whence $B'= \ell^{-1}$. This proves the first assertion.

To prove the second assertion, note that with again $\eta=\ell(\mu) $ we have
\begin{eqnarray*}
E\left[\exp\{tY\}\right]&=&\int h(y,\phi)\exp\left\{ty+\frac{y\eta-B(\eta)}{\phi}\right\}d\lambda(y)
\end{eqnarray*}
for all  $t\in\RR$. Now, denote by $[c,d]$ the compact interval to which $\eta$ is assumed to belong. Because this interval is strictly included in the domain of $B$, there exists $\eps>0$ such that $[c-\eps,d+\eps]$ is included in the domain of $B$. With $t=\pm\eps/\phi$, using the fact that the density in the exponential family with natural parameter $\eta+\phi t$ integrates to one, we obtain
\begin{eqnarray}\label{eq: momentExpo}
E\left[\exp\{tY\}\right]&=&\exp\left\{\frac{B(\eta+\phi t)-B(\eta) }{\phi}\right\}.
\end{eqnarray}
 Choosing $t=\eps/\phi$ we conclude that
\begin{eqnarray*}
E\left[\exp\{t|Y|\}\right]&\leq &\exp\left\{\frac{B(\eta+\phi t)-B(\eta) }{\phi}\right\}+\exp\left\{\frac{B(\eta-\phi t)-B(\eta) }{\phi}\right\}
\end{eqnarray*}
where the left-hand side is bounded uniformly in $\eta$ since $B$ is continuous on $[c-\eps,d+\eps]$. In the sequel, we denote by $C$ a positive number that is greater than the right-hand side for all $\eta\in[c,d]$. Since
$$E\left[\exp\{t|Y|\}\right]=\sum_{k\geq 0}\frac{t^kE|Y|^k}{k!}\geq \frac{t^sE|Y|^s}{s!}$$
for all $s\geq 0$, we conclude that $E|Y|^s\leq a_0 s! M^{s-2}$ for all integers $s\geq 1$, where $a_0=C/t^2$ and $M=1/t$. This concludes the proof of Proposition \ref{prop: EF1}.\end{proof}

\bibliographystyle{ims}
\bibliography{MonSingInd}

\end{document}